\documentclass{article}

\usepackage{graphicx} % Required for inserting images
\usepackage{graphicx} % Required for inserting images
\usepackage{graphicx,amssymb,amstext,amsmath,amsthm} 
\numberwithin{equation}{section}
\usepackage{url}
\usepackage[hidelinks,colorlinks=true,linkcolor=blue,citecolor=blue]{hyperref}
\usepackage{mathtools}
\usepackage[a4paper, margin = 3cm]{geometry}
\usepackage{xcolor}
\usepackage{float}
\usepackage{tikz-cd}
\usepackage[T1]{fontenc} %these two allow labels of equations
\usepackage[utf8]{inputenc}
\usepackage{bigints}
\usepackage{etoolbox}
\usepackage[
  backend=biber,
  hyperref=true,
  backref=true,
  isbn=false,
  doi=true,
  url=false,
  natbib=true,
  eprint=true,
  useprefix=true,
  maxcitenames=99,
  maxbibnames=99,  
  maxalphanames=3, 
  minalphanames=3,
  safeinputenc,
  style=alphabetic,
  citestyle=alphabetic,
  block=space,
  datamodel=ext-eprint,
]{biblatex}

\newcommand{\Sp}{{Spin$(7)$}} %%SHORTHAND FOR SPIN(7)
\theoremstyle{plain}
\newtheorem{thm}{Theorem}[section]
\newtheorem{lem}[thm]{Lemma}
\newtheorem{claim}[thm]{Claim}
\newtheorem{cor}[thm]{Corollary}
\newtheorem{prop}[thm]{Proposition}
\newtheorem*{thm*}{Theorem}
\theoremstyle{definition}
\newtheorem{defn}[thm]{Definition}
\newtheorem*{defn*}{Definition}
\newtheorem{remark}[thm]{Remark}
\newtheorem{example}[thm]{Example}

\def\pt{\partial}
\newcommand{\ptt}{\frac{\pt}{\pt t}}

\def\d{\mathrm{d}}
\def\Ric{\mathrm{Ric}}
\def\Riem{\mathrm{Rm}}

\def\Vol{\mathrm{Vol}}

\DeclareMathOperator\Div{div}
\DeclareMathOperator{\inj}{inj}
\title{Shi-type estimates and finite-time singularities of reasonable flows of \texorpdfstring{\Sp}{Spin(7)}-structures}
\author{Joseph Duthie\footnote{Mathematical Institute, University of Oxford, Oxford OX2 6GG, United Kingdom\\
Email: joseph.duthie@maths.ox.ac.uk, ORCID: \href{https://orcid.org/0009-0008-2661-3573}{\color{black}{0009-0008-2661-3573}}}}
\date{\today}

\begin{document}
\maketitle
\begin{abstract}
   This paper establishes foundational analytic and geometric results for a broad class of reasonable flows of \Sp-structures. We first prove Shi-type derivative estimates, showing that a uniform bound on the quantity 
   \[
   \Lambda(x,t) = \left(|\Riem(x,t)|_{g(t)}^2 + |T(x,t)|^4_{g(t)} + |\nabla T(x,t)|_{g(t)}^2 \right )^{1/2} 
   \]
   implies bounds on all covariant derivatives of the curvature $\Riem$ and torsion tensor $T$.
   We show further that $\Lambda(x,t)$ must blow up at a finite-time singularity, and establish a lower bound on the blow-up rate. We also prove a compactness theorem for solutions to such flows and apply these results to the analysis of finite-time singularities.
   These results provide a general analytic framework for studying flows of \Sp-structures; once a proposed flow is shown to satisfy the reasonable condition, our estimates, compactness theorems, and singularity analysis apply.
\end{abstract}

{\hypersetup{linkcolor=black}
\tableofcontents
}%Makes table of contents colour black instead of blue, but remain clickable hyperlinks

\section{Introduction}

Geometric flows have proven to be a powerful tool in geometry and topology. By deforming geometric structures via suitably-chosen evolution equations, one can hope to improve an initial structure, detect canonical representatives, or find obstructions to their existence. The most prominent intrinsic geometric flow is the Ricci flow of metrics, which has had a profound impact on differential geometry and low-dimensional topology, most notably in Perelman's resolution of the Poincaré conjecture \cite{perelman2002entropyformularicciflow}, \cite{perelman2003ricciflowsurgerythreemanifolds}.
For the study of geometric structures in general, flows provide a tool to study non-linear geometric PDEs dynamically; instead of attempting to solve an elliptic system directly, one evolves an initial geometric structure and studies its long-time behaviour. If the flow exists for all time and converges, it may produce a distinguished canonical limiting structure such as a Ricci-flat metric in the case of Ricci flow. On the other hand, the formation of singularities can also encode important geometric information.

\bigskip
The geometric structures we are interested in here are \Sp-structures on $8$-manifolds. A \Sp-structure on an $8$-manifold $M$ is a choice of $4$-form $\Phi$ satisfying a particular algebraic condition. Such a $4$-form induces a Riemannian metric $g_\Phi$ and an orientation $\text{vol}_\Phi$. If $\Phi$ is parallel with respect to the induced Levi-Civita connection, then the \Sp-structure is said to be torsion-free and the pair $(M,\Phi)$ is called a \Sp-manifold. Torsion-free \Sp-structures induce metrics with holonomy contained in the Lie group Spin($7$), one of the two exceptional holonomy groups in Berger's classification \cite{Berger}. \Sp-manifolds occupy an important space in the subject of differential geometry. On the level of Riemannian metrics, \Sp-manifolds currently provide the only non-trivial examples of non-complex, compact Ricci-flat $8$-manifolds. From the point of view of calibrated geometry, \Sp-manifolds carry Cayley calibrations, leading to a rich theory of the study of Cayley submanifolds. \Sp-manifolds are also of interest in mathematical physics.

\bigskip
The problem of finding \Sp-manifolds, that is to solve the equation $\nabla \Phi = 0$, is a highly non-linear PDE. In the compact setting, there is essentially only one method of construction, due to Joyce \cite{JoyceSpin7}.
Joyce’s construction produces such manifolds by resolving singular orbifolds and then perturbing \Sp-structures with sufficiently small torsion to torsion-free examples. This method is extremely powerful, but relies on carefully-engineered initial data. In this sense, all of the currently available examples constructed via this method are very similar, and so we have little understanding of the general problem. It is therefore natural to seek complementary methods, with the goal of deforming a more arbitrary \Sp-structure towards a torsion-free one, or obtaining some understanding of possible obstructions to doing so.

\bigskip
One proposed method is via geometric flows. In contrast to the $G_2$-setting, flows of \Sp-structures have received comparatively little attention. Several interesting flows have been introduced and studied (\cite{DLE24},         \cite{Dwivedi24}, \cite{DwivediRHF}), but there is currently no single flow that is widely acknowledged to be a canonical tool in the study of \Sp-structures, unlike the Laplacian flow of $G_2$-structures. Because of this variety of flows, and the likelihood of more to come, one needs analytic tools which are not tied closely to a particular choice of evolution equation, but which work for a wide range of flows.

The purpose of this paper is to provide such a toolkit, adapting successful methods from the study of $G_2$-flows (\cite{Lotay-Wei},\cite{Chen},\cite{DwivediRHF}) into the \Sp-setting.

\bigskip

\noindent\textbf{Outline of paper.}

After recalling some necessary background on \Sp-structures in Section \ref{SectionPrelims}, we begin in Section \ref{SectionReasonable} by defining a class of reasonable flows of \Sp-structures, analogously to Chen in the $G_2$ case \cite{Chen} (cf. Definition \ref{DefReasonable}).
\begin{defn*}
    Let $\Phi(t)$ be a flow of \Sp-structures, given by the evolution equation
    \begin{equation}
        \ptt \Phi(t) = (A \diamond\Phi),
    \end{equation}
    for some $A = h+X$, as discussed in Section \ref{Flowsintro}.
   We say that the flow $\Phi(t)$ is \emph{reasonable} if it has short-time existence and uniqueness, and the following schematic equations hold for the induced evolution equations of the induced Riemannian metric $g(t)$ and the torsion tensor $T(t)$, and for the skew $2$-tensor $X$:
    \begin{align}
        \frac{\pt }{\pt t}g_{ij} &= 2h_{ij}=  -2 \Ric_{ij}  + L(T) + T*T +C ,\\
        X&= L(\nabla T)  + L(\Riem) + L( T) + T*T +C, \\
         \frac{\pt }{\pt t}T_{m;is}&= \Delta T_{m;is} + L(\nabla T) + L(T) + \Riem * T \nonumber\\
         &\quad+ \nabla T * T + T*T*T + T*T ,
    \end{align}
    where $L$ denotes linear maps, $*$ denotes any multilinear contraction using $g$ and $\Phi$, and $C$ denotes tensor fields independent of $g$ and $T$, but that we require to be bounded.
\end{defn*}
This definition is discussed in more detail in Section \ref{SectionReasonable}, and we show there that Dwivedi's Ricci-Harmonic flow is reasonable. We note that this condition is easy to check for a given flow of \Sp-structures.

\bigskip 
\noindent\textbf{Shi-type estimates.}

The choice of metric and torsion evolution in the definition of reasonable flows was chosen precisely to capture the analytic features needed for the following derivative estimates for the quantity
\[
\Lambda(x,t) = \left(|\Riem(x,t)|_{g(t)}^2 + |T(x,t)|^4_{g(t)} + |\nabla T(x,t)|_{g(t)}^2 \right )^{1/2},
\] which hold for any reasonable flow. The motivation behind defining $\Lambda$ to be this particular quantity is discussed in Section \ref{SectionShiType}. We now state our Shi-type estimate (cf. Theorem \ref{ThmShiType}).
\begin{thm*}
    Let $K>0$ and $r>0$. Let $M$ be an $8$-manifold, $p \in M$ and $\Phi(t)$ be a solution to a reasonable flow of \Sp-structures on an open neighbourhood $U$ of $p$ containing $B_{g(0)}(p,r)$ as a compact subset.

    Suppose $\Lambda(x,t) \leq K$ for all $(x,t) \in U \times [0,1/K]$. Then, for all $m\in \mathbb{N}$, there exists a constant $C(K,m,r)$ such that
    \[
    |\nabla^m\Riem(x,t)| + |\nabla^{m+1}T(x,t)| \leq C(K,m,r)t^{-m/2}
    \]
    for all $x \in B_{g(0)}(p,r/2)$, $t \in [0,1/K]$.
\end{thm*}
We note that here, and throughout the paper, the notation $\nabla^k$ refers to the $k$-fold covariant derivative, not a raised index $k$.

This shows that $\Lambda$ behaves similarly to the curvature tensor $\Riem$ under Ricci flow.
Using this, in Section \ref{SectionLongTime} we show that, at any finite-time singularity of any reasonable flow of \Sp-structures, the quantity $\Lambda(x,t)$ must blow up, and we establish a  lower bound on the blow-up rate (cf. Theorem \ref{ThmLongTime}).
\begin{thm*}
    Let $\Phi(t)$ be a solution to a reasonable flow of \Sp-structures on a compact manifold $M$ on a maximal time interval $[0,T_0)$ with $T_0<\infty$, and let $\Lambda(t) = \sup_{x \in M}\Lambda(x,t)$, where $\Lambda(x,t)$ is as defined in \eqref{eqLambdaDef}.
    Then, 
    \begin{equation}
       \lim_{t\nearrow T_0} \Lambda(t) = \infty, 
    \end{equation}
    and we have the following lower bound on the blow-up rate:
    \begin{equation}
       \Lambda(t) \geq \frac{C}{T_0-t}, 
    \end{equation}
    for some constant $C>0$.
\end{thm*}
This also gives rise to a classification of singularity types (Definition \ref{defSingularitytypes}), which we expect to be useful in future studies of singularities. 

\bigskip 
\noindent\textbf{Compactness.}
In the analysis of Ricci flow, Hamilton's compactness theorem is essential in the study of the behaviour of the flow close to singularities \cite{HamiltonCompactness}. 
In Section \ref{SectionCompactness}, we prove a compactness result for the space of \Sp-structures (Theorem \ref{ThmSpin7Compact}), and use that to obtain the following compactness theorem for solutions to reasonable flows of \Sp-structures, analogously to  the $G_2$-setting \cite{Lotay-Wei} (cf. Theorem \ref{thmCompactnessSpaceofSolutions}).
\begin{thm*}
    Let $M_i$ be a sequence of compact $8$-manifolds and let $p_i \in M_i$ for each $i$. Suppose that $\Phi_i(t)$ is a sequence of solutions to a given reasonable flow of \Sp-structures on $M_i$, with the induced sequence of Riemannian metrics $g_i(t)$ on $M_i$, for $t\in (a,b)$, where $-\infty \leq a < 0 <b \leq \infty$.

    Suppose further that 
    \begin{equation}
        \sup_i\sup_{x \in M_i, t \in (a,b)}\left(\left|\nabla_{g_i(t)}T_i(x,t) \right |_{g_i(t)}^2 + \left|\Riem_i(x,t) \right |_{g_i(t)}^2 + \left|T_i(x,t) \right |_{g_i(t)}^4 \right)^{\frac{1}{2}}< \infty,
    \end{equation}
    where $T_i$ and $\Riem_i$ denote the torsion and Riemann curvature tensor induced by $\Phi_i$. Finally, suppose that the injectivity radius of the each initial manifold $(M_i,g_i(0))$ at $p_i$ satisfies
    \begin{equation}
        \inf_i \inj(M_i,g_i(0),p_i)>0.
    \end{equation}
    Then, there exists an $8$-manifold $M$, a point $p\in M$ and a solution $\Phi(t)$ to the same reasonable flow of \Sp-structures that $\Phi_i(t) $ solves for $t\in (a,b)$ such that, after passing to a subsequence,
    \begin{equation}
        (M_i,\Phi_i(t),p_i) \to (M,\Phi(t),p) \text{ as } i \to \infty.
    \end{equation}
\end{thm*}
The notion of convergence of \Sp-flows discussed here is defined in Section \ref{SectionCompactness}, and is analogous to that of Ricci flows  \cite{HamiltonCompactness}.

\bigskip 
\noindent\textbf{Finite-time singularities.}
Finally, we apply all of this to the study of finite-time singularities, culminating in the following, which shows that, under suitable assumptions on the growth rate of torsion and scalar curvature, the blow-up limit at a finite time singularity of a reasonable flow of \Sp-structures is a maximal volume growth \Sp-manifold (See Theorem \ref{ThmMaximalVolumeGrowthSpin7Manifold} for a more precise statement of the following ).
\begin{thm*}
 Let $\Phi(t)$ be a solution to a reasonable flow of \Sp-structures on a compact manifold $M^8$, on a maximal time interval $[0,T_0)$, for some $T_0<\infty.$
    Assume that
    \begin{equation}
        \int_0^{T_0}(T_0-t)\sup_{x\in M}|T(x,t)|_{g(t)}^4\d t < \infty,
    \end{equation}
    and
    \begin{equation}
        \sup_{x\in M}\left(|R(x,t)|_{g(t)} + |T(x,t)|_{g(t)}^2 \right) = o\left(\frac{1}{T_0-t} \right) \text{ as } t\to T_0.
    \end{equation}

    Then, a suitably-rescaled flow $Q_k\Phi(t_k)$ converges to a maximal volume growth torsion-free \Sp-manifold.
\end{thm*}
This result shows that, under appropriate torsion and curvature-growth assumptions, finite-time singularities have highly structured blow-up models: that of maximal volume growth torsion-free \Sp-manifolds.

\bigskip
\noindent\textbf{Acknowledgements.}
The author is grateful to his supervisor, Jason D. Lotay, for his constant support and many helpful discussions.
The author also thanks Shubham Dwivedi for interesting conversations about geometric flows, and Sam Close for useful discussions about $4$-forms inducing metrics. 

This work was supported by a doctoral scholarship from the Engineering and Physical Sciences Research Council (Project reference EP/W524311/1 2929148).

\section{\texorpdfstring{Preliminaries on \Sp-structures}{Preliminaries on Spin(7)-structures}}\label{SectionPrelims}
In this section we introduce the necessary background material on manifolds with \Sp-structures, which we will use throughout the rest of the paper. 
Much more detail can be found in, for instance, \cite[Chapter 11]{JoyceBook}.
We begin by defining the Lie group Spin($7$) and the notion of \Sp-structures on $8$-manifolds.

\begin{defn}\label{defSpin7}
    Endow $\mathbb{R}^8$ with an orthonormal basis $\{e_1,\cdots, e_8\}$ and define the $4$-form $\Phi_0$ as
    \begin{equation}\label{eqPhi0}
        \begin{split}
            \Phi_0 =& e^{1234}+ e^{1256} + e^{1278} + e^{1357} - e^{1368} - e^{1458} - e^{1467}\\ 
            &+e^{5678}+ e^{3478} + e^{3456} + e^{2468} - e^{2457} - e^{2367} -e^{2358},
        \end{split}
    \end{equation}
    where $e^i$ denotes the $1$-form dual to $e_i$, and $e^{ijkl} = e^i \wedge e^j \wedge e^k \wedge e^l$. 
    
    We define the Lie group Spin($7$) to be the subgroup of $\mathrm{GL}(8,\mathbb{R})$ that fixes $\Phi_0$.
\end{defn}

A \Sp-structure on an $8$-manifold $M$ is a reduction of the structure group of the frame bundle $Fr(M)$ from GL($8,\mathbb{R})$ to 
\[
\text{Spin}(7) \subset \text{SO}(8) \subset \text{GL}(8,\mathbb{R}).
\]
In the rest of this paper, we will use the following, equivalent, definition

\begin{defn}\label{DefSPin7Structure}
    A \Sp-structure on an $8$-manifold $M$ is a choice of $4$-form $\Phi\in \Omega^4(M)$ such that, for each $p\in M$, there exists an oriented isomorphism between $T_pM$ and $\mathbb{R}^8$ for which $\Phi|_p$ is identified with the $4$-form $\Phi_0$ on $\mathbb{R}^8$ defined above \eqref{eqPhi0}. Such a $4$-form $\Phi$ is called \emph{admissible} and referred to as a \emph{Cayley form} or a \emph{\Sp-form}.
    The space of admissible $4$-forms is denoted $\mathcal{A}M.$
\end{defn}
{The Lie group Spin($7$) is a subgroup of $\operatorname{SO}(8)$, so the existence of a \Sp-structure implies the existence of a reduction of the structure group of the frame bundle $Fr(M)$ to SO($8$), which provides orientability and existence of a Riemannian metric. We denote the Riemannian metric induced by a \Sp-structure $\Phi$ by $g_\Phi$, the associated Hodge star by $*_\Phi$, and the associated volume form by $\text{vol}_{\Phi}$. 
With respect to the metric, Hodge star and volume forms discussed above, it holds that
$*_\Phi \Phi = \Phi$, and  $\Phi \wedge \Phi = 14\operatorname{vol}_\Phi$.}

Given a \Sp-form $\Phi$, the following equation holds \cite[Corollary 4.3.2]{KarigiannisDeformations}
\begin{equation}\label{EqPositivity4form}
    (v \lrcorner w \lrcorner\Phi)\wedge(v \lrcorner w \lrcorner\Phi)\wedge \Phi = 6|v\wedge w|_{g_\Phi}^2\mathrm{vol}_\Phi.
\end{equation}
However, there exist $4$-forms $\Psi$ for which
\[
(v \lrcorner w \lrcorner\Psi)\wedge(v \lrcorner w \lrcorner\Psi)\wedge \Psi = f(v,w)\mathrm{vol}_\Psi,
\]
for some positive function $f$, which are not \Sp-structures.\footnote{The author is grateful to Sam Close for pointing this out.} One way of seeing this is to note that the above is an \emph{open} condition, whereas the admissible $4$-forms live in a $43$-dimensional subspace of the $70$-dimensional space of $4$-forms, and so being admissible is a \emph{closed} condition. We note that this is in contrast to the case of $G_2$-structures on $7$-manifolds, where the space of $G_2$-forms is an open subbundle of the space of $3$-forms, and 
\begin{equation}
    (u\lrcorner\varphi) \wedge (v \lrcorner \varphi)\wedge \varphi
\end{equation}
defines a positive $7$-form if and only if $\varphi$ defines a $G_2$-structure (see e.g \cite[Section 2.1]{Lotay-Wei}). 

In the study of long-time existence of reasonable flows of \Sp-structures and compactness of the space of \Sp-structures (Theorems \ref{ThmLongTime} and \ref{ThmSpin7Compact}) we will need a way of determining whether a given $4$-form defines a \Sp-structure. This question is studied in \cite[Section 7]{SalamonWalpuski}, and we briefly recap some of their results here. We begin with the following definition of non-degenerate $4$-forms and compatible inner products.
\begin{defn}[\cite{SalamonWalpuski}]\label{DefNonDegenerate}
    Let $W$ be an $8$-dimensional vector space. A $4$-form $\Phi \in \Lambda^4W^*$ is called \emph{non-degenerate} if, for any triple $u,v,w$ of linearly independent vectors in $W$, there exists a vector $x\in W$ such that $\Phi(u,v,w,x)\neq 0$.

    An inner product $\langle\cdot,\cdot \rangle$ is called \emph{compatible} with a $4$-form $\Phi$ if the map $W^3 \to W: (u,v,w) \to u\times v \times w$ defined by
    \begin{equation}
        \langle x,u\times v \times w \rangle = \Phi(x,u,v,w)
    \end{equation}
    is a triple cross product (See \cite[Definition 6.1]{SalamonWalpuski}.

    A $4$-form is called a Spin($7$)-form if it admits a compatible inner product.
\end{defn}
This definition rephrases the condition for a given $4$-form to be a Spin($7$)-form, but it is not yet a checkable condition. The following theorem \cite[Theorem 7.8]{SalamonWalpuski} provides such a condition, giving an intrinsic characterisation of Spin($7$)-forms.
\begin{thm}[\cite{SalamonWalpuski}]\label{ThmCharacterisation}
    Let $W$ be an $8$-dimensional vector space.
    A $4$-form $\Phi \in \Lambda^4W^*$ is a Spin($7$)-form if and only if it satisfies both of the following conditions:
    \begin{enumerate}
        \item  The $4$-form $\Phi$ is non-degenerate, in the sense of Definition \ref{DefNonDegenerate}.
        \item If $u,v$ and $w$ are linearly independent vectors in $W$ such that
        \begin{equation}
            (v \lrcorner u \lrcorner \Phi )\wedge (w\lrcorner u \lrcorner\Phi )\wedge \Phi = (u \lrcorner v \lrcorner \Phi )\wedge (w\lrcorner v \lrcorner\Phi )\wedge \Phi = 0,
        \end{equation}
        then for all $x\in W$, we have 
        \begin{equation}
            (w \lrcorner u \lrcorner \Phi )\wedge (x\lrcorner v \lrcorner\Phi )\wedge \Phi = 0 \text{ if and only if }  (w \lrcorner v \lrcorner \Phi )\wedge (x\lrcorner v \lrcorner\Phi )\wedge \Phi = 0.
        \end{equation}
    \end{enumerate}
\end{thm}
Note that an equivalent formulation of the non-degeneracy condition is that, for any linearly independent vectors $u,v$, the $8$-form given by
\[
(u \lrcorner v \lrcorner \Phi) \wedge (u \lrcorner v \lrcorner \Phi) \wedge \Phi
\]
is positive.

The useful thing about this theorem for us is that it provides an algebraic condition to check whether a given $4$-form defines a \Sp-structure. We will use this in the proofs of Theorem \ref{ThmLongTime} and Theorem \ref{ThmSpin7Compact} to ensure that certain limits of \Sp-structures are also \Sp-structures.

If a \Sp-structure is parallel with respect to the Levi-Civita connection of its induced metric, it is said to be torsion-free. We recall this definition and some consequences of existence of torsion-free \Sp-structures here.
\begin{defn}
    Let $(M,\Phi)$ be an $8$-manifold with \Sp-structure. We say $\Phi$ is \emph{torsion-free} if
    \[
    \nabla^{g_\Phi}\Phi = 0,
    \]
    where $\nabla^{g_\Phi}$ is the Levi-Civita connection with respect to the Riemannian metric $g_\Phi$.
\end{defn}
We now recall the following equivalent conditions, the first of which is the primary motivation for finding torsion-free \Sp-structures.
\begin{thm}\label{equivalence}
    Let $(M,\Phi)$ be an $8$-manifold with \Sp-structure. The following are equivalent:
    \begin{enumerate}
        \item $Hol(g_\Phi) \subseteq \text{Spin}(7)$,
        \item $\nabla^{g_\Phi}\Phi = 0$,
        \item $\mathrm{d}\Phi = 0$,
    \end{enumerate}
    where $Hol(g_\Phi)$ is the holonomy group of the metric $g_\Phi$.
\end{thm}
Moreover, if $\Phi$ is a torsion-free \Sp-structure, then the induced metric $g_\Phi$ is Ricci-flat \cite{Bonan}.
\subsection{Decomposition of the space of forms}\label{SubsectionFormDecompositions}

let $(M,\Phi)$ be an $8$-manifold with \Sp-structure. By Definition \ref{defSpin7}, the $4$-form $\Phi|_p$ is stabilised by the Lie group \Sp, for any point $p\in M$.
So, the existence of a \Sp-structure on $M$ defines an action of Spin($7$) on the spaces of differential forms, which then induces a decomposition of the space of differential forms on $M$ into irreducible representations of \Sp. We recall these decompositions here.

\begin{prop}[{\cite[Prop. 11.4.4]{JoyceBook}}]\label{PropDecompositionForms}
    Let $(M,\Phi)$ be an $8$-manifold with \Sp-structure. Then, the spaces $\Lambda^kT^*M$ decompose orthogonally (with respect to the metric on forms induced by $g_\Phi$) as follows, where $\Lambda^k_l$ corresponds to an irreducible \Sp-representation of dimension $l$:
    
    \begin{equation}\label{EquationFormSplitting}
  \begin{split}
    &\Lambda^1T^*M = \Lambda^1_8,\\
    &\Lambda^2T^*M  = \Lambda^2_7 \oplus \Lambda^2_{21},\\
    &\Lambda^3T^*M  = \Lambda^3_8 \oplus \Lambda^3_{48},\\
    &\Lambda^4T^*M  = \Lambda^4_1 \oplus \Lambda^4_7 \oplus \Lambda^4_{27} \oplus \Lambda^4_{35}.
\end{split}  
\end{equation}
Moreover, the Hodge star defines an isomorphism \[ *_\Phi : \Lambda^k_l \to \Lambda^{n-k}_l,
\]
so the above decompositions also give decompositions for $\Lambda^5T^*M ,\cdots,\Lambda^8T^*M $.
\end{prop}
We will write $\Omega^k(M)$ for the space of $k$-forms on $M$ (i.e., $\Omega^k(M) = \Gamma(\Lambda^kT^*M)$), or simply $\Omega^k$. In the presence of a \Sp-structure, each space of forms $\Omega^k$ decomposes as above, e.g.,
\begin{equation}
    \Omega^4 = \Omega^4_1 \oplus \Omega^4_7 \oplus \Omega^4_{27}\oplus \Omega^4_{35}.
\end{equation}
Although the metric $g_\Phi$ and the Hodge star $*_\Phi$ depend on $\Phi$, when there is no risk of ambiguity we will simply write $g$ and $*$.
One can describe these irreducible subspaces more explicitly, in terms of $\Phi$, as in \cite[Section 4.2]{KarigiannisDeformations}. Here, we shall only need the explicit expression for $\Omega^2_7$, which we state here.

\begin{align}\label{forms}
  \Omega^2_7   &= \{\alpha \in \Omega^2 \mid *(\alpha \wedge \Phi) = 3\alpha\}, 
  &\Omega^2_{21} &= \{\alpha \in \Omega^2 \mid *(\alpha \wedge \Phi) = -\alpha\}.
\end{align}

The equations above allow us to write down projection formulae for $\pi^k_l:\Omega^k \to \Omega^k_l$. The ones we will need are listed in the following proposition.

\begin{prop}[{\cite[Prop. 2.1]{KarigiannisFlows}}]
    Writing $\pi^k_l:\Omega^k \to \Omega^k_l$ for the projection maps from the space of $k$-forms to the (pointwise) $l$-dimensional irreducible component described above, we have the following explicit formulae:
    
\begin{align}
  \pi^2_7(\alpha)   &= \frac{\alpha + *(\Phi \wedge \alpha)}{4}, 
  &\pi^2_{21}(\alpha) &= \frac{3\alpha - *(\Phi \wedge \alpha)}{4}.
\end{align}
In local coordinates, the projections of $2$-forms can be expressed as

\begin{equation}
  \begin{split}
    &\pi^2_7(\alpha)_{ij} = \frac{1}{4}\alpha_{ij} - \frac{1}{8}\alpha_{ab}g^{ap}g^{bq}\Phi_{pqij}, \label{eqpi27coords}\\
    &\pi^2_{21}(\alpha)_{ij} = \frac{3}{4}\alpha_{ij} + \frac{1}{8}\alpha_{ab}g^{ap}g^{bq}\Phi_{pqij}.\\ 
\end{split}  
\end{equation}
\end{prop}
When there is little risk of ambiguity, we will simply write $\pi_l$ for the projection.

In order to describe deformations and flows of \Sp-structures, Karigiannis introduces the \emph{diamond} map, which we recall here \cite{KarigiannisFlows}. 
Let $A$ be a $(0,2)$-tensor and consider the following map:
\begin{equation}\label{eqDiamond}
    A \mapsto(A \diamond \Phi)_{ijkl} = \left( A_{ip} g^{pq} \Phi_{qjkl} + A_{jp} g^{pq} \Phi_{iqkl} + A_{kp}g^{pq}  \Phi_{ijql} + A_{lp} g^{pq} \Phi_{ijkq} \right). 
\end{equation}

The diamond map $\diamond$ depends on $\Phi$, but we will simply write $\diamond$ to simplify notation.
Using the metric, $A$ can be viewed as a $(1,1)$-tensor, and under the identification 
\[
\Gamma(TM\otimes T^*M) \cong \operatorname{End}(TM) \cong \mathfrak{gl}(8,\mathbb{R}),\] $\diamond$ describes the infinitesimal action of $\operatorname{GL(8,\mathbb{R})}$ on $\Phi$. Thus, any infinitesimal deformation of a \Sp-structure $\Phi$ can be written as $A\diamond\Phi$ for some $A$. Decomposing $A$ into symmetric and skew-symmetric parts, we write $A = h + X$, for $h\in S^2$ and $X \in \Omega^2$. Karigiannis shows that the kernel of $\diamond$ is $\Omega^2_{21}$, so we can in fact take $X$ to lie in $\Omega^2_7$.
\begin{prop}[{\cite[Proposition 2.3]{KarigiannisFlows}}]\label{PropDiamondKernel}
    The kernel of the map $A \mapsto A \diamond \Phi$ is isomorphic to $\Omega^2_{21}$.
\end{prop}
We also have the following.
\begin{cor}[{\cite[Corollary 2.6]{KarigiannisFlows}}]\label{CorollaryIsomorphism}
    The map $A \mapsto A\diamond \Phi$ is injective on $S^2 \oplus \Omega^2_7$, and is therefore an isomorphism onto its image $\Omega^4_1 \oplus \Omega^4_7 \oplus \Omega^4_{35} $. 
    Decompose the space of symmetric $2$-tensors $S^2$ into multiples of the metric and trace-free parts: $S^2 = \langle g_\Phi \rangle  \oplus S^2_0$.
    The summands $\langle g_\Phi \rangle,S^2_0$ and $\Omega^2_7$ are mapped isomorphically onto $\Omega^4_1,\Omega^4_{35}$ and $\Omega^4_7$, respectively.
\end{cor}
We shall use this description of deformations of \Sp-structures to study \emph{flows} of \Sp-structures in Section \ref{Flowsintro}. First, we introduce the torsion tensor
\subsection{The torsion tensor and torsion forms of a \texorpdfstring{\Sp-structure}{Spin(7)-structure}}
The obstruction to a \Sp-structure being torsion-free is called \emph{torsion}. 
In this section, we recall the definition of the torsion tensor $T$ and some other related tensors. We also recall the expression of the Ricci tensor in terms of the torsion tensor. We shall use all of this when considering flows of \Sp-structures.
We first consider the $(0,5)-$tensor $\nabla_X\Phi$, which will give rise to the torsion tensor $T$.

\begin{lem}[{\cite[Lemma 2.10]{KarigiannisFlows}}]
Let $X$ be a vector field and $\Phi$ be a \Sp-structure on an $8$-manifold $M$.
Then, $\nabla_X\Phi \in \Omega^4_7 \subset \Omega^4$. Thus, the $(0,5)$-tensor $\nabla \Phi$ lies in the space $\Omega^1_8 \otimes \Omega^4_7$.   
\end{lem}

Because of this, Corollary \ref{CorollaryIsomorphism} implies that, for each vector field $e_m$, there exists a $2$-form $T_m\in \Omega^2_7$ (where $m$ is fixed) such that \[\nabla_m\Phi = T_m\diamond\Phi.\] This motivates the following definition.

\begin{defn}[{\cite[Definition 2.12]{KarigiannisFlows}}] \label{DefTorsion}
    The \emph{torsion tensor} $T$ of a \Sp-structure $\Phi$ is the element of $\Omega^1_8 \otimes \Omega^2_7$ such that
    \begin{equation}\label{EquationDefTorsion}
        \nabla_m\Phi_{ijkl} = (T_m \diamond \Phi)_{ijkl} = T_{m;ip} g^{pq} \Phi_{qjkl} 
+ T_{m;jp} g^{pq} \Phi_{i q k l} 
+ T_{m;kp} g^{pq} \Phi_{ij q l} 
+ T_{m;lp} g^{pq} \Phi_{ijk q}.
    \end{equation}
\end{defn}
Note that the semi-colon in $T_{m;ab}$ does not denote covariant differentiation of $T_m$. Instead, for each fixed index $m$, $T_{m;ab}$ lies in $\Omega^2_7$ and the semi-colon serves only to separate the $\Omega^1_8$-index $m$ from the two $\Omega^2_7$-indices $a,b$.

We can explicitly describe the torsion tensor $T_{m;ab}$ in local coordinates, in terms of $\Phi$, as follows. 
\begin{prop}[{\cite[Lemma 2.13]{KarigiannisFlows}}]\label{PropTorsion}
    The torsion tensor $T$ of a \Sp-structure $\Phi$ can be expressed in local coordinates as
    \begin{equation}\label{EquationTorsionTensor}
        T_{m;ab} = \frac{1}{96}(\nabla_m \Phi_{ajkl})(\Phi_{bpqr})g^{jp}g^{kq}g^{lr}.
    \end{equation}
\end{prop}

For expressing other geometric quantities in terms of torsion, it will be useful to decompose the torsion tensor $T$ into two orthogonal components. We start with the following result, which is proved in \cite[p.~8]{Dwivedi24}.

\begin{prop}\label{PropTorsionSplitting}
    Let $M$ be an $8$-manifold with \Sp-structure $\Phi$.
    Then,
    \begin{equation}\label{EquationT8T48}
        \Omega^1_8 \otimes \Omega^2_7 \cong \Omega^1_8 \oplus_\perp\{\gamma_{i;jk} \in \Omega^1_8 \otimes \Omega^2_7 \mid \gamma_{i;jk}g^{ik} = 0\} \cong \Omega^1_8 \oplus \Omega ^3_{48}.
    \end{equation}
    
\end{prop}
Thus, the torsion tensor $T$ decomposes as $T = T_8 + T_{48}$.  Note also that \eqref{EquationT8T48} gives that 
\begin{equation}
    (T_8)_j = T_{i;jk}g^{ik}.
\end{equation}
Here, $T_8$ is defined as a $1$-form, but we will also write $T_8$ for the associated dual vector field. We will also need the following definition of the \emph{divergence} of the torsion tensor.
\begin{defn}
    Let $T$ be the torsion tensor of a \Sp-structure $\Phi$ on a manifold $M$. We write $\operatorname{div}T$ for the divergence of the torsion, which is an element of $\Omega^2_7$ and is defined by 
    \begin{equation}\label{eqDivT}
        \operatorname{div}T_{jk} = g^{nm}\nabla_nT_{m;jk}.
    \end{equation} 
\end{defn}
Here, we also mention the following explicit formula for the Ricci tensor of the metric induced by a \Sp-structure $\Phi$ in terms of its torsion tensor $T$, which also illustrates the fact that torsion-free \Sp-structures induce Ricci-flat metrics.
\begin{prop}[{\cite[Proposition 4.6]{KarigiannisFlows}}]\label{PropRicciTensor}
    Let $T$ be the torsion tensor of a \Sp-structure $\Phi$ on a manifold $M$. Then, the Ricci tensor of the metric $g_\Phi$ induced by $\Phi$ can be expressed as
    \begin{equation}\label{EquationRicci}
        R_{ij} = 4\nabla_i(g^{ap}T_{a;jp}) - 4g^{ap}\nabla_aT_{i;jp} - 8T_{i;jb}g^{ap}g^{bq}T_{a;qp} + 8T_{a;jb}g^{ap}g^{bq}T_{i;qp}.
    \end{equation}
\end{prop}

\subsection{Flows of \texorpdfstring{\Sp-structures}{Spin(7)-structures}}\label{Flowsintro}

Flows of \Sp-structures were first studied in \cite{KarigiannisFlows}, where Karigiannis considered the most general flow of \Sp-structures, using the discussion of arbitrary deformations of \Sp-structures that we outlined in Section \ref{SubsectionFormDecompositions}.
In this subsection, we recall this general theory, before introducing two particular flows of \Sp-structures that have received attention in the literature recently, namely the \emph{Gradient Flow} and \emph{Ricci Harmonic Flow.}

Recall that an arbitrary infinitesimal deformation of a given \Sp-structure $\Phi$ can be written as $A \diamond \Phi$ for some $A\in S^2 \oplus \Omega^2_7$, where $\diamond$ is as defined in Equation \eqref{eqDiamond}. Thus, an arbitrary flow of \Sp-structures can be written as 
\begin{equation}\label{EquationGeneralFlow}
    \frac{\partial}{\partial t}\Phi(t) = (A(t) \diamond_t \Phi(t)) = ((h(t)+X(t))\diamond_t \Phi(t) ),
\end{equation}
where $A(t)$ is a one-parameter family of tensors in $S^2 \oplus \Omega^2_7$ and the symbol $\diamond_t$ is the diamond map taken with respect to the \Sp-structure $\Phi(t)$. We will often write $A_t$ instead of $A(t)$ and $\Phi_t$ instead of $\Phi(t)$. Where there is no risk of ambiguity, we will omit writing the $t$-dependence of tensors altogether, and note that unless otherwise stated, all terms appearing in the flow equations that follow are dependent on $t$.

Karigiannis also computes the evolution equations of various induced metric and \Sp-related quantities under the most general flow of \Sp-structures \cite[Section 3]{KarigiannisFlows}.
In particular, we have that 
\begin{equation}
  \frac{\partial}{\partial t}g_{ij} = 2h_{ij}  
\end{equation}
and
\begin{equation}
  \frac{\partial}{\partial t} T_{m;\alpha\beta}
= A_{\alpha p} \, g^{pq} T_{m;q\beta}
- A_{\beta p} \, g^{pq} T_{m;q\alpha}
+ \pi_7 \left( \nabla_\beta h_{\alpha m}
- \nabla_\alpha h_{\beta m}
+ \nabla_m X_{\alpha\beta} \right).  
\end{equation}
Using these, the evolution equations of all other metric and torsion-related quantities can be obtained. We will use this in the following section to compute the evolution of the Ricci tensor and $\nabla T$, as well as their norms.

With this general setup of flows in hand, we now introduce and briefly discuss the gradient flow and Ricci Harmonic Flows of \Sp-structures.

Motivated by the problem of finding torsion-free \Sp-structures, it is natural to consider specific flows which have torsion-free \Sp-structures as critical points. A natural candidate is a negative gradient flow of the norm of the torsion tensor, in an attempt to decrease the torsion as quickly as possible. This is the motivation behind the so-called gradient flow of \Sp-structures, introduced by Dwivedi \cite{Dwivedi24}, which we recall now. All details and derivations can be found in \cite[Sections 3-5]{Dwivedi24}.

\begin{defn}
    Given a compact $8$-manifold with Spin($7$)-structure, $(M, \Phi)$, the energy functional $E$ is defined as 
    \begin{equation}
        E(\Phi) = \frac{1}{2}\int_M \lvert T_\Phi \rvert^2 \mathrm{vol}_\Phi,
    \end{equation}
    where $T_\Phi$ is the torsion tensor of the Spin($7$)-structure $\Phi$, and the norm and volume form are those induced by the Riemannian metric $g_\Phi$. 
\end{defn}
To avoid overloading notation, we will write $T$ for $T_\Phi$.
Dwivedi computes the gradient flow of this energy functional, obtaining the following evolution equation for $\Phi(t)$.

\begin{defn}[\cite{Dwivedi24}]
    Let $(M,\Phi_0)$ be a compact $8$-manifold with initial Spin($7$)-structure $\Phi_0$. The gradient flow of \Sp-structures is the following initial value problem:

    \begin{equation}
            \begin{cases}
                 \frac{\partial}{\partial t} \Phi(t) = (-\operatorname{Ric} + 2(\mathcal{L}_{T_8}g) + T\star T - |T|^2g + 2 \text{ div}T)_t \diamond_t \Phi(t),\\
                 \Phi(0) = \Phi_0, \\
                 
            \end{cases} \tag{GF}
        \end{equation}
        where each term is induced by the Spin($7$) structure $\Phi_t$
        and 
        \begin{equation}
            (T \star T)_{ij} = 4T_{b;il}T_{j;lb} + 4T_{b;jl}T_{i;lb} - 4T_{j;il}T_{b;lb} - 4T_{i;jl}T_{b;lb} + 2T_{i;lb}T_{j;lb}.
        \end{equation}
        This flow is the negative gradient flow of the functional $2E(\Phi)$.
\end{defn}
Dwivedi proves short-time existence of the gradient flow in \cite{Dwivedi24}, as well as non-existence of compact expanding solitons. In \cite{Duthie2025}, the author studied this flow in the homogeneous setting, finding explicit solutions including a shrinking soliton \Sp-structure on $\text{SU}(3)$. However, in Section \ref{SectionReasonable}, we will see that this flow is not "reasonable" in the sense discussed in this paper, pointing to the fact that other flows may potentially be more fruitful in the future study of \Sp-structures.

Another natural way to choose a particular flow of \Sp-structures is to consider a heat equation-type evolution, inspired by Hamilton's interpretation of the Ricci flow as the heat equation for the metric \cite{HamiltonSingularities}. This is the motivation behind Dwivedi's introduction of the \emph{Ricci-Harmonic flow} \cite[Section 7]{DwivediRHF}, where Dwivedi uses a Taylor series expansion of a \Sp-structure $\Phi$ to obtain an expression for the Laplacian of the components of $\Phi$. Then, considering the flow $\ptt \Phi_{ijkl} = \Delta \Phi_{ijkl}$ and expressing it in terms of the diamond operator $\diamond$, the following flow is obtained.
\begin{defn}
Let $(M,\Phi_0)$ be a compact $8$-manifold with initial Spin($7$)-structure $\Phi_0$. The Ricci-Harmonic flow of \Sp-structures is the following initial value problem:
    \begin{equation}\label{EquationRicciHarmonic} \tag{RHF}
 \begin{cases}
\displaystyle \frac{\partial \Phi}{\partial t} = (-\operatorname{Ric}+\operatorname{div} T + T*T)\diamond \Phi, \\
\Phi(0) = \Phi_0,
\end{cases}   
\end{equation}
where $T*T$ satisfies
\begin{equation}
((T * T) \diamond \Phi)_{ijkl}
= \frac{1}{2} \Big(
T_{p;is} (T_p \diamond \Phi)_{sjkl}
+ T_{p;js} (T_p \diamond \Phi)_{iskl}
+ T_{p;ks} (T_p \diamond \Phi)_{ijsl}
+ T_{p;ls} (T_p \diamond \Phi)_{ijks}
\Big).
\end{equation}
\end{defn}
We note that the $\text{div}T$ term here differs from \cite{DwivediRHF} by a factor of $2$. This renormalisation is chosen only so that the highest order term in the induced evolution of the tensor is exactly $\Delta T$ (see Proposition \ref{propRHFreasonable}).
Using the derivations of the principle symbols of the terms $\Ric$ and $\text{div}T$ from \cite{Dwivedi24}, we see that the Ricci-Harmonic flow enjoys short-time existence and uniqueness. In Proposition \ref{propRHFreasonable}, we will see that this flow is an example of a \emph{reasonable flow} of \Sp-structures, and so the results of this paper apply to it. 

\section{Reasonable flows of \texorpdfstring{\Sp-structures}{Spin(7)-structures}}\label{SectionReasonable}
In this section, we start by defining our class of \emph{reasonable flows} of \Sp-structures, adapting Chen's definition \cite[Definition 1.1]{Chen} to the \Sp-setting. We discuss some immediate consequences of this definition, before showing that the \emph{Ricci-harmonic flow} of \Sp-structures \cite{DwivediRHF} is reasonable, and that the \emph{gradient flow} of \Sp-structures \cite{Dwivedi24} is not.

We begin with the definition.
\begin{defn}\label{DefReasonable}
    Let $\Phi(t)$ be a flow of \Sp-structures, given by the evolution equation
    \begin{equation}
        \ptt \Phi(t) = (A \diamond\Phi),
    \end{equation}
    for some $A = h+X$, as discussed in Section \ref{Flowsintro}.
   We say that the flow $\Phi(t)$ is \emph{reasonable} if it has short-time existence, uniqueness for as long as it exists, and the following schematic equations hold for the induced evolution equations of the induced Riemannian metric $g(t)$ and the torsion tensor $T(t)$, and for the skew $2$-tensor $X(t)$:
    \begin{align}
        \frac{\pt }{\pt t}g &= 2h=  -2 \Ric  + L(T) + T*T +C ,\label{eqReasonableMetric}\\
        X&= L(\nabla T)  + L(\Riem) + L( T) + T*T +C, \label{eqreasonableX}\\
         \frac{\pt }{\pt t}T&= \Delta T + L(\nabla T) + L(T) + \Riem * T \nonumber\\
         &\quad+ \nabla T * T + T*T*T + T*T \label{eqReasonableTorsion},
    \end{align}
    where $L$ denotes linear maps, $*$ denotes multilinear contractions using $g$ and $\Phi$, and $C$ denotes tensor fields independent of $g$ and $T$. We require that $C(t)$ and its covariant derivatives of all orders are uniformly bounded for as long as a the flow exists.

    \end{defn}
    \begin{remark}
        We make a brief aside here to mention some examples of this schematic notation, to get a feel for how we will deal with such terms in what follows. Firstly, an example of a tensor of schematic form $T*T$ is:
        \begin{equation}
            (T * T)_{ij} = 4T_{b;il}T_{j;lb} + 4T_{b;jl}T_{i;lb} - 4T_{j;il}T_{b;lb} - 4T_{i;jl}T_{b;lb} + 2T_{i;lb}T_{j;lb},
        \end{equation}
        which is the lower order term appearing in the gradient flow of \Sp-structures \cite{Dwivedi24}. Moreover, tensors of the type $T*\Phi$ are $L(T)$, and we have equations like $T*T*T = T*T*T*\Phi$.
        There will be occasions where we have terms involving several repeated $*$ operations. So, we will write $A^n = \underset{n\text{ times}}{\underbrace{A * \cdots * A}}$ for any tensor $A$ .

        Since contractions using $*$ may involve $\Phi$, we have the following schematic equation, for any two tensors $A$ and $B$:
        \begin{equation}
            \nabla(A*B) = \nabla A * B + A * \nabla B + A*B*T, 
        \end{equation}
        where the final term involving $T$ comes from $\nabla \Phi$. We shall use such schematic equations frequently in what follows.
    \end{remark}

We now discuss the effect on a reasonable flow of \Sp-structures of rescaling $\Phi$, which we will use when discussing parabolic rescalings of reasonable flows in order to analyse their finite-time singularities. To avoid fractional powers of the scale factor, we rescale to $\tilde\Phi = c^4\Phi$ for some constant $c>0$, and we endow any term induced by $\tilde{\Phi}$ with its own $\tilde{\quad }$. Note that $\tilde \Phi$ also defines a \Sp-structure, and that the induced metric rescales by \begin{equation}\label{eqrescaledmetric1}
    \tilde g = c^2g,
\end{equation} and so $\tilde{g}^{-1} = c^{-2}g$. Moreover, from the definition of the torsion tensor (Definition \ref{DefTorsion}), and using that the diamond operator involves the inverse metric, we see that $\tilde{T} = c^2T$. Applying the same argument to $\Ric, \Riem, \nabla$ and $*$ contractions, we have that any reasonable flow (satisfying conditions \eqref{eqReasonableMetric} and \eqref{eqreasonableX}) has $\tilde{A} = A$, and hence satisfies 
\begin{equation}\label{EquationRescaledFlow}
    \tilde{A} \tilde{\diamond} \tilde{\Phi}  = c^2(A \diamond \Phi).
\end{equation}

In this rest of this paper, the schematic evolution equations in the definition of reasonable flows will be precisely what allow us to derive our analytic results on reasonable flows. Before doing so, we consider how this notion of reasonable fits into the context of the flows of \Sp-structures studied in the literature so far. In particular, we show that the Ricci-harmonic flow is reasonable, and that the gradient flow is not. 

\begin{prop}\label{propRHFreasonable}
    The Ricci-harmonic flow, given by the initial value problem
    \begin{equation}
 \begin{cases}
\displaystyle \frac{\partial \Phi}{\partial t} = (-\operatorname{Ric}+\operatorname{div} T + T*T)\diamond \Phi, \\
\Phi(0) = \Phi_0,
\end{cases}   
\end{equation}
where $T*T$ is given by
\begin{equation}
((T * T) \diamond \Phi)_{ijkl}
= \frac{1}{2} \Big(
T_{p;is} (T_p \diamond \Phi)_{sjkl}
+ T_{p;js} (T_p \diamond \Phi)_{iskl}
+ T_{p;ks} (T_p \diamond \Phi)_{ijsl}
+ T_{p;ls} (T_p \diamond \Phi)_{ijks}
\Big),
\end{equation}
is a reasonable flow of \Sp-structures.
\end{prop}
\begin{proof}
    We verify the conditions of Definition \ref{DefReasonable}.
    In this proof, we will keep the indices of the highest order terms explicit, to show that we can arrange things to make the highest order term precisely a Laplacian, but we will use schematic notation for lower order terms.
    Recall that, for a flow of \Sp-structures given by
    \[
    \frac{\pt}{\pt t}\Phi = (h+X)\diamond\Phi,
    \]
    the evolution of the metric is
    \[
    \frac{\pt}{\pt t}g_{ij} = h_{ij}.
    \]
    For the Ricci-harmonic flow, this is
    \[
    \frac{\pt}{\pt t}g_{ij} =-2\Ric + T*T,
    \]
    which is indeed of the form given by \eqref{eqReasonableMetric}.

The induced evolution equation of the torsion tensor $T$ is given by \cite[Theorem 3.4]{KarigiannisFlows}:
\[
\frac{\pt }{\pt t}T_{m;is} =A_{ip}T_{m;ps} - A_{sp}T_{m;pi} + \pi_7(\nabla_mX_{is} + \nabla_sh_{im} - \nabla_ih_{sm}),
\]
For the Ricci-Harmonic flow \eqref{EquationRicciHarmonic}, we have 
\begin{align}
    h_{ij}&= -\Ric_{ij} + (T*T)_{ij} = -4 \nabla_iT_{a;ja} + 4\nabla_aT_{i;ja} + (T*T)_{ij},\\
    X_{ij} &= (\operatorname{div}T)_{ij} + (T*T)_{ij}  = 2\nabla_a T_{a;ij} + (T*T)_{ij},
\end{align}
where we have used expressions for $\Ric$ and $\Div T$ from \eqref{EquationRicci} and \eqref{eqDivT}.

First, note that the two terms $A_{ip}T_{m;ps}$ and $A_{sp}T_{m;pi}$ can each be written in the schematic form

\begin{equation} \label{eqTorsionEvolLowerOrder}
    A_{ip}T_{m;ps} = \nabla T * T + T*T*T = A_{sp}T_{m;pi}.
\end{equation}
Dealing with the term inside the $\pi_7$-projection, we have
\begin{align}
    \nabla_m X_{is} + \nabla_sh_{im} - \nabla_ih_{sm}&=  \nabla_m \nabla_aT_{a;is} + \nabla T*T \nonumber\\
    & \qquad+ \nabla_s(-R_{im}  + T*T)\nonumber \\
    & \qquad- \nabla_i(-R_{sm} +  T*T). \label{eqHighestOrderIntermediate}
\end{align}

Now, by \cite[Lemma 4.12]{DLE24}, we can express the Laplacian of $T$ in the form 
\[
\Delta T_{m;is} = \nabla_m\nabla_aT_{a;is} + \pi_7(\nabla_aR_{mais}) + T*\Riem + \nabla T * T + \Riem * T * \Phi.
\]
Using this, we can rewrite the right hand side of \eqref{eqHighestOrderIntermediate} as
\begin{align}
    & \Delta T_{m;is} - \pi_7(\nabla_a R_{mais}) + \nabla_i R_{sm} - \nabla_s R_{im} + T*\Riem + \nabla T * T + \Riem * T * \Phi
\end{align}
Applying the $\pi_7$-projection to the above, we obtain
\begin{align}
    \pi_7( & \Delta T_{m;is}  
    - \pi_7(\nabla_a R_{mais}) + \nabla_i R_{sm} - \nabla_s R_{im} + T*\Riem + \nabla T * T + \Riem * T * \Phi) \nonumber\\
    &=\pi_7(\Delta T_{m;is} +(- \nabla_aR_{mais} + \nabla_iR_{sm} - \nabla_sR_{im}) + T*\Riem + \nabla T * T + \Riem * T * \Phi) \nonumber\\
    &=\pi_7(\Delta T_{m;is} 
      + T*\Riem + \nabla T * T + \Riem * T * \Phi),
\end{align}
where we have used the Bianchi identity to cancel the term $(- \nabla_aR_{mais} + \nabla_iR_{sm} - \nabla_sR_{im})$.
Now, by \cite[p. 30]{DLE24},
\[
\pi_7(\Delta T_{m;is}) = \Delta T_{m;is} + \nabla T * T * \Phi + T*T*T*\Phi. 
\]
Using this, and the fact that $\pi_7(\alpha)_{ij} = \frac{1}{4}\alpha_{ij} - \frac{1}{8}\alpha_{ab}\Phi_{abij}$ for any $2$-form $\alpha$ (cf. \eqref{eqpi27coords}) for the lower order terms, we write

\begin{align}
   &\pi_7(\Delta T_{m;is} 
      + T*\Riem + \nabla T * T + \Riem * T * \Phi)\\
     &=\Delta T_{m;is} + \nabla T *T*\Phi + T*T*T*\Phi + T*\Riem + \nabla T * T + \Riem * T*\Phi \nonumber \\
     &\qquad + T*\Riem * \Phi + \Riem * T * \Phi * \Phi . \label{eqTorsionEvolPi7Part}
\end{align}
Combining this with \eqref{eqTorsionEvolLowerOrder}, and reordering terms in order of descending degree,
we obtain 
\begin{align}
\frac{\pt}{\pt t}T_{m;is} = &\Delta T_{m;is} + \Riem * T * \Phi * \Phi + \Riem * T * \Phi + \Riem * T \nonumber\\
&+ \nabla T * T * \Phi + \nabla T * T + T*T*T*\Phi + T*T*T.\\
=&\Delta T + \Riem * T   + \nabla T * T + T*T*T,
\end{align}
which is of the form described by \eqref{eqReasonableTorsion}.

We note that the same argument, just ignoring the metric terms, shows that the \emph{harmonic flow of} \Sp-\emph{structures} \cite{DLE24} is also reasonable.
\end{proof}

\begin{remark}
Note that the gradient flow of \Sp-structures \cite{Dwivedi24} is \emph{not} reasonable, since the metric evolution is of the form
\begin{equation}
    \frac{\partial}{\partial t}g = -2\operatorname{Ric} + 4 \mathcal{L}_{T_8}g + 2T \star T - 2 |T|^2g,
\end{equation}
and the term $\mathcal{L}_{T_8}g$ is not of the form given by \eqref{eqReasonableMetric}. We can get around the problem of the metric evolution by gauge fixing, but doing so makes the leading order term of the evolution equation for the torsion tensor $T$ something other than $\Delta T$, so the flow does not satisfy \eqref{eqReasonableTorsion} and so is still not reasonable.
\end{remark}

\section{Derivative estimates for reasonable flows of \texorpdfstring{\Sp-structures}{Spin(7)-structures}}\label{SectionShiType}
In this section, we develop regularity theory for all reasonable flows of \Sp-structures. Many of our results will rely on assuming a bound on the quantity 
\[
\Lambda(x,t) = \left(|\Riem(x,t)|_{g(t)}^2 + |T(x,t)|^4_{g(t)} + |\nabla T(x,t)|_{g(t)}^2 \right )^{1/2},
\]
so we begin by justifying this assumption, by proving the following doubling-time estimate which shows that $\Lambda$ does not blow up too quickly along any reasonable flow of \Sp-structures. We will also use the doubling-time estimate to study long-time-existence, in the proof of Theorem \ref{ThmLongTime}. With the doubling time estimate in hand, we shall prove the first main result of this paper, a local Shi-type derivative estimate for reasonable flows of \Sp-structures, in Theorem \ref{ThmShiType}. We will use both of these results extensively in the subsequent sections of this paper. The results of this section follow the case of flows of $G_2$-structures, as in (\cite{Lotay-Wei},\cite{Chen},\cite{DwivediRHF}).

We begin with the doubling-time estimate.
\begin{prop}\label{PropDoublingTime}
    Let $\Phi(t)$ be a solution to a reasonable flow of \Sp-structures for $t\in [0,\tau]$ for some $\tau \geq0$.
    Define 
    \begin{align}
    \Lambda(x,t) &= \left(|\Riem(x,t)|_{g(t)}^2 + |T(x,t)|^4_{g(t)} + |\nabla T(x,t)|_{g(t)}^2 \right )^{1/2},\label{eqLambdaDef}\\ 
    \tilde{\Lambda}(x,t) &= \left(|\Riem(x,t)|_{g(t)}^2 + |T(x,t)|^4_{g(t)} + |\nabla T(x,t)|_{g(t)}^2 +1 \right )^{1/2} \label{eqlambdatildedef},
\end{align}
    and let 
    \begin{equation} 
        \tilde{\Lambda}(t) = \sup_{x \in M}\tilde{\Lambda}(x,t).
    \end{equation}
    Then, there exists a constant $C$ such that 
    \[
    \tilde{\Lambda}(t) \leq 2 \tilde{\Lambda}(0) \text{ for all $t$ satisfying } 0\leq t \leq \min\left \{\tau , \frac{1}{C\Lambda(0)}\right\}.
    \]
\end{prop}
\begin{proof}
    We begin by outlining the main idea of the proof, and the reason for defining $\Lambda$ and $\tilde \Lambda$ as the particular combinations of terms we chose. As we will show shortly, the evolution equations for the quantities $|\Riem|^2$ and $|\nabla T|^2$ each contain some bad terms that cannot be controlled by the quantity in whose evolution equation they appear. However, in the evolution equation for $|\Riem|^2 + |\nabla T|^2$, the \emph{good} gradient terms $-|\nabla \Riem|^2$ and $-|\nabla^2T|^2$ allow us to absorb the bad terms. Moreover, including the quantity $|T|^4$ in $\Lambda$ and the $+1$ in $\tilde \Lambda$ allows us to bound the lower order terms of the evolution equation of $\tilde\Lambda$ in terms of $\tilde\Lambda^3$, resulting in a differential inequality that is amenable to the scalar maximum principle. The choice of the power of $4$ in the term $|T|^4$ is necessary so that $\Lambda$ scales uniformly when the \Sp-structure $\Phi$ is rescaled. The problem with working directly with $\Lambda$ arises when $\Lambda < 1$, since there are points in the argument where we need to bound $\Lambda^k$ by $C\Lambda^3$ for some $k < 3$ (e.g \eqref{eqNeedGeq1}). Working with $\tilde \Lambda$ avoids this problem, since $\tilde \Lambda$ is always at least $1$. Moreover, a doubling-time estimate for $\tilde\Lambda$ suffices for us, since $\Lambda\leq \tilde \Lambda$ so control of $\tilde \Lambda$ also provides control of $\Lambda$, showing that $\Lambda$ will not blow up too quickly along any reasonable flow.

    We begin by deriving the evolution equation for the curvature tensor along a reasonable flow of \Sp-structures.
    Recall that, for a family of Riemannian metrics $g(t)$ evolving according to the equation 
\begin{equation}
    \frac{\pt g}{\pt t}=2h,
\end{equation}
the induced Riemann curvature tensor evolves as (see e.g. \cite[Chapter 3]{Chow-Knopf})
\begin{equation}
    \frac{\partial}{\partial t} R^{l}{}_{ijk}
= \frac{1}{2} g^{lp} \Big(
\nabla_i \nabla_k h_{jp}
+ \nabla_j \nabla_p h_{ik}
- \nabla_i \nabla_p h_{jk}
- \nabla_j \nabla_k h_{ip}
- R^{q}{}_{ijk} h_{qp}
- R^{q}{}_{ijp} h_{kq}
\Big).
\end{equation}
For a reasonable flow of \Sp-structures, 
\[
h = -2 \Ric  + L(T) + T*T +C.
\]
We deal with the contribution from each term of $h$ in turn. Firstly, the term $-2\Ric$ contributes
\[
(- \nabla_i \nabla_k R^{l}{}_{j}
- \nabla_j \nabla^l R_{ik}
+ \nabla_i \nabla^l R_{jk}
+ \nabla_j \nabla_k R^{l}{}_{i} 
+ R_{ijkq} R^{l}{}_{q}
+ R^{l}{}_{ijq} R_{kq})
\]
to $\frac{\partial}{\pt t}R^l_{ijk}$. These six terms are exactly those that appear in the evolution of the curvature tensor for a metric evolving along Ricci flow and so can be expressed as 
\[
\Delta\Riem + \Riem * \Riem.
\]
(cf. \cite[Lemma 6.13]{Chow-Knopf}).
For terms of the form $L(T)$, we have 
\[
\nabla L(T) = L(\nabla T) + T*T,
\]
so \[
\nabla^2 L(T) = L(\nabla^2T) + \nabla T * T + T^3.
\]
Thus, the contribution of $L(T)$ to $\frac{\pt}{\pt t}\Riem$ is 
\[
L(\nabla^2T) + \nabla T * T + T^3 + \Riem * L(T) = L(\nabla^2T) + \nabla T * T + T^3 + \Riem * T.
\]
For $T*T$, we have $\nabla(T*T) = \nabla T * T + T^3$,
so 
\[
\nabla^2 (T*T) = \nabla^2T * T + \nabla T * \nabla T + T^3 + \nabla T*T^2 + T^4.
\]
So, the total contribution from $T*T$ is 
\[
\nabla^2T * T + \nabla T * \nabla T + T^3 + \nabla T*T^2 + T^4 + \Riem*T^2.
\]
Finally, $C$ and all its covariant derivatives are uniformly bounded so the contribution from the tensor $C$ is $\Riem * C + \tilde{C}= L(\Riem) + \tilde{C}$, for some tensor $\tilde{C}$, all of whose covariant derivatives are bounded.
Combining all of this together, we obtain:
\begin{align}\label{eqddtRiem}
    \frac{\pt}{\pt t}\Riem = &\Delta \Riem + \Riem * \Riem + L(\nabla^2T) + \nabla^2T*T + \nabla T * \nabla T + \nabla T * T^2 + \nabla T * T\nonumber\\
    &+ \Riem * T^2 + \Riem * T + L(\Riem) + T^4 + \tilde{C}.
\end{align}
Using this, we can derive the evolution equation for the quantity $|\Riem|^2$ along a reasonable flow. Since for any tensor $A$, \[\Delta|A|^2 = 2 \langle A, \Delta A \rangle + 2 |\nabla A |^2,\] we have that 
\begin{align}
    \frac{\pt}{\pt t}|\Riem|^2  &= \frac{\pt}{\pt t }(g^{ia}g^{jb}g^{kc}g^{ld} R_{ijkl}R_{abcd}) \nonumber\\
    &= \Riem^2* \frac{\pt}{\pt t}g + 2\left \langle \Riem, \frac{\pt}{\pt t}\Riem \right \rangle \nonumber \\
    &=\Riem^2 * (\Ric + T * T + L(T) +C) + 2 \langle \Riem, \Delta \Riem \rangle \nonumber \\& \qquad+ 2\langle \Riem,\Riem^2 + L(\nabla^2T) + \nabla^2T*T + \nabla T * \nabla T + \nabla T * T^2 + \nabla T * T\nonumber\\
    &\qquad+\Riem * T^2 + \Riem * T + L(\Riem) + T^4+\tilde{C}\rangle \nonumber\\
    &\leq\Delta|\Riem|^2 -  2 |\nabla \Riem|^2\nonumber\\
    &\qquad +C|\Riem|(|\Riem|^2 + |\nabla^2T| + |\nabla^2T||T| +|\nabla T|^2 + |\nabla T||T|^2 + |\nabla T||T| + |\Riem||T|^2\nonumber\\
    &\qquad \qquad \qquad + |\Riem||T| +|\Riem| + |T|^4 +1)\nonumber\\
    &\qquad + C|\Riem|^2(|\Riem| + |T| + |T|^2 +C) \nonumber\\
    &\leq\Delta|\Riem|^2 -  2 |\nabla \Riem|^2\nonumber\\
    &\qquad +C|\Riem|(|\Riem|^2 + |\nabla^2T| + |\nabla^2T||T| +|\nabla T|^2 + |\nabla T||T|^2 + |\nabla T||T| + |\Riem||T|^2\nonumber\\
    &\qquad \qquad \qquad + |\Riem||T| +|\Riem| + |T|^4+1) \label{eqddtnormRm}
\end{align}
for some constant $C$. Note that we have absorbed the bounded tensor $\tilde{C}$ in the first inequality.

We continue with the evolution equations of the other parts of $\Lambda$. Recall that, for any reasonable flow of \Sp-structures,
\[\frac{\pt }{\pt t}T_{m;is}= \Delta T_{m;is} + L(\nabla T) + L(T) + \Riem * T 
         + \nabla T * T + T^3 + T^2.\]
With this, we compute
\begin{align}
    \frac{\pt}{\pt t}|T|^4 &= \frac{\pt}{\pt t}(g^{ia}g^{jb}g^{kc}g^{ld}T_{ij}T_{ab}T_{kl}T_{cd})\nonumber \\
    &=2|T|^2\left \langle 2\frac{\pt}{\pt t}T,T \right \rangle + T^4*\left(\frac{\pt}{\pt t}g\right) \nonumber \\
    &= 4|T|^2\left \langle \Delta T, T \right \rangle + 4|T|^2\langle L(\nabla T) + L(T) + \Riem * T 
         + \nabla T * T + T^3 + T^2, T\rangle \nonumber \\
         &\qquad + T^4*(\Riem + T^2 + L(T)+C)\\
    &\leq \Delta |\ T|^4 - 4|T|^2|\nabla T|^2 + C|T|^3(|\nabla T| +|T| +|\Riem| |T| +|\nabla T||T| +|T|^3 +|T|^2). \label{eqddtnormt}
\end{align}
 To derive the evolution equation for $|\nabla T|^2$, we recall the following commutation formula for any time-dependent tensor $A$:
 \begin{equation}\label{eqddtNabla}
     \frac{\pt}{\pt t}\nabla A = \nabla \ptt A + A* \nabla \ptt g.
 \end{equation}
 So,
 \begin{align}\label{eqddtnablaT}
     \ptt \nabla T &= \nabla \ptt T + T* \nabla \ptt g\nonumber \\
     &=\nabla(\Delta T + L(\nabla T) + L(T) + \Riem * T
        + \nabla T * T + T^3 + T^2) \nonumber\\
        &\quad+ \nabla( \Riem + L(T) + T^2 +C)\nonumber\\
    &=\Delta \nabla T + \Riem * \nabla T + \nabla \Riem * T + L(\nabla T) + T^2 + L(\nabla^2 T) + \nabla T*T\nonumber\\
    &\quad + \Riem *T^2 + \nabla^2T*T + \nabla T*\nabla T + \nabla T*T^2 + T^3 + T^4,
 \end{align}
 where in the last line we have used that $\nabla \Delta T = \Delta \nabla T + \Riem*\nabla T + \nabla \Riem *T$, which follows from the Ricci identity for commuting covariant derivatives.
 Using this, we see that
 \begin{align}
     \ptt |\nabla T|^2 &= 2 \left \langle \ptt \nabla T, \nabla T \right \rangle + \nabla T * \nabla T * (\Riem + L(T) + T*T +C)\nonumber \\
     &\leq \Delta|\nabla T|^2 - 2 |\nabla^2 T|^2 +C|\Riem||\nabla T|^2 + C|\nabla T||\nabla^2 T ||T| \nonumber\\
        &\quad + C|\nabla T|^3 + C|\nabla T|^2 |T|^2 + C|\nabla T||T|^4+ C|\nabla T||T|^2|\Riem| + C|\nabla T||T||\nabla \Riem| \nonumber \\
        &\quad +C|\nabla T|^2 +C|\nabla T||T|^2 +C|\nabla T||\nabla^2T| + C|\nabla T|^2|T| + C|\nabla T||T|^3. \label{eqddtnormnablaT}
 \end{align}
Combining equations \eqref{eqddtnormRm}, \eqref{eqddtnormt} and \eqref{eqddtnormnablaT} gives the following evolution inequality for $\tilde{\Lambda}:$
\begin{align}
    \ptt\tilde\Lambda^2 &\leq \Delta\tilde\Lambda^2 - 2|\nabla \Riem|^2 -2|\nabla^2T|^2\nonumber \\
                        &\quad + C|\Riem|^3 + C|\Riem||\nabla^2T|+C|\Riem||\nabla^2T||T| + C|\Riem||\nabla T|^2 + C|\Riem||\nabla T||T|^2 \nonumber \\
                        &\quad+ C|\Riem||\nabla T||T|  + C|\Riem|^2|T|^2 + C|\Riem|^2|T| + C|\Riem|^2+C|\Riem||T|^4 + C|\Riem|\nonumber \\
                        &\quad+ C|T|^2|\nabla T|^2 + C|T|^3|\nabla T| + C|T|^4 + C|T|^4|\Riem| + C|T|^4 |\nabla T| +|T|^6 + |T|^5\nonumber \\
                        &\quad + C|\nabla T||\nabla^2 T||T| + C|\nabla T|^3 + C|\nabla T|^2 + C|\nabla T||T|^2 + C|\nabla T||\nabla ^2 T| +C|\nabla T|^2 |T|\nonumber\\
                        &\quad +C|\nabla T||T||\nabla \Riem|.\label{eqLambdaSquaredLowerOrder}
\end{align}
Now, the terms above that are polynomial in $|\Riem|,|T|^2$ and $|\nabla T|$ all have total degree at most $3$ can be immediately bounded by $C\tilde\Lambda^3$ for some universal constant $C$. For instance, 
\begin{equation}\label{eqNeedGeq1}
    |T|^5 = (|T|^2)^{5/2} \leq \tilde\Lambda^{5/2} \leq \tilde\Lambda^3
\end{equation}
since $\tilde\Lambda \geq 1$. The only terms not of this form are $|\Riem||\nabla^2T|,|\Riem||\nabla^2T||T|,|\nabla T||\nabla^2 T||T|,|\nabla T||\nabla^2T|$ and $|\nabla T||T||\nabla \Riem|.$ We deal with these using Young's inequality, as follows. For any $\varepsilon>0$, we have
\begin{align}
    |\Riem||\nabla^2T| &\leq \frac{1}{2\varepsilon} |\Riem|^2 
+ \frac{\varepsilon}{2} |\nabla^2 T|^2,\\
    |\Riem||\nabla^2T||T|&\leq\frac{1}{2\varepsilon} |\Riem|^2 |T|^2
+ \frac{\varepsilon}{2} |\nabla^2 T|^2,\\
    |\nabla T||\nabla^2 T||T|&\leq \frac{1}{2\varepsilon} |\nabla T|^2 |T|^2
+ \frac{\varepsilon}{2} |\nabla^2 T|^2,\\
    |\nabla T||\nabla^2T|&\leq \frac{1}{2\varepsilon} |\nabla T|^2 
+ \frac{\varepsilon}{2} |\nabla^2 T|^2,\\
    |\nabla T||T||\nabla \Riem|&\leq \frac{1}{2\varepsilon} |\nabla T|^2 |T|^2
+ \frac{\varepsilon}{2} |\nabla \Riem|^2.
\end{align}
Substituting these five inequalities into \eqref{eqLambdaSquaredLowerOrder}, and bounding the lower order terms as discussed, we obtain:
\[
\ptt \tilde \Lambda(x,t)^2 \leq \Delta\tilde\Lambda(x,t)^2 + (C\varepsilon-2)(|\nabla \Riem|^2+|\nabla^2T|^2) + C\tilde\Lambda(x,t)^3
\]
for any $\varepsilon>0.$ Choosing $\varepsilon$ such that $C\varepsilon-2 \leq -1$ yields
\begin{equation}\label{eqLambdaTilde}
   \ptt \tilde \Lambda(x,t)^2 \leq \Delta\tilde\Lambda(x,t)^2 -(|\nabla \Riem|^2+|\nabla^2T|^2) + C\tilde\Lambda(x,t)^3.
\end{equation}

Recall that $\tilde\Lambda(t) = \sup_{x\in M}\tilde\Lambda(x,t)$, which is a Lipschitz function of $t$. Note here that $\tilde \Lambda(t)$ may not be differentiable in time, even though $\tilde \Lambda(x,t)$ is smooth. So, to consider derivatives in time in order to apply to maximum principle, we work in the sense of lim sups of forwards difference quotients, writing
\begin{equation}
    \frac{\d}{\d t}\tilde \Lambda(t) \coloneqq \limsup_{h \to 0^+}\frac{\tilde\Lambda(t+h) - \tilde \Lambda(t)}{h}, 
\end{equation}
which is well-defined and finite precisely because $\tilde \Lambda(t)$ is Lipschitz. With this in mind, an application of the maximum principle to \eqref{eqLambdaTilde} gives
\begin{equation}
    \frac{\d}{\d t}\tilde \Lambda^2(t) \leq C \tilde \Lambda^3(t),
\end{equation}
which implies
\begin{equation}
    \frac{\d}{\d t}\tilde\Lambda(t) \leq \frac{C}{2}\tilde\Lambda(t)^2.
\end{equation}
 Thus, for $t \leq \min\left\{\tau , \frac{2}{C\tilde\Lambda(0)}\right\}$ we have 
\[
\tilde\Lambda(t) \leq \frac{\tilde\Lambda(0)}{1-\frac{1}{2}\tilde\Lambda(0)t},
\]
and so $\tilde\Lambda(t) \leq 2\tilde\Lambda(0)$ for all such $t$.
\end{proof}
This shows that $\tilde \Lambda$, and hence $\Lambda$, will not blow up too quickly along any reasonable flow of \Sp-structures. So, the assumption of the following Shi-type derivative estimate is valid.
\begin{thm}\label{ThmShiType}
    Let $K>0$ and $r>0$. Let $M$ be an $8$-manifold, $p \in M$ a point, and $\Phi(t)$ be a solution to a given reasonable flow of \Sp-structures on an open neighbourhood $U$ of $p$ containing $B_{g(0)}(p,r)$ as a compact subset.

    Suppose $\Lambda(x,t) \leq K$ for all $(x,t) \in U \times [0,1/K]$. Then, for all $m\in \mathbb{N}$, there exists a constant $C(K,m,r)$ such that
    \[
    |\nabla^m\Riem(x,t)| + |\nabla^{m+1}T(x,t)| \leq C(K,m,r)t^{-m/2},
    \]
    for all $x \in B_{g(0)}(p,r/2)$, $t \in [0,1/K]$.
\end{thm}
\begin{proof}
    We proceed by induction on $m$. The key idea is to define a suitable function $F_m$ that satisfies a differential inequality amenable to the maximum principle, and such that a bound on $F_m$ implies a bound on $|\nabla^m\Riem| + |\nabla^{m+1}T|$. Our choice of function $F_m$ follows \cite{Lotay-Wei} and \cite{Chen}, and the structure of our local argument is adapted from \cite[Theorems 14.10, 14.14]{CCGGIIKLLN2} in the Ricci flow setting, which is based on Shi's argument \cite{Shi}.
    We begin with the base case $m=1$, defining
    \[
    F = (C+|\Riem|^2+|T|^4+|\nabla T|^2)(|\nabla \Riem|^2 + |\nabla^2T|^2)t,
    \]
    for some constant $C$ to be determined later. We now make the following claim, whose proof we will postpone to the end of this theorem, in order to streamline the argument.
    \begin{claim}\label{ClaimBaseCase} For $F = (C+|\Riem|^2+|T|^4+|\nabla T|^2)(|\nabla \Riem|^2 + |\nabla^2T|^2)t,$
        \[
        \left (\ptt - \Delta \right ) F \leq -\frac{CF^2}{t} +\frac{C}{t},
        \] for all $(x,t) \in U \times [0,1/K]$, where the $C$ in the coefficient of $F^2$ on the right hand side is positive.
    \end{claim}
    To localise $F$ and make use of our assumed bounds on the neighbourhood $U$, we define a cut-off function $\eta: U \to \mathbb{R}$ satisfying the following conditions:
    \begin{itemize}
        \item $\eta \equiv 1$ on $B_{g(0)}(p,r/2)$
        \item $\text{supp}(\eta) \subset B_{g(0)}(p,r)$
        \item $-\Delta_{g(t)} \eta + 2 \eta^{-1}|\nabla\eta|_{g(t)}^2 <C$ for some constant $C$.
    \end{itemize}
    Note that such a cut-off function exists by \cite[Lemma 14.3]{CCGGIIKLLN2}.

    Then, we see that 
    \begin{equation}\label{eqEvolutionetaF}
        \left(\ptt-\Delta\right)(\eta F) \leq \eta\left(-\frac{CF^2}{t} + \frac{C}{t} \right) -(\Delta\eta)F - 2\nabla \eta \cdot \nabla F.
    \end{equation}

    Let $(x_0,t_0)$ be a point in $B_{g(0)}(p,r) \times[0,1/K]$ where $\eta F$ attains its maximum. Note that since $\text{supp}(\eta)\subset B_{g(0)}(p,r)$, this point must be a spatial interior point. We have two cases to consider. If $t_0=0$, then $F(x_0,t_0) = 0$ and so $F=0$, which gives the required bound. If $t_0>0$, we have 
    \begin{equation}\label{eqSpatialMaximum}
        \ptt(\eta F) \geq 0 , \qquad\Delta(\eta F) \leq 0, \qquad0=\nabla(\eta F) = \eta \nabla F +F\nabla \eta.
    \end{equation}
       For the purpose of applying the maximum principle, we can deal with $\eta$ as if it is smooth, using the so-called \emph{Calabi trick} (explained in \cite[pp. 453-456]{CCGGIIKLLN1}).
       Substituting \eqref{eqSpatialMaximum} into \eqref{eqEvolutionetaF} and multiplying through by $t\eta$ gives
       \begin{equation}
           0 \leq -C(\eta F)^2 + t(-\Delta\eta + 2 \eta^{-1}|\nabla \eta|^2)\eta F + C\eta^2.
       \end{equation}
        By construction of the cutoff function $\eta$, we have that $-\Delta_{g(t)} \eta + 2 \eta^{-1}|\nabla\eta|_{g(t)}^2 <C$ for some constant $C$. This, together with the fact that $t_0 \leq \frac{1}{K}$ gives
        \begin{equation}
            0 \leq -C_1(\eta F)^2 +C_2\eta F + C_3,
        \end{equation}
        for some constants $C_i$ depending only on $r$ and $K$. 
        Thus,
       \[\eta F \leq C\] for all $(x,t) \in B_{g(0)}(p,r)\times [0,1/K]$, for some constant $C$ depending only on $K$ and $r$.
    
       Now, on $B_{g(0)}(p,r/2)\times [0,1/K]$, $\eta = 1$ so $\eta F = F$ and so $F\leq C$. The definition of $F$ then yields
       \[
       |\nabla \Riem|^2 +|\nabla^2 T|^2 \leq \frac{C}{t},
       \]
       which proves the base case, modulo the proof of Claim \ref{ClaimBaseCase}.

       We now proceed with the inductive step. We assume that, given $m\in \mathbb{N}$,
       \[
            |\nabla^j\Riem| + |\nabla^{j+1}T| \leq C(K,r,j)t^{-j/2} \text{ in }B_{g(0)}(p,r/2^j) \times [0,1/K],
        \]
        for all $j =0,1,\cdots,m$, and we shall prove 
        \[
            |\nabla^{m+1}\Riem| + |\nabla^{m+2}T| \leq C(K,r,m)t^{-(m+1)/2} \text{ in }B_{g(0)}(p,r/2^{m+1}) \times [0,1/K].
        \]

        The choice of radius $r/2^j$ here is just a convenient bookkeeping device for the inductive process. As we will see shortly, all we actually need for the proof is for the radius in the $(j+1)$th step to be strictly smaller than in the $j$th step. So, we see that essentially the same argument will give the estimate on all of $B_{g(0)}(p,r/2)$.

       Analogously to the base case, we define
       \[
       F_m = (\mu_m+t^m(|\nabla^m\Riem|^2+|\nabla^{m+1}T|^2))t^{m+1}(|\nabla^{m+1}\Riem|^2+|\nabla^{m+2}T|^2),
       \]
       for some constant $\mu_m$ to be determined later.
       A long and uninspiring calculation results in the following, which we shall prove at the end of this theorem.
       \begin{claim}\label{ClaimInductiveStep}
            For $F_m = t^{m+1}(\mu_m+t^m(|\nabla^m\Riem|^2+|\nabla^{m+1}T|^2))(|\nabla^{m+1}\Riem|^2+|\nabla^{m+2}T|^2),$
        \[
        \left (\ptt - \Delta \right ) F_m \leq -\frac{CF_m^2}{t} +\frac{C}{t}.
        \]
        in $B_{g(0)}(p,r/2^m)\times[0,1/K]$, for some constant $C>0$ and some constant $\mu_m$ to be determined later.   
       \end{claim}
       Again, to localise $F_m$, we define a cut-off function $\eta_m$ satisfying the following conditions:
       \begin{itemize}
        \item $\eta_m \equiv 1$ on $B_{g(0)}(p,r/2^{m+1})$
        \item $\text{supp}(\eta) \subset B_{g(0)}(p,r/2^m)$
        \item $-\Delta_{g(t)} \eta + 2 \eta^{-1}|\nabla\eta|_{g(t)}^2 <C$ for some constant $C$.
    \end{itemize}
    Then, as before,
    \begin{equation}
        \left(\ptt-\Delta\right)(\eta_m F_m) \leq \eta_m\left(-\frac{CF_m^2}{t} + \frac{C}{t} \right) -(\Delta\eta_m)F_m - 2\nabla \eta_m \cdot \nabla F_m.
    \end{equation}

    Let $(x_0,t_0)$ be a point in $B_{g(0)}(p,r/2^m) \times[0,1/K]$ where $\eta_m F_m$ attains its maximum. We again have two cases to consider. If $t_0=0$, then $F_m(x_0,t_0) = 0$ and so $F_m=0$, which gives the required bound. If $t_0>0$, we have 
    \[
    \ptt(\eta_m F_m) \geq 0 , \qquad\Delta(\eta_m F_m) \leq 0, \qquad0=\nabla(\eta_m F_m) = \eta_m \nabla F_m +F_m\nabla \eta_m,
    \]
    where we again use Calabi's trick.
       Thus,
       \begin{equation}
           0 \leq -C(\eta_m F_m)^2 + t(-\Delta\eta_m + 2 \eta_m^{-1}|\nabla \eta_m|^2)\eta_m F_m + C\eta^2,
       \end{equation}
       which again implies that
       \[
       \eta_mF_m \leq C,
       \]
       for all $(x,t)\in B_{g(0)}(p,r/2^m)\times [0,1/K]$, for some constant $C$ depending only on $K,m$ and $r$. So, for $(x,t)\in B_{g(0)}(p,r/2^{m+1})\times [0,1/K]$, $\eta_m \equiv 1$ so
       $F_m \leq C$ which implies
       \[
       |\nabla^{m+1}\Riem| + |\nabla^{m+2}T| \leq C(K,r,m)t^{-(m+1)/2} ,
       \]
       as required. To conclude the proof of Theorem \ref{ThmShiType}, it remains only to prove Claims \ref{ClaimBaseCase} and \ref{ClaimInductiveStep}.
       \begin{proof}[Proof of Claim \ref{ClaimBaseCase}]
       First, we deal with $\tilde{F} = (C+|\Riem|^2+|T|^4+|\nabla T|^2)(|\nabla \Riem|^2 + |\nabla^2T|^2)$, i.e. without the factor of $t$.
           We compute:
           \begin{align}\label{FtildeEquation}
             \left (\ptt - \Delta \right ) \tilde{F} = &(|\nabla \Riem|^2 + |\nabla^2T|^2)\left (\ptt - \Delta \right )(C+|\Riem|^2+|T|^4+|\nabla T|^2)\nonumber\\
             &+ (C+|\Riem|^2+|T|^4+|\nabla T|^2)\left (\ptt - \Delta \right )(|\nabla \Riem|^2 + |\nabla^2T|^2)\nonumber\\
             &-2\nabla (C+|\Riem|^2+|T|^4+|\nabla T|^2) \cdot \nabla (|\nabla \Riem|^2 + |\nabla^2T|^2).
           \end{align}
           We have already computed the first term (see \eqref{eqLambdaTilde}), so we begin by focussing on the second. Using \eqref{eqddtNabla} and \eqref{eqddtRiem}, we see that
           \begin{align}
               \ptt \nabla \Riem &= \nabla \ptt\Riem + \Riem * \nabla \ptt g\nonumber\\\
               &=\nabla(\Delta \Riem + \Riem * \Riem + L(\nabla^2T) + \nabla^2T*T + \nabla T * \nabla T + \nabla T * T^2 + \nabla T * T\nonumber\\
               &\quad+ \Riem * T^2 + \Riem * T + L(\Riem) + T^4) \\&+ \Riem*\nabla(\Riem + L(T)+ T^2 +C)\nonumber\\
               &= \Delta\nabla\Riem + \nabla \Riem * \Riem + \Riem^2*T + L(\nabla^3T) + \nabla^2T*T + \nabla^3T*T + \nabla^2T*\nabla T\nonumber\\
               &\quad + (\nabla T)^2+\nabla T *T^2 + \nabla \Riem * T^2 + \nabla T * \Riem * T + \Riem * T^3 + \nabla \Riem * T + \Riem * \nabla T \nonumber \\
               &\quad + \Riem *T^2 + L(\nabla \Riem ) + \Riem * T + T^5,
           \end{align}
           where we have again used the Ricci identity to commute covariant derivatives. With this, we see that
           \begin{align}
               \ptt |\nabla \Riem|^2 \leq &\Delta|\nabla\Riem|^2 -2|\nabla^2\Riem|^2 +C|\nabla \Riem|\big(|\nabla\Riem||\Riem| + |\Riem|^2|T|+|\nabla^2T| +|\nabla^2T||T|\nonumber\\
               &+ |\nabla^3T||T| + |\nabla^2T||\nabla T| + |\nabla T|^2|T| + |\nabla T|^2|T|^2 + |\nabla T||T|^3 + |\nabla T|^2\nonumber\\
               &+ |\nabla T||T|^2 + |\nabla\Riem||T|^2 + |\nabla T||\Riem||T| + |\Riem||T|^3 + |\nabla \Riem||T| + |\Riem||\nabla T| \nonumber \\
               &+|\Riem||T|^2 + |\nabla \Riem| + |\Riem||T| + |T|^5\big).
           \end{align}
            Now, considering $(x,t) \in U \times [0,1/K]$, we have that $\Lambda \leq K$, so each of $|\Riem|$, $|T|$ and $|\nabla T|$ is bounded. Using this, we can write:
            \begin{equation}\label{eqddtnormnablaRm}
                \left( \ptt-\Delta\right) |\nabla \Riem|^2 \leq -2|\nabla^2\Riem|^2 +C|\nabla \Riem|\big( |\nabla \Riem| + 1 + |\nabla^2T| + |\nabla^3T| \big),
            \end{equation}
            for some constant $C$ depending only on $K$ and $r$.

           Likewise, using \eqref{eqddtnablaT}, we have
           \begin{align}
               \ptt \nabla^2 T &= \nabla \ptt \nabla T + \nabla T* \nabla \ptt g\nonumber \\
               &=\nabla(\Delta \nabla T + \Riem * \nabla T + \nabla \Riem * T + L(\nabla T) + T^2 + L(\nabla^2 T) + \nabla T*T\nonumber\\
                &\quad + \Riem *T*T + \nabla^2T*T + (\nabla T)^2+ \nabla T*T^2 + T^3 + T^4) \nonumber \\
                &\quad + \nabla T*\nabla(\Riem + L(T) +T*T +C)\\
            &=\Delta\nabla^2T + \Riem* \nabla^2T + \nabla\Riem*\nabla T + \nabla^2\Riem*T + \Riem * \nabla T*T + \nabla \Riem *T^2 \nonumber\\
            &\quad + L(\nabla^2T) + \nabla T *T + T^3 + L(\nabla^3T) + \nabla^2 T*T + \nabla T *T^2 + \Riem * T^3 \nonumber \\
            &\quad + \nabla^3 T * T+ \nabla^2 T * \nabla T + \nabla^2T * T^2 + (\nabla T)^2 *T + \nabla T * \nabla T * T + \nabla T*T^3\nonumber\\
            & \quad + T^4 + T^5.\label{eqddtnormnablasquaredT}
           \end{align}
           Using this, we obtain
        
               \begin{align}
    \ptt |\nabla^2 T|^2
    &\leq \Delta |\nabla^2 T|^2 - 2 |\nabla^3 T|^2 \nonumber \\
    &\quad + C |\nabla^2 T| \big(
    |\nabla^2 T||\Riem| + |\nabla^2 T||T|^2 + |\nabla^2 T||T| + |\nabla^2 T||\nabla T| \nonumber \\
    &\quad \quad \quad \quad \quad  + |\nabla^3 T| + |\nabla^3 T||T| + |\nabla \Riem||\nabla T| + |\nabla^2 \Riem||T| \nonumber \\
    &\quad \quad \quad \quad \quad + |\nabla T||T| + |\nabla T||T|^2 + |T|^3 + |T|^4 + |T|^5 + |\nabla T|^2|T|
    \big).
\end{align}
Again, working on $(x,t) \in U \times [0,1/K]$ and bounding the lower order terms, we have
\begin{equation}
    \left(\ptt - \Delta \right)|\nabla^2T|^2 \leq - 2|\nabla^3T|^2 +C|\nabla^2T|\big(|\nabla^2T| + |\nabla^3T| + |\nabla \Riem| + |\nabla^2 \Riem| \big).
\end{equation}
We now combine \eqref{eqddtnormnablaRm} and \eqref{eqddtnormnablasquaredT} to obtain an inequality for $\left(\ptt-\Delta\right)(|\nabla \Riem|^2 +|\nabla^2T|^2)$:

\begin{align}\label{eqddtminusdeltanablaRmplusnablasquaredT}
    \left(\ptt-\Delta\right)(|\nabla \Riem|^2 +|\nabla^2T|^2) \leq &-2|\nabla^2\Riem|^2 -2|\nabla^3T|^2 \nonumber\\
    &+ C|\nabla \Riem|\big( |\nabla \Riem| + 1 + |\nabla^2T| + |\nabla^3T| \big)\nonumber\\
    &+C|\nabla^2T|\big(|\nabla^2T| + |\nabla^3T| + |\nabla \Riem| + |\nabla^2 \Riem| \big).
\end{align}
We apply Young's inequality as follows on the right hand side:
\begin{align*}
|\nabla \Riem|
&\leq \frac{1}{2\varepsilon} |\nabla \Riem|^2 + \frac{\varepsilon}{2},\\
|\nabla \Riem|\,|\nabla^2T|
&\leq \frac{1}{2\varepsilon} |\nabla \Riem|^2
+ \frac{\varepsilon}{2} |\nabla^2 T|^2,\\
|\nabla \Riem|\,|\nabla^3T|
&\leq \frac{1}{2\varepsilon} |\nabla \Riem|^2
+ \frac{\varepsilon}{2} |\nabla^3 T|^2,\\
|\nabla^2T|\,|\nabla^3T|
&\leq \frac{1}{2\varepsilon} |\nabla^2 T|^2
+ \frac{\varepsilon}{2} |\nabla^3 T|^2,\\
|\nabla^2T|\,|\nabla \Riem|
&\leq \frac{1}{2\varepsilon} |\nabla^2 T|^2
+ \frac{\varepsilon}{2} |\nabla \Riem|^2,\\
|\nabla^2T|\,|\nabla^2 \Riem|
&\leq \frac{1}{2\varepsilon} |\nabla^2 T|^2
+ \frac{\varepsilon}{2} |\nabla^2 \Riem|^2,
\end{align*}
for any $\varepsilon>0$.
Now, choosing $\varepsilon$ sufficiently small gives
\begin{align}\label{eqddtB}
    \left(\ptt-\Delta\right)(|\nabla \Riem|^2 +|\nabla^2T|^2) &\leq -|\nabla^2\Riem|^2 -|\nabla^3T|^2 +C\big(|\nabla \Riem|^2 + |\nabla^2 T|^2+1\big).
\end{align}
We now deal with the rest of \eqref{FtildeEquation}. Writing $A = (C+|\Riem|^2+|T|^4+|\nabla T|^2)$ and $B=(|\nabla \Riem|^2 + |\nabla^2T|^2)$, we have 
\[
\left( \ptt -\Delta \right)\tilde{F} = A \left( \ptt -\Delta \right)B + B\left( \ptt -\Delta \right)A - 2\nabla A \cdot \nabla B.
\]
The first and second terms are now handled by equations \eqref{eqLambdaTilde} and \eqref{eqddtB}, respectively, so we focus on the term $\nabla A \cdot \nabla B$.
We have 
\begin{align}
    \nabla A &= \nabla|\Riem|^2 + \nabla|T|^4 +|\nabla T|^2 \nonumber \\
            &= 2\langle \nabla \Riem,\Riem \rangle   + 2\nabla (|T|^2) |T|^2 + 2\langle \nabla^2 T,\nabla T \rangle \nonumber \nonumber \\
            &=2\langle \nabla \Riem,\Riem \rangle   + 4\langle\nabla T, T \rangle|T|^2 + 2\langle \nabla^2 T,\nabla T \rangle\nonumber \\ 
            &\leq C|\nabla \Riem| +C + C|\nabla^2 T| \nonumber\\
            &\leq C(\sqrt{B}+1),\end{align}
where we have used the fact that the metric is parallel with respect to the Levi-Civita connection to take covariant derivatives of the inner products in the second and third line, and that $|T|,|\nabla T|$ and $|\Riem|$ are bounded in the fourth line.
Likewise,
\begin{equation}
    \nabla B \leq C(\sqrt{B}+1)(|\nabla^2 \Riem | + |\nabla^3 T|).
\end{equation}
So, by Cauchy-Schwarz,  
\begin{align}\label{eqNablaANablaB}
    \nabla A \cdot \nabla B &\leq C\left(\sqrt{B}+1\right)^2(|\nabla^2\Riem| + |\nabla^3T|) \nonumber \\
    &\leq C(B+1)(|\nabla^2\Riem| + |\nabla^3T|) \nonumber \\
    &\leq \frac{C\varepsilon}{2}(|\nabla^2\Riem|^2 +|\nabla^3T|^2) + \frac{C}{2\varepsilon}(B+1)^2,
\end{align}
for any $\varepsilon>0$, where we have used Young's inequality and the inequality $(a+b)^2 \leq 2(a^2+b^2)$.
Combining equations \eqref{eqLambdaTilde}, \eqref{eqddtB} and \eqref{eqNablaANablaB}, we obtain     
\begin{align}
    \left( \ptt -\Delta \right)\tilde{F} \leq &(C+|\Riem|^2+|T|^4+|\nabla T|^2)\left(-|\nabla^2\Riem|^2 -|\nabla^3T|^2 +C\big(|\nabla \Riem|^2 + |\nabla^2 T|^2+1\big)\right)\nonumber\\
    &+(|\nabla \Riem|^2 + |\nabla^2T|^2)(-(|\nabla \Riem|^2+|\nabla^2T|^2) + C\tilde\Lambda(x,t)^3)\nonumber\\
    &+\frac{C\varepsilon}{2}(|\nabla^2\Riem|^2 +|\nabla^3T|^2) + \frac{C}{2\varepsilon}(B+1)^2.
\end{align}
Now, we have that $(C+|\Riem|^2+|T|^4+|\nabla T|^2)$ and $\tilde{\Lambda}^3$ are bounded by assumption, so there exists $\tilde{C}>0$ such that
\begin{align}
    \left( \ptt -\Delta \right)\tilde{F} \leq &- \tilde{C}\left( |\nabla^2\Riem|^2 + |\nabla^3T|^2\right) + C'(B+1)\nonumber\\
    &-CB^2 +C \nonumber \\
    &+\frac{C\varepsilon}{2}(|\nabla^2\Riem|^2 +|\nabla^3T|^2) + \frac{C}{2\varepsilon}(B+1)^2,
\end{align}
which, upon choosing $\varepsilon$ so that $-\tilde{C} + \frac{C\varepsilon}{2} \leq -1$, and using Young's inequality to absorb the terms linear in $B$, can be expressed as
\begin{align}
    \left( \ptt -\Delta \right)\tilde{F} \leq &- \left( |\nabla^2\Riem|^2 + |\nabla^3T|^2\right) -CB^2 +C\\
    & \leq -CB^2+C,
\end{align}
for some universal constant $C>0$.

Now $B = \frac{\tilde{F}}{A} \geq \frac{\tilde{F}}{C}$ so $-CB^2 \leq -C\tilde{F}^2$ (where the two constants $C$ are possibly different). So,
\[
\left( \ptt -\Delta \right)\tilde{F} \leq -C\tilde{F}^2 +C.
\]
Finally, applying the product rule to $F = t\tilde{F}$ and using that $t\leq \frac{1}{K}$ results in 
\begin{equation}
    \left (\ptt - \Delta \right ) F \leq -\frac{CF^2}{t} +\frac{C}{t},
\end{equation}
for all $(x,t) \in U \times [0,1/K]$,
which concludes the proof of Claim \ref{ClaimBaseCase}.
\end{proof}
\begin{proof}[Proof of Claim \ref{ClaimInductiveStep}]
\renewcommand{\qedsymbol}{}
Here, we essentially aim to follow the idea of the proof of Claim \ref{ClaimBaseCase}, but working inductively to obtain a good inequality for $\left (\ptt - \Delta \right ) F_m$, assuming bounds on $|\nabla^j \Riem| + |\nabla^{j+1}T|$ for all $j \leq m$. Obtaining the evolution equations for $|\nabla^m\Riem|$ and $|\nabla^{m+1}T|$ via the same methods as in the proof of Claim \ref{ClaimBaseCase} is possible (and has been done in the $G_2$ setting for the Laplacian flow in \cite{Lotay-Wei} and the $G_2$ Ricci-Harmonic flow in \cite{DwivediRHF}), but the argument is very long and contains no new ideas compared to the base case. So, we present a different argument, based on \cite[Theorem 2.1]{Chen} for the case of reasonable flows of $G_2$-structures. Once we have the evolution equations for $|\nabla^m\Riem|$ and $|\nabla^{m+1}T|$, the structure of the proof will closely follow that of Claim \ref{ClaimBaseCase}.

We first use this method to derive the differential inequality for $|\nabla \Riem|^2 +|\nabla^2T|^2$, where we can compare with the more explicit method we used in the proof of Claim \ref{ClaimBaseCase}, before proceeding to the case of $|\nabla^m \Riem|^2 +|\nabla^{m+1}T|^2$.
Following \cite{Chen}, we introduce the following notation as a bookkeeping device to help obtain the required differential inequalities. Let the degree of $\Riem$ be $2$, and the degrees of $T$ and $\nabla$ both be $1$. We use this to compute the degrees of polynomials in these terms, for instance the degree of $\nabla \Riem * \nabla T$ is $1+2+1+1 = 5$, and the degree of $\nabla \Riem + \nabla T$ is $\max\{3,2\} = 3$. The degree of a given term in this notation can be thought of as an upper bound on the number of covariant derivatives of the \Sp-structure $\Phi$ that term contains.

In this sense, the degree of $(\ptt-\Delta)\Riem$ is $4$, and it contains no $\nabla^2\Riem $ or $\nabla^3T$ terms (cf. \eqref{eqddtRiem}) and the degree of $(\ptt-\Delta)T$ is $2$ but it contains no $\nabla\Riem$ or $\nabla^2T$ terms (cf.  \eqref{eqReasonableTorsion}).

Now, using that
\[\ptt \nabla T = \nabla \ptt T + T* \nabla \ptt g\] 
and
\[
\nabla \Delta T = \Delta \nabla T + \Riem * \nabla T + \nabla \Riem * T,
\]
we see that the degree of $(\ptt - \Delta)\nabla T$ is $4$ but it contains no $\nabla^2 \Riem$ or $\nabla^3 T$ term (cf. \eqref{eqddtnablaT}). Using all of this, we consider the terms $(\ptt -\Delta)|\Riem|^2+ 2|\nabla\Riem|^2$, $(\ptt -\Delta)|T|^4$, and $(\ptt -\Delta)|\nabla T|^2+ 2|\nabla^2T|^2$. All three of these terms have degree $6$. A given monomial  of degree $6$ falls into one of the following four classes:
\begin{enumerate}
    \item Firstly, it may contain a factor of degree at least $4$, such as $\nabla^2 \Riem$.
    \item Secondly, it may be the product of exactly two terms of degree $3$, such as $\nabla \Riem * \nabla^2T$.
    \item Thirdly, it may consist of a factor of degree $3$ together with several factors of degree less than $3$, such as $\nabla\Riem * \nabla T *T$. 
    \item Finally, it may consist only of factors of degree less than $3$, such as $\nabla T * \Riem *T*T$.
\end{enumerate}

Using the discussion above and the fact that $\Delta|A|^2 = 2 \langle A, \Delta A \rangle + 2 |\nabla A |^2$ for any tensor $A$, we see that the first two of the four cases above do not occur in any of the three terms
$(\ptt -\Delta)|\Riem|^2+ 2|\nabla\Riem|^2$, $(\ptt -\Delta)|T|^4$, and $(\ptt -\Delta)|\nabla T|^2+ 2|\nabla^2T|^2$ (cf. Proof of Proposition \ref{PropDoublingTime}). Applying Young's inequality to terms of the third type allow us to bound such terms by $\varepsilon$ times the degree $3$ factor plus $C_\varepsilon$ times a monomial in terms of the fourth type. Finally, terms of the fourth type can all be bounded by some constant times $(|\Riem|^2 + |\nabla T|^2 + |T|^4 +1)^{3/2}$. Thus, every term of $(\ptt -\Delta)|\Riem|^2+ 2|\nabla\Riem|^2$, $(\ptt -\Delta)|T|^4$, and $(\ptt -\Delta)|\nabla T|^2+ 2|\nabla^2T|^2$ can be bounded by 
\[
\varepsilon(|\nabla \Riem|^2 + |\nabla^2 T|^2) + C_\varepsilon(|\Riem|^2 + |\nabla T|^2 + |T|^4 +1)^{3/2},
\]
for any $\varepsilon>0$. Choosing $\varepsilon =1 $ gives
\begin{align}
    \left (\ptt  - \Delta\right )(|\Riem|^2 + |\nabla T|^2 + |T|^4 +1) \leq &- (|\nabla \Riem|^2 + |\nabla^2 T|^2)\nonumber\\
    &+ C(|\Riem|^2 + |\nabla T|^2 + |T|^4 +1)^{3/2},
\end{align}
reproducing \eqref{eqLambdaTilde}. So far, we have just used this idea to obtain differential inequalities which we had already obtained via direct methods. The point is that this less explicit method generalises nicely by induction to obtain differential inequalities for the quantities $|\nabla^m\Riem|$ and $|\nabla^{m+1}T|$ for arbitrary $m$.

Repeating the argument above, we see that the degrees of $(\ptt - \Delta)\nabla^m\Riem$ and  $(\ptt - \Delta)\nabla^{m+1}T$ are both degree $m+4$, and neither of them contains a $\nabla^{k+2}\Riem$ or $\nabla^{k+3}T$ term. 
All remaining terms are products of lower-order factors, which therefore can be estimated by repeated applications of Young's inequality. This yields a bound by
\[
C(m)\Bigg(\sum_{j=0}^m\left(|\nabla^j\Riem|^{\frac{2(m+3)}{j+2}}+|\nabla^{j+1}T|^{\frac{2(m+3)}{j+2}}\right)
+|T|^{2(m+3)}+1\Bigg),
\]
and so we have that
\begin{align}\label{eqptt-DeltaMthcase}
    \left(\ptt - \Delta \right)(|\nabla^m \Riem|^2 + |\nabla^{m+1}T|^2) \leq &-(|\nabla^{m+1}\Riem|^2+|\nabla^{m+2}T|^2)\nonumber\\
    &+C(m) \Bigg(\sum_{j=0}^m\left(|\nabla^j\Riem|^{\frac{2(m+3)}{j+2}}+|\nabla^{j+1}T|^{\frac{2(m+3)}{j+2}}\right)\nonumber\\
    &\quad \quad \quad \quad+|T|^{2(m+3)}+1\Bigg).
\end{align}
Using this, we derive the required differential inequality for 
\[F_m = t^{m+1}(\mu_m+t^m(|\nabla^m\Riem|^2+|\nabla^{m+1}T|^2))(|\nabla^{m+1}\Riem|^2+|\nabla^{m+2}T|^2),\] 
following a similar method to the proof of Claim \ref{ClaimBaseCase}. 
To simplify notation, we define:
\begin{align}
    U_m &= |\nabla^m\Riem|^2+|\nabla^{m+1}T|^2,\\
    V_m &= |\nabla^{m+1}\Riem|^2+|\nabla^{m+2}T|^2,\\
    A_m & = \mu_m + t^mU_m,\\
    B_m& = t^{m+1}V_m,
\end{align}
so that $F_m = A_mB_m$.
Then,
\begin{equation}\label{eqFmInintial}
\left(\ptt -\Delta \right)F_m = \left[\left(\ptt-\Delta \right)A_m\right]B_m + A_m\left[\left(\ptt-\Delta \right)B_m\right] - 2 \nabla A_m \cdot \nabla B_m.
\end{equation}
We deal with each term one-at-a-time. Firstly, 
\begin{align}
    \left(\ptt-\Delta \right)A_m &= \left(\ptt-\Delta \right)(\mu_m+t^mU_m)\\
    &=mt^{m-1}U_m + t^m\left(\ptt-\Delta \right)U_m, \label{eqAmTerm}
\end{align}
so for this term it suffices to control $U_m$. Using \eqref{eqptt-DeltaMthcase}, we have that
\begin{align}
  \left(\ptt-\Delta \right)U_m  \leq &-(|\nabla^{m+1}\Riem|^2+|\nabla^{m+2}T|^2)\nonumber \\
    &+C(k) \Bigg(\sum_{j=0}^m\left(|\nabla^j\Riem|^{\frac{2(m+3)}{j+2}}+|\nabla^{j+1}T|^{\frac{2(m+3)}{j+2}}\right) \nonumber\\
    &\quad \quad \quad \quad+|T|^{2(m+3)}+1\Bigg).
\end{align}
Now, by the inductive hypothesis, we have that
\[
|\nabla^j\Riem| + |\nabla^{j+1}T| \leq C(K,r,j)t^{-j/2} \text{ in }B_{g(0)}(p,r/2^j) \times (0,1/K],
\]
for all $j =0,1,\cdots,m$, so every term inside the sum is bounded. More precisely, we have that
\[
|\nabla^j\Riem|^{\frac{2(m+3)}{j+2}} \leq \left(C(K,r,j)t^{-j/2}\right)^{\frac{2(m+3)}{j+2}} = C(K,r,j,m)t^{\frac{-j(m+3)}{j+2}},
\]
and likewise for $|\nabla^{j+1}T|^{\frac{2(m+3)}{j+2}}$. Note that, for all $j=0,\cdots ,m$, we have that 
\[
\frac{j(m+3)}{j+2} \leq m+1.
\]
so, recalling that $0\leq t\leq 1/K$, it holds that
\[
C(K,r,j,m)t^{\frac{-j(m+3)}{j+2}} \leq C(K,r,j,m)t^{-(m+1)} \text{ for all } j\leq m.
\]
The terms $|T|^{2(m+3)}$ and $1$ are also bounded, so we obtain
\begin{align}
    \left(\ptt-\Delta \right)U_m  \leq &-(|\nabla^{m+1}\Riem|^2+|\nabla^{m+2}T|^2) + \frac{C}{t^m}\\
    &=-V_m + \frac{C}{t^m},
\end{align}
for some constant $C$ and for all $0\leq t \leq \min\{1,1/K\}$. Plugging this into \eqref{eqAmTerm}, we have
\begin{align}
    \left(\ptt-\Delta \right)A_m &\leq mt^{m-1}U_m -t^mV_m +\frac{C}{t}\\
    &\leq -t^mV_m +\frac{C}{t},\label{eqFinalAmInequality}
\end{align}
because $U_m \leq Ct^{-m}$, by assumption, so $mt^{m-1}U_m \leq \frac{C}{t}$ for some $C$.

We now deal with $\left(\ptt-\Delta \right)B_m$, in a very similar way. We have that
\begin{equation}\label{eqBm}
  \left(\ptt-\Delta \right)B_m = (m+1)t^mV_m + t^{m+1}\left(\ptt-\Delta \right)V_m,  
\end{equation}
so we really need to control $\left(\ptt-\Delta \right)V_m$. Recalling that $V_m = |\nabla^{m+1}\Riem|^2+|\nabla^{m+2}T|^2$, the $(m+1)th$ case of \eqref{eqptt-DeltaMthcase} gives
\begin{align}
  \left(\ptt-\Delta \right)V_m  \leq &-(|\nabla^{m+2}\Riem|^2+|\nabla^{m+3}T|^2)\nonumber\\
    &+C(k) \Bigg(\sum_{j=0}^{m+1}\left(|\nabla^j\Riem|^{\frac{2(m+4)}{j+2}}+|\nabla^{j+1}T|^{\frac{2(m+4)}{j+2}}\right)\nonumber\\
    &\quad \quad \quad \quad+|T|^{2(m+4)}+1\Bigg).
\end{align}
For the terms inside the sum, we have two cases. If $j\leq m$, we can proceed as we did for $U_m$ above, using that $\frac{j(m+4)}{j+2} \leq m+2$ to bound everything by $Ct^{-(m+2)}$ for some constant $C$. More care is needed in the case $j=m+1$, since we do not have a bound on $|\nabla^{m+1}\Riem|$ or $|\nabla^{m+2}T|$.

The contribution of the summand when $j=m+1$ is 
\[
|\nabla^{m+1}\Riem|^{\frac{2(m+4)}{m+3}}+|\nabla^{m+2}T|^{\frac{2(m+4)}{m+3}} \leq CV_m^{\frac{m+4}{m+3}}.
\]
Note that $1<\frac{m+4}{m+3}<2$. We write $q=\frac{m+4}{m+3}$.
Overall, this gives that
\begin{equation}\label{eqddtVm}
    \left(\ptt-\Delta \right)V_m  \leq -W_m + CV_m^q +\frac{C}{t^{m+2}},
\end{equation}
where $W_m = |\nabla^{m+2}\Riem|^2+|\nabla^{m+3}T|^2)$, and plugging this into \eqref{eqBm} gives
\begin{equation}
    \left(\ptt-\Delta \right)B_m \leq (m+1)t^mV_m - t^{m+1}W_m + Ct^{m+1}V_m^q + \frac{C}{t}.
\end{equation}
Now, $(m+1)t^mV_m = \frac{C}{t}B_m$. The only suspicious term is $Ct^{m+1}V_m^q$, but we can deal with that as follows.
Since $V_m = \frac{B_m}{t^{m+1}}$, we have that $t^{m+1}V_m^q = B^q_mt^{m+1-(m+1)q}$. Now, $m+1-(m+1)q = \frac{-(m+1)}{m+3} >-1$. So,
\[
t^{m+1}V_m^q = t^{\frac{-(m+1)}{m+3}}B^q_m < \frac{1}{t}B_m^q,
\]
since $t\leq 1$. Now, since $q<2$, for any $\delta>0$ there exists a constant $C_\delta$ such that \[B_m^q \leq \delta B_m^2 +C_\delta. \]
Combining all of this, we obtain
\begin{equation}\label{eqBmFinalInequality}
    \left(\ptt-\Delta \right)B_m \leq -t^{m+1}W_m +\frac{C}{t}B_m + \frac{\delta}{t}B_m^2 + \frac{C_{\delta}}{t}.
\end{equation}
Looking back to \eqref{eqFmInintial}, it remains only to estimate the final term: $\nabla A_m \cdot \nabla B_m$. Firstly, recall that $A_m = \mu_m + t^mU_m,$ so $\nabla A_m = t^m\nabla U_m$.
Now, using that 
\[
\nabla|S|^2 = 2\langle \nabla S,S \rangle \text{ and so } |\nabla |S|^2| \leq 2|S||\nabla S|,
\]
it holds that
\begin{align}
    |\nabla U_m| &\leq C(|\nabla^m \Riem||\nabla^{m+1}\Riem| + |\nabla^{m+1}T||\nabla^{m+2}T|)\\
    & \leq C\sqrt{U_m}\sqrt{V_m}.\label{eqnablaUm}
\end{align}
Likewise, $\nabla B_m = t^{m+1}\nabla V_m$ and
\begin{equation}\label{eqnablaVm}
    |\nabla V_m| \leq C\sqrt{V_m}\sqrt{W_m},
\end{equation}
where $W_m = |\nabla^{m+2}\Riem|^2+|\nabla^{m+3}T|^2$. Combining \eqref{eqnablaUm} and \eqref{eqnablaVm}, we have
\begin{align}
    |\nabla A_m \cdot \nabla B_m| &\leq Ct^{2m+1}\sqrt{U_m}V_m\sqrt{W_m}\\
    &\leq \frac{C}{\sqrt{t}}\sqrt{t^mU_m}B_m\sqrt{t^{m+1}W_m}\\
    &\leq \frac{C}{\sqrt{t}}B_m\sqrt{t^{m+1}W_m},\label{eqFinalLine}
\end{align}
where we have used that $t^mU_m \leq C$, by the inductive hypothesis. Applying Young's inequality to \eqref{eqFinalLine} above results in
\begin{align}
    |\nabla A_m \cdot \nabla B_m | &\leq \frac{\varepsilon B_m^2}{t} + C_{\varepsilon}(t^{m+1}W_m)\\
    &\leq \frac{\varepsilon B_m^2}{t} + \frac{C_{\varepsilon} A_m t^{m+1}W_m}{\mu_m}, \label{eqFinalNablaANalbaB}
\end{align}
for any $\varepsilon >0$, since $A_m \geq \mu_m$. Now, we substitute \eqref{eqFinalAmInequality}, \eqref{eqBmFinalInequality}, and \eqref{eqFinalNablaANalbaB} into \eqref{eqFmInintial}, resulting in
\begin{align}
   \left(\ptt -\Delta \right)F_m &\leq \left[ -t^mV_m +\frac{C}{t}\right]B_m + A_m \left[ -t^{m+1}W_m +\frac{C}{t}B_m + \frac{\delta}{t}B_m^2 + \frac{C_{\delta}}{t}\right] + \frac{\varepsilon B_m^2}{t} \nonumber\\
   & \quad+ \frac{C_{\varepsilon} A_m t^{m+1}W_m}{\mu_m}\\
   &\leq \frac{-B_m^2}{t} + \frac{C}{t}B_m -t^{m+1}A_mW_m + \frac{C}{t}F_m + \frac{\delta}{t}A_mB_m^2 + \frac{C_{\delta}}{t}A_m \nonumber \\
   & \quad + \frac{\varepsilon B_m^2}{t} + \frac{C_{\varepsilon} A_m t^{m+1}W_m}{\mu_m}\\
   & \leq -(1-\varepsilon)\frac{B_m^2}{t} - \left( 1- \frac{C_\varepsilon}{\mu_m} \right)A_m t^{m+1}W_m + \frac{\delta A_mB_m^2}{t}\nonumber \\
   & \quad + \frac{CB_m}{t} + \frac{CF_m}{t} + \frac{C_\delta A_m}{t},
\end{align}
where so far we have just expanded brackets and collected like terms. Now, define a constant $D_m$ depending on $m$ such that 
\[
\mu_m \leq A_m \leq \mu_m +C_m =D_m,
\]
noting that $D_m$ exists by the inductive hypothesis.
Thus,
\begin{align}
    \left(\ptt -\Delta \right)F_m &\leq -(1-\varepsilon -\delta D_m)\frac{B_m^2}{t} - \left ( 1-\frac{C_\varepsilon}{\mu_m}\right )A_mt^{m+1}W_m \nonumber \\
    &\quad +\frac{C}{t}B_m +\frac{C}{t}F_m + \frac{C_\delta}{t}.\label{eqFmSemiFinal}   
\end{align}
Recall that $D_m$ is fixed and the above holds for any $\varepsilon,\delta >0$, and that we have not yet chosen $\mu_m$. We now carefully choose $\varepsilon,\delta$ and $\mu_m$ to obtain the desired conclusion:
\begin{itemize}
    \item First, choose $\varepsilon < \frac{1}{4}$.
    \item Then, choose $\delta$ so that $\delta D_m < \frac{1}{4}$ and  choose $\mu_m$ large enough so that $\frac{C_\varepsilon}{\mu_m}< \frac{1}{2}$.
\end{itemize}
Plugging these three coefficients into \eqref{eqFmSemiFinal} gives
\begin{equation}
    \left(\ptt -\Delta \right)F_m \leq -\frac{1}{2}\frac{B_m^2}{t} - \frac{1}{2}A_mt^{m+1}W_m + \frac{C}{t}(B_m +F_m +1).
\end{equation}
Discarding the negative middle term and using that 
\[
\frac{F_m}{D_m} \leq B_m = \frac{F_m}{A_m} \leq \frac{F_m}{\mu_m}
\]
results in
\begin{equation}
    \left(\ptt -\Delta \right)F_m \leq -\frac{1}{2D_m^2}\frac{F_m^2}{t} + \frac{C}{t}F_m + \frac{C}{t} .
\end{equation}
A final application of Young's inequality to absorb the linear $F_m$ term yields the desired inequality:
\begin{equation}
    \left(\ptt -\Delta \right)F_m \leq -\frac{C}{t}F_m^2  + \frac{C}{t},
\end{equation}
where $C>0$, which concludes the proof of Claim \ref{ClaimInductiveStep}, and therefore the proof of Theorem \ref{ThmShiType}.
\end{proof}
\end{proof}
Now that we have a local derivative estimate, we can obtain a global estimate on any compact manifold. We state and prove the precise statement now.
\begin{thm}
Let $K>0$. Let $M$ be a compact $8$-manifold and $\Phi(t)$ a solution to a reasonable flow of \Sp-structures for $t\in [0,1/K]$ on $M$. 
Suppose $\Lambda(x,t) \leq K$ for all $(x,t) \in M \times [0,1/K]$
For all $m\in \mathbb{N}$, there exists a constant $C_k$ such that
\[
    |\nabla^m\Riem| + |\nabla^{m+1}T| \leq C(K,m)t^{-m/2}
    \]
    for all $x \in M$, $t \in [0,1/K]$.
\end{thm}
\begin{proof}
    Fix $r_0 >0$, and cover $M$ with balls of radius $r_0/2$:
    \[
    M = \bigcup_{p_i\in M}B_{g(0)}(p_i,r_0/2),
    \]
    Since $M$ is compact, we can take a finite subcover 
    \[
    M = \bigcup_{i=1}^NB_{g(0)}(p_i,r_0/2),
    \]
    for some $N$ points $p_i \in M$ and $N<\infty$.
For each $i = 1, \cdots ,N$, apply Theorem \ref{ThmShiType} to the ball $B_{g(0)}(p_i,r_0)$, obtaining the conclusion that, for all $m\in \mathbb{N}$, there exists a constant $C(K,m,r,i)$ such that
    \[
    |\nabla^m\Riem(x,t)| + |\nabla^{m+1}T(x,t)| \leq C(K,m,i)t^{-m/2}
    \]
    for all $x \in B_{g(0)}(p_i,r_0/2)$, $t \in [0,1/K]$. Then, taking $C(K,m) = \max_{i}C(K,m,i)$ gives the result.
\end{proof}
\section{Long time existence}\label{SectionLongTime}
Given any initial \Sp-structure $\Phi(0)$ evolving according to any reasonable flow of \Sp-structures, short time existence guarantees a solution $\Phi(t)$ on a \emph{maximal time interval} $[0,T_0)$, where \emph{maximal} is taken to mean that either $T_0 = \infty$, or $T_0 < \infty$ and there do not exist $\varepsilon>0$ and a solution $\tilde{\Phi}(t)$ on the time interval $[0,T_0+\varepsilon)$, such that $\tilde\Phi(0) = \Phi(0)$. This $T_0$ is called the \emph{singular time}, and the following theorem shows that $\Lambda$ must blow up at a singular time of any reasonable flow.

\begin{thm}\label{ThmLongTime}
    Let $\Phi(t)$ be a solution to a reasonable flow of \Sp-structures on a compact manifold $M$ on a maximal time interval $[0,T_0)$, with $T_0<\infty$, and let $\Lambda(t) = \sup_{x \in M}\Lambda(x,t)$, where $\Lambda(x,t)$ is as defined in \eqref{eqLambdaDef}.
    Then, 
    \begin{equation}\label{eqLambdaBlowsup}
       \lim_{t\nearrow T_0} \Lambda(t) = \infty, 
    \end{equation}
    and we have the following lower bound on the blow-up rate:
    \begin{equation}\label{eqLambdaLowerBound}
       \Lambda(t) \geq \frac{C}{T_0-t}, 
    \end{equation}
    for some constant $C>0$.
\end{thm}
\begin{proof}
    We follow the proof of \cite[Theorem 5.1]{Lotay-Wei} in the $G_2$-setting, pointing out where our proof necessarily differs due to the different flows under consideration.
    We first show that $\lim_{t\nearrow T_0}\Lambda(t) = \infty$, by contradiction. We shall assume this is not the case, and show that we can then obtain a limiting \Sp-structure $\Phi(T_0)$, and show that $\Phi(t)$ converges smoothly to $\Phi(T_0)$, which contradicts the maximality of $T_0$, as we continue the flow past $T_0$. After that, we shall conclude the lower bound on the blow-up rate via an application of the maximum principle.

    Suppose $\Phi(t)$ is a solution to a reasonable flow of \Sp-structures on a maximal time interval $[0,T_0)$, and suppose for the sake of contradiction that \eqref{eqLambdaBlowsup} does not hold. Then, there exists $K>0$ such that 
    \begin{equation}\label{eqAssumptionBoundedLambda}
        \sup_{M \times [0,T_0)} \Lambda(x,t) = \sup_{M \times [0,T_0)}(|\nabla T|^2_{g(t)} + |\Riem|^2_{g(t)}+ |T|^4_{g(t)})^{1/2}\leq K.
    \end{equation}
    In particular, this gives the following uniform curvature and torsion bounds on $M \times [0,T_0)$:
    \begin{align}
        \sup_{M \times [0,T_0)}|\Riem|_{g(t)} &\leq C(K),\label{eqRiemBound} \\ 
         \sup_{M \times [0,T_0)}|T|_{g(t)}&\leq C(K), \label{eqTbound}
    \end{align}
for some universal constant $C(K)$ depending only on $K$. Note here that the authors of \cite{Lotay-Wei} do not require the bound on $|T|$ \eqref{eqTbound}, since in the $G_2$ setting $|R|=|T|^2$ so the torsion bound is implied by the curvature bound. This is one reason for the difference in our quantity $\Lambda$ compared with that of \cite{Lotay-Wei}.

Now, using \eqref{eqRiemBound}, \eqref{eqTbound} and the evolution equation for the metric \eqref{eqReasonableMetric}, we have that
\begin{equation}\label{eqimpliesmetricsunifequiv}
    \sup_{M \times [0,T_0)}\left|\ptt g(t) \right |_{g(t)} = \sup_{M \times [0,T_0)}\left|\Riem + T*T + L(T) +C \right |_{g(t)} \leq C(K)
\end{equation}
which implies that all of the metrics $g(t)$ for $0 \leq t < T_0$ are uniformly equivalent (cf. \cite[Lemma 14.2]{Hamilton3Manifolds}). Also, because of the requirement for coefficients to be bounded by $\Lambda$ in the definition of reasonable flows \eqref{eqReasonableMetric}, \eqref{eqreasonableX}, we have that
\begin{equation} \label{eqBoundedPhiDerivatives}
    \sup_{M \times [0,T_0)}\left|\ptt \Phi(t) \right |_{g(t)} \leq C(K),
\end{equation}
for some universal constant $C$. We now aim to use these ingredients to obtain a limiting \Sp-structure $\Phi(T_0)$ and show that $\Phi(t)$ converges smoothly to $\Phi(T_0)$. With this in mind, we fix a background metric $\bar{g} = g(0)$. Because of \eqref{eqBoundedPhiDerivatives} and the uniform equivalence of the metrics $g(0)$ and $g(t)$, we have that 
\begin{equation}
    \left |\ptt \Phi(t) \right |_{\bar{g}} \leq  C\left |\ptt \Phi(t) \right |_{g(t)} \leq CK.
\end{equation}
So, for any $0< t_1 < t_2 <T_0$, we have 
\begin{equation}
    |\Phi(t_2) - \Phi(t_1)|_{\bar{g}} \leq \bigintssss_{t_1}^{t_2}\left|\ptt \Phi \right |_{\bar{g}} \d t\leq CK(t_2-t_1),
\end{equation}
and so $\Phi(t)$ converges to a $4$-form $\Phi(T_0)$ as $t\nearrow T_0$. We now argue that this limiting $4$-form is actually a \Sp-structure. Recall that, by \eqref{EqPositivity4form}, for all $0 \leq t < T_0$, we have that
\begin{equation}\label{eqPosDef}
    (v \lrcorner w \lrcorner\Phi(t))\wedge(v \lrcorner w \lrcorner\Phi(t))\wedge \Phi(t) = 6|v\wedge w|_{g(t)}^2\mathrm{vol}_{\Phi(t)},
\end{equation} 
and we have that the metric $g(t)$ and volume form $\mathrm{vol}_{\Phi(t)}$ converge to a metric $g(T_0)$ and a volume form $\mathrm{vol}_{\Phi(T_0)}$. So, the right-hand-side of \eqref{eqPosDef} converges to a positive $8$-form-valued function of $v$ and $w$, and so the limiting $4$-form $\Phi(T_0)$ is \emph{non-degenerate}, in the sense of Definition \ref{DefNonDegenerate}. Note that, as discussed in Section \ref{SectionPrelims}, this is not enough to show that $\Phi(T_0)$ is a \Sp-structure, unlike in the case of flows of $G_2$-structures. However, we note that the extra conditions of Theorem \ref{ThmCharacterisation} are continuous algebraic conditions on $\Phi(t)$, and so are preserved in the limit as $t \to T_0$. This, together with the preservation of non-degeneracy discussed above, gives a continuous extension $\Phi(t)$ to the time interval $[0,T_0]$, but does not yet provide a contradiction.

To obtain a contradiction, we show that this convergence is actually smooth.
In doing so, we will need the following two claims, whose proofs we postpone until the end of the proof of this theorem.
\begin{claim}\label{ClaimSupMetricDerivative}
    For all $m \in \mathbb{N}$, there exist constants $C_m$ such that
    \begin{equation}
       \sup_{M \times [0,T_0)} \left|\overline{\nabla}^m g(t)\right |_{\bar{g}} \leq C_m,
    \end{equation}
    where $\overline{\nabla}$ is the Levi-Civita connection induced by the metric $\bar{g}$.
\end{claim}
\begin{claim}\label{ClaimSupSpin7Derivative}
        For all $m \in \mathbb{N}$, there exist constants $C_m$ such that
    \begin{equation}
       \sup_{M \times [0,T_0)} \left|\overline{\nabla}^m \Phi(t)\right |_{\bar{g}} \leq C_m,
    \end{equation}
    where $\overline{\nabla}$ is the Levi-Civita connection induced by the metric $\bar{g}$.
\end{claim}

Now, we have from just before Claim \ref{ClaimSupMetricDerivative} that a continuous limit \Sp-structure $\Phi(T_0)$ exists, and in a fixed local coordinate chart $U$ it satisfies
\begin{equation}\label{eqIntegralPhiEquation}
    \Phi(T_0) = \Phi(t) + \bigintssss_t^{T_0}(A \diamond\Phi)(s) \d s,
\end{equation}
where $A = h+X$, as in Definition \ref{DefReasonable}.
Let $\alpha = (\alpha_1, \cdots, \alpha_r)$ be a multi-index with $|\alpha| = m.$ By Claim \ref{ClaimSupSpin7Derivative} and its proof \eqref{eqNablamAdiamondPhiBounded}, we have that
\begin{equation}
    \frac{\pt^m}{\pt x^\alpha}\Phi \text{ and } \frac{\pt^m}{\pt x^\alpha}(A \diamond \Phi)
\end{equation}
are uniformly bounded on $U \times [0,T_0)$, since we are working on a local chart. So, Equation \eqref{eqIntegralPhiEquation} gives that $\frac{\pt^m}{\pt x^\alpha}\Phi(T_0)$ is bounded, and hence $\Phi(T_0)$ is a smooth \Sp-structure. Moreover, \eqref{eqIntegralPhiEquation} also gives that
\[
\left| \frac{\pt^m}{\pt x^\alpha}\Phi(T_0) - \frac{\pt^m}{\pt x^\alpha}\Phi(t)\right | \leq C(T_0-t),
\]
and so $\Phi(t)$ converges uniformly to $\Phi(T_0)$ in any $C^m$ norm as $t \nearrow T_0$, for any $m \geq2$.

Now, short time existence gives a solution $\bar{\Phi}(t)$ to the reasonable flow, starting from the initial condition $\bar{\Phi}(0) = \Phi(T_0)$, for some time $0\leq t < \varepsilon$. But, since $\Phi(t)$ converges to $\Phi(T_0)$ smoothly as $t \nearrow T_0$, we have that
\begin{equation}\label{eqMaximalityContradictionConstruction}
    \tilde{\Phi}(t) = \begin{cases}
        \Phi(t), &0\leq t<T_0,\\
        \bar{\Phi}(t-T_0), &T_0\leq t < T_0+\varepsilon,
    \end{cases}
\end{equation}
is another solution to the reasonable flow of \Sp-structures starting from the same initial condition $\tilde{\Phi}(0) = \Phi(0)$, and so by uniqueness we have a contradiction to the maximality of $T_0$. Thus, we must have that 
\begin{equation}
    \limsup_{t\nearrow T_0}\Lambda(t) = \infty,
\end{equation}
and it remains to prove that, in fact, $\lim_{t \nearrow T_0} \Lambda(t) = \infty.$ Indeed, if this is not the case, then also $\lim_{t \nearrow T_0} \tilde\Lambda(t) < \infty,$ where $\tilde\Lambda$ is as defined in Equation \eqref{eqlambdatildedef}. So, there exists a constant $K_0$ and a sequence $t_i \nearrow T_0$ such that $\tilde\Lambda(t_i) < K_0$ for all $i$. Then, by the doubling-time estimate for $\tilde\Lambda$ (Theorem \ref{PropDoublingTime}), it holds that
\begin{equation}
    \tilde{\Lambda}(t) \leq 2 \tilde\Lambda(t_i) \leq 2K_0,
\end{equation}
for all $t \in \left[t_i, \min \big\{T_0, t_i + \frac{1}{CK_0}\big\}\right]$. Since $t_i \rightarrow T_0$ and ${CK_0}$ is a positive constant, $T_0 < t_i + \frac{1}{CK_0}$ for all sufficiently large $i$. So, 
\begin{equation}
\min  \bigg\{T_0, t_i + \frac{1}{CK_0}\bigg\} =  T_0 \text{ for all sufficiently large }i,
\end{equation}
and so we have that 
\begin{equation}
    \sup_{M\times[t_i,T_0)} \tilde\Lambda(x,t)\leq K_0,
\end{equation}
and $\Lambda(x,t) \leq \tilde \Lambda(x,t)$ for all $x,t$, so
\begin{equation}
    \sup_{M\times[t_i,T_0)} \Lambda(x,t)\leq K_0,
\end{equation}
for all $i$ sufficiently large. But, we have already shown (above, ending with Equation \eqref{eqMaximalityContradictionConstruction}) that this leads to a contradiction to the maximality of $T_0$, so this proves \eqref{eqLambdaBlowsup}. Finally, we prove the lower bound on the blow-up rate.
Applying the maximum principle to Equation \eqref{eqLambdaTilde} gives 
\begin{equation}
    \frac{\d}{\d t}\tilde \Lambda^2(t) \leq C \tilde \Lambda^3(t),
\end{equation}
which implies
\begin{equation}\label{eqLambdaTildeInverseEvolution}
    \frac{\d}{\d t}\tilde \Lambda^{-1}(t) \geq -\frac{C}{2}.
\end{equation}
Integrating \eqref{eqLambdaTildeInverseEvolution} from $t$ to $t'$ for some $t<t'<T_0$ gives
\begin{equation}
    \tilde{\Lambda}^{-1}(t')- \tilde{\Lambda}^{-1}(t) \geq -\frac{C}{2}(t'-t).
\end{equation}
Taking the limit as $t'\nearrow T_0$ and using that $\lim_{t\nearrow T_0} \Lambda(t) = \infty$ and hence $\lim_{t\nearrow T_0} \Lambda^{-1}(t) = 0$, we obtain
\begin{equation}
    - \tilde{\Lambda}^{-1}(t) \geq -\frac{C}{2}(T_0-t),
\end{equation}
and so
\begin{equation}
    \tilde{\Lambda}(t) \geq \frac{2}{C(T_0-t)}.
\end{equation}
Finally, we have that
\[
\lim_{t \to T_0}\Lambda(t) = \lim_{t \to T_0}\tilde \Lambda(t),
\]
and so this gives the same lower bound on the blow-up rate for $\Lambda$, modulo Claims \ref{ClaimSupMetricDerivative} and \ref{ClaimSupSpin7Derivative}, which we prove now.

\begin{proof}[Proof of Claim \ref{ClaimSupMetricDerivative}]
This is very similar to the case of Ricci flow \cite[Proposition 6.48]{ChowRicci}, whose proof we follow.
Let $S = \overline\nabla - \nabla$ be the difference of the Levi-Civita connections with respect to the metrics $\bar g$ and $g$, respectively. Note that $S$ is therefore a tensor.
We prove, by induction, the following statements:
\begin{equation}\label{eqSupMetricInductive1}
    \sup_{M \times [0,T_0)} \left|\overline{\nabla}^j g(t)\right |_{\bar{g}} \leq C_j, \text{ for  } 0\leq j \leq m,
\end{equation}
and
\begin{equation}\label{eqSupMetricInductive2}
    \sup_{M \times [0,T_0)} \left|\overline{\nabla}^j S(t)\right |_{\bar{g}} \leq C_j', \text{ for  } 0\leq j \leq m-1,
\end{equation}
for some constants $C_j$ and $C_j'$. 
Note that 
\begin{equation}
    \overline{\nabla}g^{-1} = -g^{-1}*\left (\overline{\nabla }g\right )*g^{-1},
\end{equation} so inductive bounds on $\overline{\nabla}^jg$ also imply bounds on $\overline{\nabla}^jg^{-1}$. 
Since the metrics $g(t)$ are all uniformly equivalent to $\bar g$ \eqref{eqimpliesmetricsunifequiv}, the case $m=0$ for $g$ holds immediately, so we begin with the case $m=1$.
Since $\overline \nabla$ is fixed in time, we have that
\begin{equation}
    \ptt S_{ij}^k = -\ptt\Gamma_{ij}^k = -\frac{1}{2}g^{kl}\left (\nabla_i \ptt g_{jl} + \nabla_j \ptt g_{il} -\nabla_l\ptt g_{ij} \right ) ,
\end{equation}
where $\Gamma_{ij}^k$ are the Christoffel symbols of the metric $g$.
So, schematically,
\begin{align}
    \ptt S &= -g^{-1}*\nabla\ptt g\\
    &=-g^{-1}* \nabla(\Ric + T*T + L(T) +C). \label{eqrighthandsideS}
\end{align}
By the assumption \eqref{eqAssumptionBoundedLambda}, $\nabla \Ric$, $\nabla T$ and $T$ are all uniformly bounded, so every term on the right-hand-side of \eqref{eqrighthandsideS} is uniformly bounded on $M\times [0,T_0)$.
Since $M$ is compact and the initial data are smooth, the initial quantity
\begin{equation}
    \sup_M |S(0)|_{\bar g}
\end{equation}
is finite.
Integrating with respect to time $t$, we obtain
\begin{align}
    |S(t)|_{\bar g} &\leq |S(0)|_{\bar g} +\int_0^t\left|\frac{\partial}{\partial s}S(s) \right|_{\bar g}\mathrm{d}s  \\
    &\leq |S(0)|_{\bar g} +C\int_0^t\left|\frac{\partial}{\partial s}S(s) \right|_{ g(s)}\mathrm{d}s \leq |S(0)|_{\bar g}+CT_0 \leq C, \label{eqSBounded}
\end{align}
where the second inequality holds by uniform equivalence of the metrics $\bar g$ and $g(s)$, and the third by the boundedness of the right-hand side of \eqref{eqrighthandsideS}.

Now $\nabla g = 0$ so $\overline \nabla g = (\overline \nabla g - \nabla g) = S*g$. Since $S$ is bounded and $\bar g$ is uniformly equivalent to $g$, we conclude that
\begin{equation}
    |\overline \nabla g|_{\bar g} \leq C,
\end{equation}
proving the base case of Claim \ref{ClaimSupMetricDerivative}.

For the inductive step, we suppose that \eqref{eqSupMetricInductive1} and \eqref{eqSupMetricInductive2} hold, and we first show that
$|\overline \nabla S|^m_{\bar g}$ is bounded, before using that to deduce the claim for $g$.

Again, since $\overline \nabla $ is constant in time, we have that
\begin{equation}\label{eqddtnablamS}
    \ptt \overline \nabla^m S = \overline \nabla^m \ptt S = - \overline  \nabla^m\left( g^{-1}*\nabla(\Ric + T*T +L(T) +C)\right) .
\end{equation}
For any tensor $A$, it holds that
\begin{equation}
    \overline{\nabla}A = \nabla A + S*A.
\end{equation}
Applying this fact iteratively, we obtain, schematically:
\begin{equation}
    \overline{\nabla}^l A
    =
    \sum
    \overline{\nabla}^{a_1}S*\cdots *
    \overline{\nabla}^{a_q}S*
    \nabla^b A,
\end{equation}
where the case \(q=0\) and \(b=l\) gives the leading term
\(\nabla^l A\), while all remaining terms satisfy
\begin{equation}
    a_1+\cdots+a_q+b \leq l-1.
\end{equation}
Applying this to \eqref{eqddtnablamS} and using the inductive bounds on $\overline{\nabla}S$, $\overline{\nabla}g$ and $\overline{\nabla}g^{-1}$ together with our Shi-type estimates (Theorem \ref{ThmShiType}) yields
\begin{equation}
    \left|\ptt \overline{\nabla}^mS \right|_{\bar g} \leq C_m,
\end{equation}
for some constant $C_m$ depending on $m$.
Again, since $M$ is compact and the initial data are smooth,
\begin{equation}
    \sup_M |\overline{\nabla}^m S(0)|_{\bar g}<\infty.
\end{equation}
As before, integrating with respect to time gives
\begin{align}
    |\overline{\nabla}^mS(t)|_{\bar g} &\leq |\overline{\nabla}^mS(0)|_{\bar g} +\int_0^t\left|\frac{\partial}{\partial s}\overline{\nabla}^mS(s) \right|_{\bar g}\mathrm{d}s  \\
    &\leq |\overline{\nabla}^mS(0)|_{\bar g} +C\int_0^t\left|\frac{\partial}{\partial s}\overline{\nabla}^mS(s) \right|_{ g(s)}\mathrm{d}s \leq C_m,\label{eqnablamsbounded}
\end{align}
Finally, we use this to prove the claim. As above, we have that $\overline{\nabla}g = S*g$, so 
\begin{equation}
    \overline{\nabla}^{m+1}g = \overline{\nabla}^m(S*g) = \sum_{i=0}^m\overline{\nabla}^iS*\overline{\nabla}^{m-i}g,
\end{equation}
which is uniformly bounded by applying \eqref{eqnablamsbounded} to the first factor and the inductive assumption to the second factor.
\end{proof}

\begin{proof}[Proof of Claim \ref{ClaimSupSpin7Derivative}]
We begin by proving the claim in the case $m=1$.
Writing our reasonable flow equation as $\ptt \Phi = A\diamond\Phi$, as in Definition \ref{DefReasonable}, we have that
\begin{align}
    \ptt \overline \nabla \Phi & = \overline{\nabla}\ptt\Phi\\
    &=\overline{\nabla}(A\diamond\Phi)\\
    &= \nabla(A\diamond\Phi) + S*(A\diamond\Phi),
\end{align}
since $S =\overline{\nabla}- \nabla$. Now, by \eqref{eqSBounded}, $S$ is bounded and we have that $\nabla(A\diamond \Phi)$ is bounded by the structure of $A = h +X$ (see Equations \eqref{eqReasonableMetric} and \eqref{eqreasonableX}) and the Shi type estimates (Theorem \ref{ThmShiType}). So, we have that
\begin{equation}
    \left| \ptt \overline{\nabla}\Phi \right|_{\bar{g}} \leq C,
\end{equation}
and integrating with respect to time gives
\begin{equation}
    \left|  \overline{\nabla}\Phi(t) \right|_{\bar{g}} \leq \left|  \overline{\nabla}\Phi(0) \right|_{\bar{g}} + \int_{0}^t\left| \frac{\pt}{\pt s} \overline{\nabla}\Phi(s) \right|_{\bar{g}} \leq \left|  \overline{\nabla}\Phi(0) \right|_{\bar{g}} + CT_0 \leq C,
\end{equation}
which is the case $m=1$ of Claim \ref{ClaimSupSpin7Derivative}.

For $m \geq 2$, we have
\begin{align}
    \left| \ptt \overline{\nabla}^m\Phi(t)\right |_{\bar{g}} &= \left| \overline{\nabla}^m \ptt\Phi(t)\right |_{\bar{g}}\\
    &=\left| \overline{\nabla}^m (A\diamond\Phi)(t)\right |_{\bar{g}}\\
    & \leq C \sum_{i=0}^m|S|^i\left|\nabla^{m-i}(A\diamond\Phi)(t) \right| + C\sum_{i=0}^{m-1}|\overline{\nabla}S|^i\left|\nabla^{m-1-i}(A\diamond\Phi)(t) \right|\label{eqNablamAdiamondPhiBounded}
\end{align}
This, together with the Shi-type estimates (Theorem \ref{ThmShiType}), structure of $A\diamond\Phi$ ( \eqref{eqReasonableMetric} and \eqref{eqreasonableX}), and \eqref{eqnablamsbounded} gives,
\begin{equation}
    \left| \ptt \overline{\nabla}^m\Phi(t)\right |_{\bar{g}} \leq C_m,
\end{equation}
for some constant $C_m$, and integrating with respect to time again results in the Claim.

\end{proof}
\end{proof}
Combining Theorem \ref{ThmLongTime} with Proposition \ref{PropDoublingTime} gives the following Corollary.
\begin{cor}\label{CorollaryDoublingLongTime}
    Let $\Phi_0$ be a \Sp-structure on a compact $8$-manifold $M^8$ such that
    \[
    \tilde \Lambda_{\Phi_0}(x) = (|\Riem(x)|^2_{g_0} + |T(x)|^2_{g_0}+ |\nabla T(x)|^2_{g_0} +1)^{1/2} <K
    \]
    on $ M$, for some constant $K$, where $g_0$ is the Riemannian metric induced by $\Phi_0$.

    Then, the unique solution $\Phi(t)$ to a given reasonable flow of \Sp-structures exists, at least for time $t \in [0, \frac{1}{CK}]$, where $C$ is a universal constant as in Proposition \ref{PropDoublingTime}.
\end{cor}
\begin{proof}
    By Proposition \ref{PropDoublingTime}, 
    \[
    \tilde \Lambda(t) \leq 2 \tilde \Lambda(0) \leq 2K
    \]
    for all $t\in [0, \frac{1}{CK}].$ 
    In particular, 
    \[
    \Lambda(t)\leq \tilde \Lambda(t)< \infty,
    \] for all $t \in [0,\frac{1}{CK}]$
    so the contrapositive of Theorem \ref{ThmLongTime} gives that a solution exists for at least $t\in [0,\frac{1}{CK}]$.
\end{proof}
Using Theorem \ref{ThmLongTime}, we can classify the types of singularities for reasonable flows of \Sp-structures, in the same way as for other geometric flows (e.g. \cite[Definition 16.2]{HamiltonSingularities}, \cite[Definition 5.3]{DwivediRHF} ).
\begin{defn}\label{defSingularitytypes}
    Let $\Phi(t)$ be a solution to a reasonable flow of \Sp-structures on a compact manifold $M^8$, on a maximal time interval $[0,T_0)$, for some $T_0 < \infty$. Then, we say that the flow develops a finite-time singularity at $T_0 < \infty$, and we say the solution forms:
    \begin{itemize}
        \item A Type I singularity if $\sup_{t\in [0,T_0)}(T_0-t)\Lambda(t) < \infty$,
        \item A Type IIa singularity if $\sup_{t\in [0,T_0)}(T_0-t)\Lambda(t) = \infty$.
    \end{itemize}
    If a singularity occurs in infinite time, at $T_0 = \infty$, we say the solution forms:
    \begin{itemize}
        \item A Type IIb singularity if $\sup_{t\in [0,\infty)}t\Lambda(t) = \infty$,
        \item A Type III singularity if $\sup_{t\in [0,\infty)}t\Lambda(t) < \infty$.
    \end{itemize}
    
\end{defn}
    \begin{example}\label{exampleRHF}
       In \cite[Example 4.5]{Duthie2025}, a $1$-parameter family $\Phi(t)$ of \Sp-structures is given on $\text{SU}(3)$, and $\Phi(t)$ is shown to be a shrinking soliton for the gradient flow of \Sp-structures. The gradient flow is not a reasonable flow, so we use that example here to construct a soliton solution to the Ricci-Harmonic flow, and consider the type of the singularity formed. By \cite[Appendix A]{Duthie2025}, the bi-invariant \Sp-structure given by \cite[Equation 45]{Duthie2025} satisfies
       \begin{equation}\label{eqExample}
           \Ric = \text{diag}\left (\frac{1}{4} \right ), \quad \text{div}T = 0.
       \end{equation}
       Moreover, for the lower order term $T*T$ in the definition of the Ricci-harmonic flow \eqref{EquationRicciHarmonic}, a calculation with the aid of Maple shows that $(T*T)\diamond\Phi = 0$. So, $A\diamond\Phi = -\Phi$ and so this \Sp-structure gives rise to a Ricci-harmonic soliton. We solve the equation explicitly by writing down the following ansatz:
       \[
       \Phi(t) = f(t)^4\Phi(0).
       \]
       Plugging this into \eqref{EquationRicciHarmonic} gives the following ODE for $f$:
       \[
       4f(t)^3 \frac{\d f(t)}{\d t} = -f(t)^2, \quad f(0)=1.
       \]
       Solving this, we see that $f(t) = \sqrt{1-\frac{1}{2}t}$, and so $\Phi(t) = (1-\frac{1}{2}t)^2\Phi_0$, giving an explicit soliton solution for the Ricci-harmonic flow. 
       
       This induces the family of metrics $g(t) = (1-\frac{1}{2}t)g(0)$, and so we see that $\Lambda(t) = \frac{2}{2-t}\Lambda(0)$, which shows that $\Phi(t)$ is a Type I singularity, as is expected for shrinking solitons. 
    \end{example}

\section{Compactness}\label{SectionCompactness}
In this section, we prove a Cheeger-Gromov-type compactness result for the space of \Sp-structures, as well as a Cheeger-Gromov-Hamilton-type result for compactness of the space of solutions to a given reasonable flow of \Sp-structures. Both of these results use our Shi-type estimates (Theorem \ref{ThmShiType}), and the proofs are based upon \cite[Section 7]{Lotay-Wei} in the case of the $G_2$ Laplacian flow. We shall use these results to discuss parabolic rescalings of reasonable flows, which in turn shall be used to study singularity formation in Section \ref{SectionSingularities}.

\subsection{Compactness for \texorpdfstring{\Sp-structures}{Spin(7)-structures}}
We begin by defining a notion of convergence for manifolds with \Sp-structures.

Let $M_i$ be a sequence of $8$-manifolds, $p_i\in M_i$ and $\Phi_i$ a \Sp-structure on $M_i$ for each $i$, such that the induced metric $g_i$ is complete for all $i$. Let $M$ be an $8$-manifold, $p\in M$ a point and $\Phi$ a \Sp-structure on $M$. We say that the triple $(M_i,\Phi_i,p_i)$ converges to $(M,\Phi,p)$, written 
\[
(M_i,\Phi_i,p_i) \rightarrow (M,\Phi,p) \text{ as } i \rightarrow\infty,
\]
if there exists:
\begin{itemize}
    \item a sequence of compact subsets $\Omega_i \subset M$ exhausting $M$ with $p\in int(\Omega_i)$ for each $i$ and
    \item a sequence of diffeomorphisms $F_i \colon \Omega_i \to F_i(\Omega_i) \subset M_i$ with $F_i(p) = p_i$ and such that 
    \[
    F_i^*\Phi_i \to \Phi \text{ as } i \to \infty
    \]
    in the sense that the $4$-form $F_i^*\Phi_i - \Phi$ and all its covariant derivatives of all orders, taken with respect to any fixed background metric, converge uniformly to zero on every compact subset of $M$.
\end{itemize}
With this, we state our compactness theorem for \Sp-structures.
\begin{thm}\label{ThmSpin7Compact}
    Let $M_i$ be a sequence of smooth $8$-manifolds, and for each $i$ let $p_i \in M_i$ be a point and $\Phi_i$ be a \Sp-structure on $M_i$ inducing a \emph{complete} Riemannian metric $g_i$. Suppose further that
    \begin{equation}\label{eqBoundedSups}
    \sup_i\sup_{x \in M_i}\left(\left|\nabla^{k+1}_{g_i}T_i(x) \right |_{g_i}^2 + \left|\nabla^{k}_{g_i}\Riem_i(x) \right |_{g_i}^2 \right)^{\frac{1}{2}}< \infty \text{ for all } k \geq 1,
    \end{equation}

\begin{equation}\label{eqInitialBoundedSups}
\sup_i\sup_{x \in M_i}\left(\left|\nabla_{g_i}T_i(x) \right |_{g_i}^2 + \left|\Riem_i(x) \right |_{g_i}^2 + \left|T_i(x) \right |_{g_i}^4 \right)^{\frac{1}{2}}< \infty,
\end{equation}
and 
\begin{equation}\label{eqBoundedInj}
    \inf_i \inj(M_i,g_i,p_i)>0,
\end{equation}
where $T_i$ is the torsion tensor induced by $\Phi_i$, $\Riem_i$ is the curvature tensor of $g_i$ and $\inj(M_i,g_i,p_i)$ is the injectivity radius of $(M_i,g_i)$ at the point $p_i$.

Then, there exists an $8$-manifold $M$ with \Sp-structure $\Phi$ and a point $p \in M$ such that, after passing to a subsequence,
\[
(M_i,\Phi_i,p_i) \to (M,\Phi,p) \text{ as } i \to \infty.
\]
\end{thm}
\begin{proof}
We begin by outlining the strategy of this proof. We first work on a fixed compact subset $\Omega_i$ in the exhaustion defined above, and argue that there exists a limiting \Sp-structure $\Phi_{i,\infty}$ on $\Omega_i$. A diagonal argument then produces compatible local limits on an exhaustion of the limit manifold, which patch together to give a global candidate limiting \Sp-structure $\Phi$ on all of $M$. Finally, we verify that $(M_i,\Phi_i,p_i) \to (M,\Phi,p)$ in the sense defined at the start of this section. At several points in the proof, we will pass to a subsequence, but shall continue to use the index $i$ to avoid an overload of notation.

 Firstly, \eqref{eqBoundedSups} and \eqref{eqInitialBoundedSups} imply, in particular, bounds on the Riemann curvature tensor and all of its covariant derivatives, so the Cheeger-Gromov compactness theorem for complete pointed Riemannian manifolds \cite[Theorem 2.3]{HamiltonCompactness} applies (using our lower bound on the injectivity radius \eqref{eqBoundedInj}).
 Thus, there exists a complete Riemannian $8$-manifold $(M,g)$ and a point $p\in M$ such that
 \begin{equation}\label{eqCheegerGromovConvergence}
     (M_i,g_i,p_i) \to (M,g,p) \text{ as } i \to \infty,
 \end{equation}
 after passing to a subsequence.

 We need to show that the sequence of \Sp-structures $\Phi_i$ has a subsequence converging to a limiting $4$-form, and that this $4$-form is a \Sp-structure inducing the limiting metric $g$, and finally that the triple $(M_i,\Phi_i,p_i)$ converges to $(M,\Phi,p)$, in the sense discussed at the start of this section.
 
 We first show there exists a limiting $4$-form. By definition of the convergence in \eqref{eqCheegerGromovConvergence},  we have nested compact sets $\Omega_i \subset M$ exhausting $M$ with $p \in int(\Omega_i)$ for all $i$, and diffeomorphisms $F_i : \Omega_i \to F_i(\Omega_i) \subset M_i$ with $F_i(p) = p_i$ such that $F_i^*g_i$ converge smoothly to $g$ as $i\to \infty$ on any compact subset of $M$.

 Fix $i$ sufficiently large. For any $j\geq0$, we have that $\Omega_i \subset \Omega_{i+j}$ and we have a diffeomorphism $F_{i+j}: \Omega_{i+j} \to F_{i+j}(\Omega_{i+j})\subset M_{i+j} $. Define a restricted diffeomorphism
 \[
 F_{i,j} = F_{i+j}|_{\Omega_i}:\Omega_i \to F_{i+j}(\Omega_i) \subset M_{i+j} \text{ for all } j \geq 0.
 \]
 The definition of the convergence in \eqref{eqCheegerGromovConvergence} implies that the sequence $\{g_{i,j} = F^*_{i,j}g_{i+j} \}_{j=0}^{\infty}$ of Riemannian metrics on $\Omega_i$ converges to a limit $g_{i,\infty} = g$ on $\Omega_i$ as $j \to \infty$, and by Cheeger-Gromov, the limit $g$ is a Riemannian metric on $\Omega_i$.
 Now, let $\nabla$ be the Levi-Civita connection of $g$, and $\nabla_{g_{i,j}}$ that of $g_{i,j}$. Let $h = g-g_{i,j}$ and $S = \nabla - \nabla_{g_{i,j}}$ be the differences of the metrics and their connections, respectively. Note that $S$ is a tensor, and locally we have that 

 \begin{equation}\label{eqS}
     S^c_{ab} = \frac{1}{2}(g_{i,j})^{cd}(\nabla_ah_{bd} + \nabla_b h_{ad} - \nabla_d h_{ab}). 
 \end{equation}
We now show that $S$ and all its covariant derivatives are uniformly bounded, which we will use to show the same is true for $\Phi_{i,j}$, and we can therefore apply Arzelà-Ascoli to extract a convergent subsequence.

Since $g_{i,j}\to g$ smoothly on $\Omega_i$ as $j\to \infty$ (as discussed above), we have that $g_{i,j}$ is equivalent to $g$ for sufficiently large $j$. Also, $|\nabla^kh|_g$ tends to $0$ as $j \to \infty$ for all $k \geq 0$. So, for all $j$ sufficiently large, we have from \eqref{eqS} that $S$ is uniformly bounded with respect to $g$.

For the covariant derivatives of $S$, we have that
 \begin{align}
     \nabla^{k}S_{ab}^c &= \frac{1}{2}\sum_{l=1}^k\nabla^{(k+1-l)}(g_{i,j})^{cd}\left(\nabla^{l}\nabla_ah_{bd} + \nabla^{l}\nabla_bh_{ad}  - \nabla^{l}\nabla_dh_{ab} \right) \\
     &= -\frac{1}{2}\sum_{l=1}^k\nabla^{(k+1-l)}(g^{cd}-(g_{i,j})^{cd})\left(\nabla^{l}\nabla_ah_{bd} + \nabla^{l}\nabla_bh_{ad}  - \nabla^{l}\nabla_dh_{ab} \right)\label{eqNablakS},
 \end{align}
where the second line follows from the fact that $\nabla g = 0$. Since $g_{i,j}\to g$ smoothly on $\Omega_i$, the inverse metrics $(g_{i,j})^{-1}$ also converge smoothly to $g^{-1}$. Hence $h=g-g_{i,j}$, $g^{-1}-(g_{i,j})^{-1}$, and all their covariant derivatives are uniformly bounded. Equation \eqref{eqNablakS} therefore gives constants $c_k$, independent of $j$, such that $|\nabla^kS|_g\leq c_k$ on $\Omega_i$ (recalling that $S$ depends on $j$).

Using each diffeomorphism $F_{i,j}$ we define a \Sp-structure on $\Omega_i$ by pulling back the \Sp-structure $\Phi_{i+j}$ on $M_{i+j}:$ \[\Phi_{i,j} = F_{i,j}^*\Phi_{i+j}.\]
We now use the uniform bounds on $\nabla^kS$ to obtain uniform bounds on $\Phi_{i,j}$ and $\nabla^k \Phi_{i,j}$, for each $k\geq0$, with respect to the Riemannian metric $g$.
Firstly, since $g$ and $g_{i,j}$ are equivalent for all $j$ sufficiently large, it holds that
\begin{equation}
    |\Phi_{i,j}|_g \leq c_0 |\Phi_{i,j}|_{g_{i,j}} \leq \sqrt{14} c_0 = \tilde c_0,
\end{equation}
where the second inequality holds because $g_{i,j}$ is the Riemannian metric induced by $\Phi_{i,j}$. 

Now,
\begin{equation}
    \nabla \Phi_{i,j} = \nabla_{g_{i,j}}\Phi_{i,j} + (\nabla - \nabla_{g_{i,j}})\Phi_{i,j}.
\end{equation}
Since $\Phi_{i,j}=F_{i,j}^*\Phi_{i+j}$ and $g_{i,j}=F_{i,j}^*g_{i+j}$, the torsion tensor $T_{i,j}$ of $\Phi_{i,j}$ is $F_{i,j}^*T_{i+j}$. Hence the assumptions \eqref{eqBoundedSups} and \eqref{eqInitialBoundedSups} imply uniform bounds for $T_{i,j}$ and all its $g_{i,j}$-covariant derivatives.
Using this and the schematic identity $\nabla_{g_{i,j}}\Phi_{i,j} = T_{i,j}*\Phi_{i,j}$, together with the uniform bounds on $S$ discussed above, yield the existence of a constant $\tilde{c_1}$ such that
\begin{equation}
    |\nabla\Phi_{i,j}|_g \leq c_0 |\nabla_{g_{i,j}}\Phi_{i,j}|_{g_{i,j}} + C|S|_g|\Phi_{i,j}|_g \leq \tilde c_1.
\end{equation}
Similarly, we have that 
\begin{align}
    \nabla^2\Phi_{i,j} = &\nabla^2_{g_{i,j}}\Phi_{i,j} + (\nabla - \nabla_{g_{i,j}})\nabla_{g_{i,j}} \Phi_{i,j} \nonumber\\
    &+(\nabla(\nabla - \nabla_{g_{i,j}}))\Phi_{i,j} + (\nabla - \nabla_{g_{i,j}})\nabla\Phi_{i,j},
\end{align}
so our uniform bounds on $S$ and $\nabla S$ give that there exists a constant $\tilde c_2$ such that
\begin{align}
    |\nabla^2\Phi_{i,j}|_g \leq & C|\nabla^2_{g_{i,j}}\Phi_{i,j}| + C|S|_g|\nabla_{g_{i,j}} \Phi_{i,j}|_g \nonumber\\
    &+C|\nabla S|_g|\Phi_{i,j}|_g + C|S|_g|\nabla\Phi_{i,j}|_g,\\
    \leq &\tilde c_2,
\end{align}
where the bound on the term $|\nabla^2\Phi_{i,j}|_g$ comes from taking a $g_{i,j}$-covariant derivative of the schematic equation $\nabla_{g_{i,j}}\Phi_{i,j} = T_{i,j}*\Phi_{i,j}$, and applying \eqref{eqBoundedSups}. The other terms are already bounded by the discussion above.

Lastly, for $k\geq 2$, we have that
\begin{equation}
    |\nabla^{k} \Phi_{i,j}|_g \leq C \sum_{l=0}^k|S|^l_g|\nabla_{g_{i,j}}^{k-l}\Phi_{i,j}|_{g_{i,j}} + C\sum_{l=1}^{k-1}|\nabla^lS|_g|\nabla^{k-1-l}\Phi_{i,j}|_g,
\end{equation}
so an inductive argument, using that $|\nabla^kS|_g \leq c_k$ and the assumptions \eqref{eqBoundedSups} and \eqref{eqInitialBoundedSups}, gives that there exists a constant $\tilde c_k$ such that
\begin{equation}
    |\nabla^k \Phi_{i,j}|_{g} \leq \tilde c_k,
\end{equation}
for all $j,k\geq 0$ (after possibly enlarging $\tilde{c}_k$ to handle the finitely many $j$ that are not "sufficiently large").

Now, the $C^k$-Arzelà-Ascoli theorem implies that there exists a subsequence of $\Phi_{i,j}$ in $j$ (which we still denote $\Phi_{i,j}$) and a limiting $4$-form $\Phi_{i,\infty}$ such that $\Phi_{i,j}$ converges to $\Phi_{i,\infty}$ smoothly on $\Omega_i$. That is, for all $k \geq0$,

\begin{equation}
    \left| \nabla^k(\Phi_{i,j}- \Phi_{i,\infty}) \right|_g \to 0 \text{ as }j\to \infty,
\end{equation}
uniformly on $\Omega_i$. This gives a limiting $4$-form $\Phi_{i,\infty}$ on $\Omega_i$, and we now argue that it is actually a \Sp-structure on $\Omega_i$, in the same way that we obtained the contradiction in Theorem \ref{ThmLongTime}. Indeed, for all $i,j$, $\Phi_{i,j}$ is a \Sp-structure on $\Omega_i$ inducing the Riemannian metric $g_{i,j}$. So, for each $i,j$, we have that 
\begin{equation}\label{eqPosDef2}
    (v \lrcorner w \lrcorner\Phi_{i,j})\wedge(v \lrcorner w \lrcorner\Phi_{i,j})\wedge \Phi_{i,j} = 6|v\wedge w|_{g_{i,j}}^2\mathrm{vol}_{\Phi_{i,j}}.
\end{equation}
Taking the limit as $j \to \infty,$ and using Cheeger-Gromov to say that $g_{i,j}$ converges to a Riemannian metric $g_{i,\infty}$, we see that the right-hand side converges to a positive volume form and so $\Phi_{i,j}$ converges to a $4$-form $\Phi_{i,\infty}$ inducing the Riemannian metric $g_{i,\infty}$. That is, the limiting $4$-form $\Phi_{i,\infty}$ is non-degenerate, and it remains to check that $\Phi_{i,\infty}$ satisfies the other condition of Theorem \ref{ThmCharacterisation}. Indeed, just as we argued in Theorem \ref{ThmLongTime}, the extra conditions of Theorem \ref{ThmCharacterisation} are continuous algebraic conditions on $\Phi(t)$, and so are preserved in the limit as $j\to \infty.$

This shows that there exists a limiting \Sp-structure $\Phi_{i,\infty}$ on each $\Omega_i$, and we now use a diagonalisation argument to show that we can use this to deduce the existence of a limiting \Sp-structure $\Phi$ on all of $M$.

Applying the preceding argument to $\Omega_1$, we obtain a subsequence along which $\Phi_{1,j}$ converges smoothly on $\Omega_1$. Applying the same argument to $\Omega_2$, and passing to a further subsequence, we obtain smooth convergence on $\Omega_2$. Continuing inductively and taking the diagonal subsequence, we obtain a single subsequence, which we still index by $j$, such that for every $i$,
\[
g_{i,j}\to g_{i,\infty},
\qquad
\Phi_{i,j}\to \Phi_{i,\infty},
\]
smoothly on \(\Omega_i\). Here \(g_{i,\infty}=g|_{\Omega_i}\), where \(g\) is the Cheeger-Gromov limit metric.

We now show that the limiting \Sp-structures $\Phi_{i,\infty}$ agree on intersections of the nested exhaustion $\{\Omega_i\}$, and hence glue together to define a smooth \Sp-structure on all of $M$.
For $k\geq i$, we denote the inclusion map of $\Omega_i$ into $\Omega_k$ by \[
I_{ik}: \Omega_i \to \Omega_k.
\]
By definition of the inclusion map $I$,
\[
I^*_{ik}g_{k,j} = g_{i,j}, \text{ and } I^*_{ik}\Phi_{k,j} = \Phi_{i,j}.
\]
Taking the limit as $j\to \infty$ (and noting that $I_{i,k}^*$ is independent of $j$, so passes through the limit), we find that
\[
I^*_{ik}g_{k,\infty} = g_{i,\infty}, \text{ and } I^*_{ik}\Phi_{k,\infty} = \Phi_{i,\infty}. 
\]
This shows that the metrics $g_{i,\infty}$ and $4$-forms $\Phi_{i,\infty}$ agree on the overlaps of the nested exhaustion $\{\Omega_i\}$ so, writing $I_i:\Omega_i \to M$ for the inclusion, we have a metric $g$ and $4$-form $\Phi$ such that
\[
I^*_{i}g = g_{i,\infty}, \text{ and } I^*_{i}\Phi = \Phi_{i,\infty},
\]
and $\Phi$ is a \Sp-structure as before, inducing the Cheeger-Gromov limit metric $g$. This gives a candidate limiting triple $(M,\Phi,p)$, and we finally show that indeed $(M_i,\Phi_i,p_i)$ converges to $(M,\Phi,p)$.

Since $\{\Omega_i\}$ exhausts $M$, we have that for any compact $\Omega \subset M$, there exists $i_0$ such that $\Omega \subset \Omega_i$ for all $i \geq i_0$. Fix $i$ so that $\Omega \subset \Omega_i$. Then, on $\Omega$ we have that
\begin{equation}
    \left |\nabla^k(F^*_{i+j}\Phi_{i+j}-\Phi) \right |_{g} = \left |\nabla^k(\Phi_{i,j}-\Phi_{i,\infty}) \right |_{g} \to 0,
\end{equation}
as $l=i+j \to \infty$, for all $k \geq0$. Thus, $(M_i,\Phi_i,p_i) \to (M,\Phi,p)$ in the sense defined at the start of this section.

\end{proof}

\subsection{Compactness for solutions to reasonable flows of \texorpdfstring{\Sp-structures}{Spin(7)-structures}}
Using the compactness theorem for the space of \Sp-structures, we now state and prove our compactness theorem for solutions of reasonable flows of \Sp-structures. 
\begin{thm}\label{thmCompactnessSpaceofSolutions}
    Let $M_i$ be a sequence of compact $8$-manifolds and let $p_i \in M_i$ for each $i$. Suppose that $\Phi_i(t)$ is a sequence of solutions to a given reasonable flow of \Sp-structures on $M_i$, with the induced sequence of Riemannian metrics $g_i(t)$ on $M_i$, for $t\in (a,b)$, where $-\infty \leq a < 0 <b \leq \infty$.

    Suppose further that 
    \begin{equation}\label{eqFirstAssumption}
        \sup_i\sup_{x \in M_i, t \in (a,b)}\left(\left|\nabla_{g_i(t)}T_i(x,t) \right |_{g_i(t)}^2 + \left|\Riem_i(x,t) \right |_{g_i(t)}^2 + \left|T_i(x,t) \right |_{g_i(t)}^4 \right)^{\frac{1}{2}}< \infty,
    \end{equation}
    where $T_i$ and $\Riem_i$ denote the torsion and Riemann curvature tensor induced by $\Phi_i$. Finally, suppose that the injectivity radius of each initial manifold $(M_i,g_i(0))$ at $p_i$ satisfies
    \begin{equation}\label{eqSecondInj}
        \inf_i \inj(M_i,g_i(0),p_i)>0.
    \end{equation}
    Then, there exists an $8$-manifold $M$, a point $p\in M$ and a solution $\Phi(t)$ to the same reasonable flow of \Sp-structures that $\Phi_i(t) $ solves for $t\in (a,b)$ such that, after passing to a subsequence,
    \begin{equation}
        (M_i,\Phi_i(t),p_i) \to (M,\Phi(t),p) \text{ as } i \to \infty,
    \end{equation}
    where this convergence is in the sense discussed in the previous section.
\end{thm}
\begin{proof}
We follow the proof of the analogous result for solutions to the Laplacian flow of $G_2$-structures \cite[Theorem 7.2]{Lotay-Wei}.
    We first note that it suffices to prove this in the case that $a$ and $b$ are both finite, since a standard diagonalisation argument then gives the result in the case that either $a$ or $b$ (or both) are infinite.

    Using the assumption \eqref{eqFirstAssumption} and our Shi-type estimates for reasonable flows of \Sp-structures (Theorem \ref{ThmShiType}), we have that there exist constants $C_k$, independent of $i$, such that
    \begin{equation}\label{eqCompactnessCurvatureTorsion}
        \left|\nabla^k_{g_i(t)}\Riem_i(x,t) \right|_{g_i(t)} + \left|\nabla^{k+1}_{g_i(t)}T_i(x,t) \right|_{g_i(t)} \leq C_k,
    \end{equation}
    for all $k$. This, together with our assumption on the injectivity radius \eqref{eqSecondInj}, enables us to apply Theorem \ref{ThmSpin7Compact} to extract a subsequence $(M_i,\Phi_i(0),p_i)$, converging to some complete limit $(M,\tilde\Phi_{\infty}(0),p)$. Using the notation of the previous theorem, we write
    \[
    F_i^*\Phi_i(0)\to \tilde\Phi_{\infty}(0),
    \]
    smoothly on any compact subset $\Omega \subset M$ as $i\to \infty.$ This gives convergence at the initial time, and we now extend this to obtain a convergent subsequence $\Phi_i(t)$ for all $t \in (a,b)$.

    To that end, let $\tilde\Phi_i(t) = F_i^*\Phi_i(t)$, for $F_i$ as above, and fix a compact subset $\Omega \times [c,d] \subset M \times (a,b)$. Fix $i$ sufficiently large so that $\Omega \subset \Omega_i,$ for some set $\Omega_i$ an element of the exhaustion of $M$ defined in the previous theorem. Then, $\tilde\Phi_i(t)$ is a sequence of solutions to the same reasonable flow of \Sp-structures as $\Phi_i(t)$, for $t\in [c,d]$, with induced metrics $\tilde g_i(t) =F^*_ig_i(t)$. By Claims \ref{ClaimSupMetricDerivative} and \ref{ClaimSupSpin7Derivative} (and the fact that these claims are diffeomorphism invariant so pulling back by $F$ does not affect their veracity), we can deduce from \eqref{eqCompactnessCurvatureTorsion} that there exist constants $C_k$, independent of $i$, such that
    \begin{equation}
        \sup_{\Omega\times [c,d]}\left(\left|\nabla^k_{g_{\infty}(0)}\tilde{g}_i(t) \right|_{g_\infty(0)} +\left|\nabla^k_{g_{\infty}(0)}\tilde{\Phi}_i(t) \right|_{g_\infty(0)} \right) \leq C_k,
    \end{equation}
    for all sufficiently large $i$. Now, since we can rewrite time derivatives of $\nabla^k_{\tilde{g_\infty}(0)}\tilde g_i(t)$ and $\nabla^k_{\tilde{g_\infty}(0)}\tilde \Phi_i(t)$ in terms of further space derivatives,  (using the known evolution equations for $\tilde g_i$ and $\tilde\Phi_i$), we have that
    \begin{equation}
        \sup_{\Omega\times [c,d]}\left(\left|\frac{\pt^l}{\pt t^l}\nabla^k_{g_{\infty}(0)}\tilde{g}_i(t) \right|_{g_\infty(0)} +\left|\frac{\pt^l}{\pt t^l}\nabla^k_{g_{\infty}(0)}\tilde{\Phi}_i(t) \right|_{g_\infty(0)} \right) \leq C_k,
    \end{equation}
    for all $l$ and $k$. Thus, the Arzelà-Ascoli theorem gives the existence of a subsequence of $\tilde{\Phi}_i(t)$ converging smoothly on $\Omega \times [c,d]$. 
    Since the flow is a diffeomorphism-invariant evolution equation involving only finitely many derivatives of $\Phi$, smooth convergence of the sequence $\tilde{\Phi}_i(t)$ on $\Omega \times [c,d]$ yields smooth convergence of every term appearing in the evolution equation. Consequently, the limit $\tilde\Phi_\infty(t)$ is a solution to the given reasonable flow on $\Omega \times [c,d]$.
    A standard diagonalisation argument then allows us to obtain a subsequence converging smoothly on any compact subset of $M \times(a,b)$ to a solution $\tilde\Phi_\infty(t)$ of the reasonable flow in question.
\end{proof}
One key application of compactness results for geometric flows is in the study of the formation of singularities. We briefly describe how this can be done here, by way of parabolic rescalings. We will also use this parabolic rescaling in Section \ref{SectionSingularities}.
Let $M^8$ be a compact $8$-manifold and let $\Phi(t)$ be a solution to a given reasonable flow of \Sp-structures on $M$, on a finite maximal time interval $[0,T_0),$ for some $T_0<\infty.$ By Theorem \ref{ThmLongTime}, the quantity $\Lambda(t)$, given by $\Lambda(t) = \sup_{x\in M}\Lambda(x,t),$ where $\Lambda(x,t)$ is as defined in \eqref{eqLambdaDef} satisfies
\[
\lim_{t \nearrow T_0} \Lambda(t) = \infty.
\]
So, define a sequence of spacetime points $(x_i,t_i)$ such that $t_i\nearrow T_0$ and
\[
\Lambda(x_i,t_i) = \sup_{x\in M, t\in [0,t_i]}\left ( \left|\nabla T(x,t) \right |^2_{g(t)} + \left |\Riem(x,t) \right |^2_{g(t)} + \left|T(x,t) \right|_{g(t)}^4 \right )^\frac{1}{2}.
\]
Consider a sequence of parabolic rescalings $\Phi_i(t)$ of $\Phi(t)$, defined by
\begin{equation}\label{eq:parabolicrescaling}
    \Phi_i(t) = \Lambda(x_i,t_i)^4\Phi(t_i+\Lambda(x_i,t_i)^{-2}t),
\end{equation}
and define
\begin{equation}\label{eq:lambdaPhii}
    \Lambda_{\Phi_i(t)}(x,t) = \left( \left|\nabla_{g_i(t)} T_i(x,t) \right |^2_{g_i(t)} + \left |\Riem_i(x,t) \right |^2_{g_i(t)} + \left|T_i(x,t) \right|_{g_i(t)}^4 \right )^\frac{1}{2},
\end{equation}
where $g_i,T_i$ and $\Riem_i$ are those induced by $\Phi_i$.
Now, by the rescaling property for reasonable flows \eqref{eqrescaledmetric1}, and the fact that differentiating \eqref{eq:parabolicrescaling} gives
\begin{align*}
    \ptt\Phi_i(t) &= \Lambda(x_i,t_i)^4\Lambda(x_i,t_i)^{-2}\ptt\Phi(t_i+\Lambda(x_i,t_i)^{-2}t) \\&= \Lambda(x_i,t_i)^{2}\ptt\Phi(t_i+\Lambda(x_i,t_i)^{-2}t),
\end{align*}
we have that $\Phi_i(t)$ is still a solution to the reasonable flow, now on the time interval 
\[
[-t_i\Lambda(x_i,t_i)^2, (T_0-t_i)\Lambda(x_i,t_i)^2],\] for each $i$. Moreover, from the definition of $\Lambda_{\Phi_i}$ and the scaling property of the right hand side of \eqref{eq:lambdaPhii}, we have for all $i$ that 
\begin{equation}
    \sup_{x \in M} \Lambda_{\Phi_i}(x,t) = \frac{\sup_{x \in M}\Lambda(x,t_i + \Lambda(x_i,t_i)^{-2}t) }{\Lambda(x_i,t_i)} \leq 1, \text{ for }t\leq 0,
\end{equation}
and 
\begin{equation}
    \Lambda_{\Phi_i}(x_i,0) = 1.
\end{equation}
By construction, the rescaled solutions $\Phi_i$ satisfy $\sup\Lambda_{\Phi_i}(x,0) = 1$. So, the modified quantity $\tilde \Lambda_{\Phi_i(x,0)}$ \eqref{eq:lambdaPhii} for the rescaled flow satisfies 
\[
\sup_{x \in M} \tilde{\Lambda}_{\Phi_i}(x,0) = \sup_{x \in M} \left( \Lambda_{\Phi_i}(x,0)^2 + 1 \right)^{1/2} \leq \sqrt{1^2 + 1} = \sqrt{2}.
\]
Now, applying the doubling time estimate (Proposition \ref{PropDoublingTime}) and Corollary \ref{CorollaryDoublingLongTime}, we see that there exists $b >0$ such that
\[
\sup_{x \in M}\Lambda_{\Phi_i}(x,t)  \leq  \sup_{x \in M}\tilde\Lambda_{\Phi_i}(x,t) \leq 2\sqrt{2}\text{ for } t \leq b.
\]

Thus, we obtain a sequence $(M,\Phi_i(t))$ of solutions to the reasonable flow of \Sp-structures, defined on $(a,b)$ for some $a<0$, and satisfying 
\[
\sup_i\sup_{x \in M}\Lambda_{\Phi_i}(x,t)<\infty.
\]
So, if we can also establish the bound on the injectivity radius
\[
\inf_i \inj(M,g_i(0),x_i)>0,
\]
then the result of Theorem \ref{thmCompactnessSpaceofSolutions} applies, allowing us to extract a convergent subsequence.
We expect that this parabolic blow-up at finite time singularities will be a useful tool in the further study of these reasonable flows of \Sp-structures, and particularly for the Ricci-Harmonic flow, which has already received attention \cite{DwivediRHF}. We will also use this in the following section.

\section{Finite time singularities}\label{SectionSingularities}
In this section, we use the Shi-type estimates of Section \ref{SectionShiType}, the characterisation of finite-time singularities of Section \ref{SectionLongTime} and the compactness results of Section \ref{SectionCompactness} to study finite time singularities of reasonable flows of \Sp-structures. We begin in Subsection \ref{SectionExtensionAssumingBoundedTorsion} by proving that, under the assumption that the induced metrics are uniformly continuous, a reasonable flow can be extended for as long as the torsion tensor remains bounded. We then recall Chen's $\kappa$-non-collapsing theorem \cite{Chen} in Subsection \ref{Subsectionkappa}, before using this to study the structure of finite-time singularities in Subsection \ref{SubsectionFiniteTime}, under suitable assumptions that allow us to apply the $\kappa$-non-collapsing theorem.
\subsection{Extension assuming bounded torsion}\label{SectionExtensionAssumingBoundedTorsion}

In Theorem \ref{ThmLongTime}, we showed that a solution to a reasonable flow of \Sp-structures can be extended for as long as the quantity $\Lambda(t)$ remains bounded. Here, we prove another extension result, using the compactness results of Section \ref{SectionCompactness}. 
In particular, we prove, under the additional assumption of uniform continuity of the induced metrics along the flow, that a solution can be extended for as long as the torsion tensor remains bounded, as is the case for $G_2$-flows \cite{Lotay-Wei}.
We first recall the definition of uniform continuity for a one-parameter family of Riemannian metrics, and some consequences for their induced distance functions and volumes of geodesic balls.

We say that a family $g(t)$ of metrics is \emph{uniformly continuous} on the time interval $[0,T_0)$ for $T_0<\infty$ if, for any $\varepsilon>0$, there exists $\delta >0$ such that for any $0\leq t_0 <t<T_0$ with $t-t_0 \leq \delta$ we have
\[
|g(t)-g(t_0)|_{g(t_0)}\leq \varepsilon.
\]
As symmetric $2$-tensors, this implies that
\[
(1-\varepsilon)g(t_0) \leq g(t)\leq (1+\varepsilon) g(t_0),
\]
for all such $t$, which implies the following inequality for the induced volume forms:
\begin{equation}\label{eqVolumeFormInequality}
    (1-\varepsilon)^{4}\text{vol}_{g(t_0)} \leq \text{vol}_{g(t)} \leq (1+\varepsilon)^{4}\text{vol}_{g(t_0)}.
\end{equation}

Then, writing $d_g$ for the distance function induced by a Riemannian metric $g$, we have that for any $x,y\in M$ and $t-t_0<\delta$,
\[
\sqrt{1-\varepsilon}d_{g(t_0)}(x,y) \leq d_{g(t)}(x,y) \leq \sqrt{1+\varepsilon}d_{g(t_0)}(x,y).
\]
So, the geodesic balls centred at $x$ with radius $r$ satisfy 
\begin{equation}\label{eqBallINclusion}
    B_{g(t_0)}\left(x,\frac{r}{\sqrt{1+\varepsilon}}\right) \subset B_{g(t)}(x,r),
\end{equation}
and so their volumes satisfy the following inequality: 
\begin{equation}\label{BallVolumeInequality}
    (1-\varepsilon)^{4}\Vol_{g(t_0)}\left(B_{g(t_0)}\left(x,\frac{r}{\sqrt{1+\varepsilon}}\right)  \right) \leq \Vol_{g(t)}\left( B_{g(t)}(x,r)\right),
\end{equation}
where the factor of $(1-\varepsilon)^{4}$ comes from \eqref{eqVolumeFormInequality}.
We will use this final inequality in the proof of the following theorem.
\begin{thm}\label{thmExtensionBoundedTorsion}
    Let $M^8$ be a compact $8$-manifold, and $\Phi(t)$ a solution to a reasonable flow of \Sp-structures, for $t\in[0,T_0),$ $T_0 < \infty$. Let $g(t)$ be the associated family of metrics, and suppose that $g(t)$ is uniformly continuous, and that the torsion tensor $T$ satisfies
    \begin{equation}\label{eqTbounded}
         \sup_{M \times [0,T_0)}|T(x,t)|_{g(t)}< \infty.
    \end{equation}
    Then, the solution $\Phi(t)$ can be extended past time $T_0$.
\end{thm}
\begin{proof}
    We begin by outlining the strategy of the proof. Aiming for a contradiction, we will suppose that the conditions of the theorem hold, but that the flow cannot be extended past time $T_0$. Then, $\Lambda(t)$ \eqref{eqLambdaDef} must tend to infinity along the flow, and we will use this to construct a rescaled flow. For this rescaled flow, we will obtain a volume lower bound, which we will use to obtain an injectivity radius lower bound. At this point, we can apply the result of Theorem \ref{thmCompactnessSpaceofSolutions} to obtain a convergent subsequence. The limit has vanishing torsion, hence is Ricci-flat. Finally, the Euclidean volume growth of the limit and the Bishop–Gromov relative volume comparison theorem force the limit to be flat, contradicting the normalization of the blow-up.
    
    Assume, for the sake of contradiction, that the conditions of the theorem hold but that the solution $\Phi(t)$ cannot be extended past $T_0$. Then, $[0,T_0)$ is a maximal time interval, so Theorem \ref{ThmLongTime} gives that there exists a sequence of space-time points $(x_i,t_i)$ such that $t_i\nearrow T_0$ and 
    \[
    \Lambda(x_i,t_i) = \sup_{x\in M, t\in [0,t_i]}\left(\left|\nabla T(x,t) \right|_{g(t)}^2 + \left|\Riem(x,t) \right|_{g(t)}^2+ \left|T(x,t) \right|_{g(t)}^4\right)^{\frac{1}{2}} \to \infty.
    \]
    Then, using the notation at the end of Section \ref{SectionCompactness}, we define $\Phi_i(t)$ by \eqref{eq:parabolicrescaling}, obtaining a sequence of flows $(M,\Phi_i(t),x_i)$ defined on $[-t_i\Lambda(x_i,t_i)^2,0]$, and we have that $\Lambda_{\Phi_i(t)}(x,t)$ defined by \eqref{eq:lambdaPhii} satisfies
\begin{equation}\label{boundedrescaledlambda}
    \sup_{M \times [-t_i\Lambda(x_i,t_i)^2,0]}|\Lambda_{\Phi_i(t)}(x,t)| \leq 1 \text{ and } |\Lambda_{\Phi_i(t)}(x_i,0)| =1.
\end{equation}
The metric $g_i(t)$ induced by $\Phi_i(t)$ is 
\[
g_i(t) = \Lambda_i(x_i,t_i)^2
g(t_i + \Lambda(x_i,t_i)^{-2}t).\]
Thus, for any $x\in M$ and $r\leq \Lambda(x_i,t_i)$ we have that 
\begin{align}
    \Vol_{g_i(0)}(B_{g_i(0)}(x,r)) &= \Lambda(x_i,t_i)^8\Vol_{g(t_i)}(B_{g(t_i)}(x,\Lambda(x_i,t_i)^{-1}r)), \text{ by definition of } g_i\\
    & \geq \Lambda(x_i,t_i)^8\Vol_{g(t_i)}(B_{g(t_0)}(x,(1+\varepsilon)^{-\frac{1}{2}}\Lambda(x_i,t_i)^{-1}r)), \text{ by \eqref{eqBallINclusion}}\\
    &\geq\Lambda(x_i,t_i)^8(1-\varepsilon)^4\Vol_{g(t_0)}(B_{g(t_0)}(x,(1+\varepsilon)^{-\frac{1}{2}}\Lambda(x_i,t_i)^{-1}r)), \text{ by \eqref{BallVolumeInequality}}\\
    &\geq C\Lambda(x_i,t_i)^8(1-\varepsilon)^4((1+\varepsilon)^{-\frac{1}{2}}\Lambda(x_i,t_i)^{-1}r)^8 \label{Penultimateline}\\
    &\geq C(1-\varepsilon)^4(1+\varepsilon)^{-4}r^8,
\end{align}
where the penultimate line \eqref{Penultimateline} follows from the uniform lower bound $\Vol_{g(t_0)}(B_{g(t_0)}(x,\rho)) \ge C \rho^8$ on $(M,g(t_0))$, where $\rho = (1+\varepsilon)^{-1/2}\Lambda(x_i,t_i)^{-1} r$ is sufficiently small.
So, we have that 
\begin{equation}\label{VolumeBound}
    \Vol_{g_i(0)}(B_{g_i(0)}(x,r)) \geq Cr^8,
\end{equation}
for all $x\in M,r\in[0,\Lambda(x_i,t_i)]$, for some universal constant $C$. 

Now, by definition of $\Lambda_{\Phi_i}$ \eqref{eq:lambdaPhii}, we have
\begin{equation}\label{eqBoundedRiem}
   |\Riem_{g_i}(x,0)| \leq1 ,
\end{equation}
for all $x\in M$. By the volume bound \eqref{VolumeBound}, bounded curvature \eqref{eqBoundedRiem} and the Cheeger-Gromov-Taylor injectivity radius theorem \cite[Theorem 5.42]{ChowLuNi}, we have a uniform injectivity radius lower bound: 
\[
\inj(M,g_i(0),x_i)\geq c,
\]
for some constant $c$. This, together with \eqref{boundedrescaledlambda}, allows an application of our compactness theorem (Theorem \ref{thmCompactnessSpaceofSolutions}), yielding a subsequence of $(M,\Phi_i(t),x_i)$ converging to a limit $(M_\infty,\Phi_\infty(t),x_\infty)$, for $t \in (\infty,0]$, with $|\Lambda_\infty(x_\infty,0)| = 1$. Now, since $T$ remains bounded \eqref{eqTbounded} and $\Lambda(x_i,t_i) \to  \infty$ as $t_i \to \infty,$ we have that
\begin{equation}
    |T_i(x,t)|^2_{g_i(t)} = \Lambda(x_i,t_i)^{-2}|T(x,t_i+\Lambda(x_i,t_i)^{-2}t)|_{g(t_i + \Lambda(x_i,t_i)^{-2}t)} \to 0
\end{equation}
as $i \to \infty.$

So, the limit $(M_\infty,\Phi_{\infty}(t))$ has zero torsion for all $t \in (-\infty,0]$, and hence the induced metric $g_\infty(0)$ is Ricci-flat. Then, as in \cite[pp. 239-240]{ChowLuNi}, we have that $g_\infty(0)$ has precisely Euclidean volume growth:
\begin{equation}\label{eqEuclideanVolumeGrowth}
    \Vol_{g_\infty(0)}(B_{g_{\infty(0)}}(x_\infty,r)) = \Vol_{g_{\mathbb{R}^8}}(B_{g_{\mathbb{R}^8}}(0,1))r^8.
\end{equation}
This, together with the fact that $M_\infty$ is complete and Ricci-flat allows the application of the Bishop-Gromov relative volume comparison theorem, which implies that
\begin{equation}\label{eqContradiction}
   \Riem(g_\infty(0)) = 0 \text{ on }M_\infty. 
\end{equation}
But, since the torsion vanishes, we have that
\[
|\Riem_{g_\infty}(x_\infty,0)| = |\Lambda_{\Phi_\infty}(x_\infty,0)| = 1,
\]
which contradicts \eqref{eqContradiction}, and so our solution $\Phi(t)$ can be extended past $T_0$.
\end{proof}
In the statement of Theorem \ref{thmExtensionBoundedTorsion}, we assumed that the family of metrics $g(t)$ was uniformly continuous. This was used only to obtain the volume bound \eqref{VolumeBound} which was used to argue that the limit manifold had Euclidean volume growth \eqref{eqEuclideanVolumeGrowth}. In certain cases, we can drop this assumption, if there is some other way to obtain Euclidean volume growth in the limit. We illustrate one such way in Theorem \ref{ThmIntegralBlowup1} where we instead assume an integral bound on the Ricci curvature, which we show implies uniform continuity of the metrics. Here we also only need an integral bound on the torsion to get extension of the flow. 
Later, in Subsection \ref{SubsectionFiniteTime}, we will use a $\kappa$-non-collapsing theorem to prove similar results, without making any claims about uniform continuity of the metrics. The $\kappa$-non-collapsing result avoids these assumptions, by providing an alternative pathway to the conclusion about Euclidean volume growth, as we will see in Theorem \ref{ThmMaximalVolumeGrowthSpin7Manifold}.

\begin{remark}
    We also note that, during the proof of Theorem \ref{thmExtensionBoundedTorsion}, we extracted a maximal volume growth torsion-free limit in the rescaled flow $\Phi_i(t)$, under the assumption that the torsion remained bounded at a finite-time singularity. This is similar to Theorem \ref{ThmMaximalVolumeGrowthSpin7Manifold}, where we will prove the same result under looser assumptions, using $\kappa$-non-collapsing.
\end{remark}

We begin with the following theorem, before discussing $\kappa$-non-collapsing and its consequences.

\begin{thm}\label{ThmIntegralBlowup1}
    Let $\Phi(t)$ be a solution to a reasonable flow of \Sp-structures on a compact manifold $M^8$, on a maximal time interval $[0,T_0)$, for some $T_0<\infty.$
    Then,
    \begin{equation}
        \int_0^{T_0}\sup_{x \in M}(|\Ric(x,t)|_{g(t)}+|T(x,t)|_{g(t)}^2)\d t = \infty.
    \end{equation}
\end{thm}
\begin{proof}
We adapt the ideas of \cite[Theorem 5.2]{Chen} to the \Sp-setting.
    Aiming for a contradiction, we start by assuming that
    \begin{equation}\label{eqContradictionAssum}
        \int_0^{T_0}\sup_{x \in M}(|\Ric(x,t)|_{g(t)}+|T(x,t)|_{g(t)}^2)\d t < \infty.
    \end{equation}
Since $\Phi(t)$ is a reasonable flow,
\[
\frac{\partial}{\partial t}g
=
-2\Ric+L(T)+T*T+C
\]
by \eqref{eqReasonableMetric}. Hence,
\[
\sup_{x\in M}\left|\frac{\partial}{\partial t}g(x,t)\right|_{g(t)}
\le
C\sup_{x\in M}\bigl(
|\Ric(x,t)|_{g(t)}+|T(x,t)|_{g(t)}+|T(x,t)|_{g(t)}^2+1
\bigr).
\]
Using the fact that \(|T|\le |T|^2+1\), we have
\[
\sup_{x\in M}\left|\frac{\partial}{\partial t}g(x,t)\right|_{g(t)}
\le
C\sup_{x\in M}\bigl(
|\Ric(x,t)|_{g(t)}+|T(x,t)|_{g(t)}^2+1
\bigr).
\]
Therefore, by \eqref{eqContradictionAssum},
\[
\int_0^{T_0}
\sup_{x\in M}\left|\frac{\partial}{\partial t}g(x,t)\right|_{g(t)}\,dt
<\infty.
\]
It follows that the family of metrics \(g(t)\) is uniformly continuous on \([0,T_0)\).
Moreover, since $T_0$ is the maximal existence time, Theorem \ref{ThmLongTime} gives that the quantity $\Lambda(t)$ blows up as $t$ approaches $T_0.$ We now work as in the proof of Theorem \ref{thmExtensionBoundedTorsion}, choosing a sequence of spacetime points $(x_i,t_i)$ such that $t_i\nearrow T_0$ and 
\[
\Lambda(x_i,t_i) = \sup_{x\in M, t\in [0,t_i]}\left(\left|\nabla T(x,t) \right|_{g(t)}^2 + \left|\Riem(x,t) \right|_{g(t)}^2+ \left|T(x,t) \right|_{g(t)}^4\right)^{\frac{1}{2}} \to \infty.
\]
Then, again as in the proof of Theorem \ref{thmExtensionBoundedTorsion}, we define a sequence of parabolically rescaled flows
\begin{equation}\label{eq:parabolicrescaling2}
    \Phi_i(t) = \Lambda(x_i,t_i)^4\Phi(t_i+\Lambda(x_i,t_i)^{-2}t).
\end{equation}
This again gives a sequence of flows $(M,\Phi_i(t),x_i)$ defined on $[-t_i\Lambda(x_i,t_i)^2,0]$, and we have that $\Lambda_{\Phi_i(t)}(x,t)$ defined by \eqref{eq:lambdaPhii} satisfies

\begin{equation}
    \sup_{M \times [-t_i\Lambda(x_i,t_i)^2,0]}|\Lambda_{\Phi_i(t)}(x,t)| \leq 1 \text{ and } |\Lambda_{\Phi_i(t)}(x_i,0)| =1.
\end{equation}
The metric $g_i(t)$ induced by $\Phi_i(t)$ is 
\begin{equation}\label{eqRescaledMetric}
   g_i(t) = \Lambda(x_i,t_i)^2
g(t_i + \Lambda(x_i,t_i)^{-2}t). 
\end{equation}
By the Shi-type estimates (Theorem \ref{ThmShiType}), covariant derivatives of all orders of curvature and torsion are bounded on every compact time interval $[-A,0]$.

Still working as in the proof of Theorem \ref{thmExtensionBoundedTorsion}, we obtain the following uniform lower bound on the volume of the $g_i(0)$-balls:
\begin{equation}
    \Vol_{g_i(0)}(B_{g_i(0)}(x,r)) \geq Cr^8,
\end{equation}
for all $x\in M,r\in[0,\Lambda(x_i,t_i)]$, for some universal constant $C$ (see the sequence of inequalities leading up to  \eqref{VolumeBound}). Again, by the definition of $\Lambda_{\Phi_i}$ \eqref{eq:lambdaPhii}, we have that
\[
|\Riem_{g_i}(x,0)|\leq 1,
\]
for all $x\in M$, and we again get the uniform lower bound on the injectivity radius:
\begin{equation}
\inj(M,g_i(0),x_i)\geq c,
\end{equation}
for some constant $c$. Theorem \ref{thmCompactnessSpaceofSolutions} then yields a subsequence of $(M_i,\Phi_i(t),x_i)$ converging to a limit $(M_\infty,\Phi_\infty(t),x_\infty)$, for $t \in (-\infty,0]$, with 
\begin{equation}\label{eqLimitLambda1}
    \Lambda_\infty(x_\infty,0) = 1.
\end{equation}

On the other hand, let $s = t_i + \Lambda(x_i,t_i)^{-2}t$, so that $dt = \Lambda(x_i,t_i)^2ds$ and $g_i(t) = \Lambda(x_i,t_i)^2g(s)$ (cf. \eqref{eqRescaledMetric}). So, by the homogeneity of $\Ric$ and $T$, we have that for any $A>0$,
\begin{align}
    &\int_{-A}^0\sup_{x\in M}(|\Ric_{g_i(t)}(x,t)|_{g_i(t)}+|T_i(x,t)|_{g_i(t)}^2)dt \\
    &= \int_{t_i -A/\Lambda(x_i,t_i)^2}^{t_i}\sup_{x\in M}(|\Ric_{g(s)}(x,s)|_{g(s)}+|T(x,s)|_{g(s)}^2)ds.
\end{align}
Taking the limit as $i\to\infty$, using the fact that such limit exists by the discussion above, we have that the right hand side tends to $0$. 
So, the left-hand side also tends to $0$. If it were the case that
\[
\bigl|\Ric_{g_\infty}(t_0)(x_0)\bigr| 
+ \bigl|T_\infty(x_0,t_0)\bigr|^2 > 0,
\]
at some point $(x_0,t_0)$, then by smooth pointed convergence the same quantity would be bounded below by a positive constant on a small spacetime neighbourhood for all sufficiently large $i$, contradicting the vanishing of the above integrals. Hence,
\[
\Ric_{g_\infty}(t) \equiv 0,
\qquad
T_\infty(t) \equiv 0,
\]
for all $t \leq 0$.

Then, arguing exactly as in the proof of Theorem \ref{thmExtensionBoundedTorsion}, we obtain that $g_\infty(0)$ has precisely Euclidean volume growth which, together with the fact that $M_\infty$ is complete and Ricci-flat allows the application of the Bishop-Gromov relative volume comparison theorem, which implies that
\begin{equation}\label{eqContradiction2}
   \Riem(g_\infty(0)) = 0 \text{ on }M_\infty. 
\end{equation}
Thus, $\Lambda_{\infty}(x_{\infty},0) = 0$, which contradicts \eqref{eqLimitLambda1}. So, our assumption that
\begin{equation}\label{eqContradictionAssum2}
        \int_0^{T_0}\sup_{x \in M}(|\Ric(x,t)|_{g(t)}+|T(x,t)|_{g(t)}^2)\d t < \infty.
    \end{equation}
led to a contradiction, completing the proof of this theorem.
\end{proof}

\subsection{A \texorpdfstring{$\kappa$}{k}-non-collapsing theorem}\label{Subsectionkappa}
In \cite[Section 4]{Chen}, Chen proves a generalisation of Perelman's $\kappa$-non-collapsing theorem for a whole class of flows of metrics satisfying a certain condition \eqref{kappacondition}, which he then uses to study finite-time singularities of reasonable flows of $G_2$-structures. Our reasonable flows of \Sp-structures satisfy the same schematic equations as Chen's $G_2$ flows, so we use the same techniques to study finite-time singularities in Subsection \ref{SubsectionFiniteTime}. In this subsection, we state Chen's $\kappa$-non-collapsing theorem. 
We begin by recalling Chen's definition of $\kappa$-non-collapsing \cite[Definition 4.1]{Chen}.
\begin{defn}
    Let $(M^n,g)$ be a Riemannian manifold. The Riemannian metric $g$ is said to be \emph{ $\kappa$-non-collapsing relative to an upper bound of scalar curvature on the scale $\rho$} if, for any $B_g(p,r)\subset M$ with $r<\rho$ such that $\sup_{B_g(p,r)}R_g \leq r^{-2}$, we have that $\Vol_gB_g(p,r)\geq \kappa r^n$.
\end{defn}

With this, Chen proves a modification of Perelman's $\kappa$-non-collapsing theorem, which we state now \cite[Theorem 4.2]{Chen}.

\begin{thm}\label{kappaNonCollapsed}
    Let $\ptt g = -2\Ric + E$ be a geometric flow on a compact manifold $M^n$, for some symmetric $2$-tensor $E$. Fix constants $\rho_0,T_0 \in (0,\infty)$.
    Suppose that for a scale $\rho$ and a time $t_0$ satisfying
\[
0 < \rho \leq \rho_0 \quad \text{and} \quad 0 < \frac{T_0}{2} \leq t_0 \leq T_0 < \infty,
\]
the integral 
\begin{equation}\label{kappacondition}
    I = I(t_0,\rho) = \int_0^{t_0} (t_0 + \rho^2 - t) \sup_M |E|^2 \, \mathrm{d}t
\end{equation}
is finite.  Then the Riemannian metric $g(t_0)$ is $\kappa$-non-collapsing relative to an upper bound of the scalar curvature on the scale $\rho$, where $\kappa = \kappa(g(0), T_0, \rho_0, I) > 0$.

\end{thm}
We will use this in the following subsection to study finite time singularities of reasonable flows of \Sp-structures, under suitable assumptions that allow us to apply Theorem \ref{kappaNonCollapsed}, as Chen does in the $G_2$ case \cite[Section 5]{Chen}.
\subsection{Blow-ups at finite time singularities}\label{SubsectionFiniteTime}

We now use the $\kappa$-non-collapsing theorem to prove more refined results about finite-time singularities, without requiring an assumption on uniform continuity of the metrics. The results of this subsection follow \cite[Section 5]{Chen}
In Section \ref{SectionExtensionAssumingBoundedTorsion} we showed that certain quantities must blow up at finite-time singularities. The following theorems say something more about the \emph{structure} of these finite-time singularities, under certain weighted integral torsion bounds, which are weaker than the assumption of Theorem \ref{thmExtensionBoundedTorsion}. In particular, Theorem \ref{ThmLimSups} demonstrates that, under a weighted integral bound on the torsion, the Ricci curvature and torsion must blow up at least at a Type-I rate. Theorem \ref{ThmMaximalVolumeGrowthSpin7Manifold} shows that, under an additional assumption on the growth of the scalar curvature and torsion, the blow-up limit at a finite-time singularity is a torsion-free \Sp-manifold with maximal volume growth.

We begin with Theorem \ref{ThmLimSups}.
\begin{thm}\label{ThmLimSups}
    Let $\Phi(t)$ be a solution to a reasonable flow of \Sp-structures on a compact manifold $M^8$, on a maximal time interval $[0,T_0)$, for some $T_0<\infty.$
    Assume that
    \begin{equation}
        \int_0^{T_0}(T_0-t)\sup_{x\in M}|T(x,t)|_{g(t)}^4\d t < \infty.
    \end{equation}
    Then,
    \begin{equation}\label{eqFirstEstimate}
        \limsup_{t\to T_0}(T_0-t)\sup_{x\in M}\left (|\Ric(x,t)|_{g(t)}+|T(x,t)|_{g(t)}^2\right )>0.
    \end{equation}
    Moreover, defining
    \begin{equation}
        P(t_0) = \sup_{t\leq t_0}\sup_{x \in M}\left(1+|R(x,t)|_{g(t)} +|T(x,t)|_{g(t)}^2\right)
    \end{equation}
    and
    \begin{equation}\label{Qdef}
        Q(t_0) = \sup_{t\leq t_0}\sup_{x \in M}\left(|\Riem(x,t)|_{g(t)}+|T(x,t)|_{g(t)}^2+|\nabla T(x,t)|_{g(t)}\right),
    \end{equation}
    we have that
    \begin{equation}\label{eqSecondEstimate}
        \limsup_{t_0 \to T_0}\left[(T_0-t_0)^2Q(t_0)P(t_0) \right] >0.
    \end{equation}
\end{thm}
\begin{proof}We start by proving the first estimate \eqref{eqFirstEstimate}.
    Firstly, the condition that 
    \begin{equation}
        \int_0^{T_0}(T_0-t)\sup_{x\in M}|T(x,t)|_{g(t)}^4\d t < \infty
    \end{equation}
    means that the flow is $\kappa$-non-collapsed on the scale $\sqrt{T_0-t_0}$, by Theorem \ref{kappaNonCollapsed} and the fact that the metric evolution is given by \eqref{eqReasonableMetric}. 

    By Theorem \ref{ThmLongTime}, we have that $Q(t)\to \infty$ as $t \to T_0$. So, there exists an increasing sequence $t_k \to T_0$ such that $Q(t_k)=2^k$, and we choose points $p_k$ so that the space-time supremum is achieved at the point $(p_k,t_k)$. Again by Theorem \ref{ThmLongTime}, there exists a subsequence of $t_k$ (still denoted by $t_k$), such that
    \[
    T_0 - t_k \geq \frac{T_0-t_{k-1}}{3},
    \]
    so
    \begin{equation}\label{tktk-1bound}
        t_k - t_{k-1} \leq 2(T_0 - t_k).
    \end{equation}
    Now, assume for the sake of contradiction that the first estimate does not hold, i.e.,
       \begin{equation}
        \limsup_{t\to T_0}(T_0-t)\sup_{x\in M}\left (|\Ric(x,t)|_{g(t)}+|T(x,t)|_{g(t)}^2\right )=0.
    \end{equation}
    Then, since $\frac{\pt }{\pt t}g = -2 \Ric  + L(T) + T*T +C$, we have the estimate
    \begin{equation}
        \left|\ptt g \right| \leq C(|\Ric| +|T|^2 +1),
    \end{equation}
    for some constant $C$, and so
    \begin{equation}
        \int_{t_{k-1}}^{t_k} \left|\ptt g(t) \right| \d t \leq C(T_0-t_k)\sup_{t\leq t_k}\sup_{x\in M}\left(1+|R(x,t)|_{g(t)} +|T(x,t)|_{g(t)}^2\right) \to 0 \text{ as }k\to \infty,
    \end{equation}
    where we used \eqref{tktk-1bound} in the first inequality.
    Thus, the metrics $g(t_k)$ converge to the same limit as $k\to \infty$, and hence $g(t_k)$ are all uniformly equivalent for $k$ sufficiently large.

    We now define two parabolically rescaled flows, around times $t_k$ and $t_{k-1}$:
    \begin{align}
        \Phi^A_k(t) = 2^{2k}\Phi(t_k+2^{-k}t)&,\quad g^A_k(t) = 2^kg(t_k+2^{-k}t),\\
        \Phi^B_k(t) = 2^{2k}\Phi(t_{k-1}+2^{-k}t)&,\quad g^B_k(t) = 2^kg(t_{k-1}+2^{-k}t).   
    \end{align}
    Arguing as before (e.g.,\ in the proof of Theorem \ref{thmExtensionBoundedTorsion}), the two sequences $(M,p_k,\Phi^A_k,g^A_k)$ and $(M,p_k,\Phi^B_k,g^B_k)$ converge to torsion-free ancient solutions $(M,p_\infty,\Phi^A_\infty,g^A_\infty)$ and $(M,p_\infty,\Phi^B_\infty,g^B_\infty)$.

    Since the un-rescaled metrics $g(t_k)$ and $g(t_{k-1})$ converge to the same limit, the two limit metrics $g_\infty^A$ and $g_\infty^B$ must be isometric. Moreover, since $\Phi_\infty^{A}(t)$ and $\Phi_\infty^B(t)$ are torsion-free for all $t$, the quantities $Q^A_\infty$ and $Q^B_\infty$ must agree for all time. However, for the first sequence, we have that:
    \[
    Q_\infty^A(0) = \lim_{k \to \infty} \frac{Q(t_k)}{2^k} = \frac{2^k}{2^k} = 1.
    \]
    For the second sequence, since the maximum of the right-hand side of \eqref{Qdef} up to $t_{k-1}$ was $Q(t_{k-1}) = 2^{k-1}$ (by construction), we have that:
    \[
    Q_\infty^B(0) = \lim_{k \to \infty} \frac{Q(t_{k-1})}{2^k} = \frac{2^{k-1}}{2^k} = \frac{1}{2}.
    \]
    The isometry of the limits implies $Q_\infty^A(0) = Q_\infty^B(0)$, leading to the contradiction $1 = 1/2$, which concludes the proof of \eqref{eqFirstEstimate}.

    We now proceed to the second estimate \eqref{eqSecondEstimate}. Note that, in light of \eqref{eqFirstEstimate}, it suffices to show that, for $t_0$ sufficiently close to $T$,
    \begin{equation}
        \sup_{t\leq t_0}\sup_{x\in M}\left (|\Ric(x,t)|_{g(t)}+|T(x,t)|_{g(t)}^2\right ) \leq C \sqrt{P(t_0)Q(t_0)},
    \end{equation}
    for some constant $C$. This, together with \eqref{eqFirstEstimate} will give the result. Now, the torsion term $|T|^2$ is immediately controlled by the definitions of $P$ and $Q$, so we need only prove that
    \begin{equation}
        \sup_{t\leq t_0}\sup_{x\in M}\left (|\Ric(x,t)|_{g(t)}\right ) \leq C \sqrt{P(t_0)Q(t_0)}.
    \end{equation}

    We make a brief aside here to motivate the strategy of what follows.
    We are aiming to obtain a pointwise bound on $|\Ric|$ by $C \sqrt{P(t_0)Q(t_0)}$. We first rescale the flow so that $Q$ is $1$. We will then see that Theorem \ref{kappaNonCollapsed} gives $\kappa$-non-collapsing on the constant scale $1$, and so we have control of metric coefficients and their covariant derivatives inside balls of finite radius. With this, we aim for an $L^2$-bound on the rescaled Ricci tensor $|\widetilde\Ric|$, which we obtain by integrating the evolution equation for the scalar curvature. We then upgrade this integral bound to a pointwise one, by considering the evolution equation for the Ricci tensor and applying the parabolic mean value inequality. Undoing the rescaling will provide the required bound on $\Ric$. At several points in these estimates, it will be necessary to bound the rescaled quantities $|\tilde{T}|$ and $|\widetilde{\nabla}\tilde{T}|$, with respect to the rescaled flow, so we start now by rescaling and bounding these tensors.

    Fix $\tau\leq t_0$, sufficiently close to $T_0$.
    By Theorem \ref{ThmLongTime}, we have that $Q(\tau)\geq \frac{C}{T_0-\tau}$, since $T_0$ is the maximal existence time. Thus,
    \begin{equation}
        Q(\tau)^{-1/2} \leq C(T_0-\tau)^{1/2},
    \end{equation}
    and since the flow is $\kappa$-non-collapsed on the scale $\sqrt{T_0-\tau}$, it is in particular $\kappa$-non-collapsed on the scale $Q(\tau)^{-1/2}$. We now rescale the flow as before:
    \begin{equation}
        \tilde{g}(s) = Q(\tau)g\left (\tau+\frac{s}{Q(\tau)} \right), \quad \tilde\Phi(s) = Q(\tau)\Phi\left (\tau+\frac{s}{Q(\tau)} \right),
    \end{equation}
    for $s\in [-1,0]$. We decorate any tensor induced by $\tilde{g}$ or $\tilde{\Phi}$ with its own tilde.
    Define the parabolic cylinder
    \begin{equation}
        P_r(p,0) = B_{\tilde{g}(0)}(p,r)\times [-r^2,0],
    \end{equation}
    which corresponds to 
    \begin{equation}
        B_{g(\tau)}(p,rQ(\tau)^{-1/2})\times \left [\tau-\frac{r^2}{Q(\tau)},\tau\right ]
    \end{equation}
    in the unscaled metric and time.
    
    Under this rescaling, on $P_1(p,0)$, we have
    \begin{equation}
        |\widetilde{\Riem}|_{\tilde{g}} + |\widetilde{T}|_{\tilde{g}} + |\widetilde{\nabla T}|_{\tilde{g}} \leq 1,
    \end{equation}
    by definition of $Q$.
    The Shi-type estimates (Theorem \ref{ThmShiType}) then give uniform bounds for all higher derivatives of curvature and torsion on smaller cylinders (of a fixed, definite size). Together with $\kappa$-non-collapsing, this gives a uniform lower bound for the harmonic radius. Hence, on a fixed-size ball in the rescaled metric, the metric coefficients and all their derivatives are uniformly controlled.

In these coordinates, the evolution equation for torsion can be written as
\begin{equation}
    \left(\ptt -D\right)\widetilde{T} = 0,
\end{equation}
for some elliptic operator $D$ with uniformly bounded coefficients, and uniformly bounded derivatives of coefficients, of all orders, using \eqref{eqReasonableTorsion} and the Shi-type estimates.
Standard parabolic estimates then give
    \begin{equation}
        \sup_{P_{1/4}(p,0)}|\widetilde{\nabla}^k\widetilde{T}| \leq C(k)\sup_{P_{1/2}(p,0)}|\widetilde{T}|,
    \end{equation}
    for all $k\geq0,$ and we know that
    \begin{equation}
        |\widetilde{T}|_{\tilde{g}} = Q(\tau)^{-1}|T|_g \leq \frac{P(\tau)}{Q(\tau)},
    \end{equation}
    after rescaling.

We now consider the evolution equation of the scalar curvature $\widetilde{R}$ along a reasonable flow of \Sp-structures. This contains terms of the form $|\widetilde{\Ric}|$, and integrating this equation will provide the $L^2$-bound on $\widetilde\Ric$. 
We have that, for a metric evolving according to the evolution equation $\ptt g = 2h$, the induced evolution equation for the scalar curvature $R$ is \cite[Lemma 2.7]{ChowLuNi}
\begin{equation}\label{eqScalEvolution}
    \ptt R = -2 \Delta (\operatorname{tr} h) + 2 \nabla_i\nabla_jh_{ij} - 2h_{ij}\Ric_{ij}.
\end{equation}
Using \eqref{eqReasonableMetric}, we write $\ptt \tilde{g} = -\widetilde\Ric_{ij}+\widetilde{E}_{ij}$, where $E$ represents the lower order torsion terms.
Plugging this into \eqref{eqScalEvolution}, we have
\begin{equation}\label{eqPluggedin}
    \ptt \widetilde{R}
=
2\widetilde\Delta \widetilde R
-2\widetilde\Delta\bigl(\operatorname{tr}_{\widetilde g}\widetilde E\bigr)
-2\widetilde\nabla_i\widetilde\nabla_j \widetilde{\Ric}_{ij}
+2\widetilde\nabla_i\widetilde\nabla_j \widetilde E_{ij}
+2\lvert \widetilde{\Ric}\rvert_{\widetilde g}^{2}
-2\widetilde E_{ij}\widetilde{\Ric}_{ij}.
\end{equation}
Now, using the contracted Bianchi identity, we have that 
\[
2\widetilde\nabla_i\widetilde\nabla_j \widetilde{\Ric}_{ij} = \widetilde{\Delta}\widetilde{R}
\]
and so \eqref{eqPluggedin} simplifies to 
\begin{equation}\label{eqSimplifiedByBianchi}
        \ptt \widetilde{R}
=
\widetilde\Delta \widetilde R
-2\widetilde\Delta\bigl(\operatorname{tr}_{\widetilde g}\widetilde E\bigr)
+2\widetilde\nabla_i\widetilde\nabla_j \widetilde E_{ij}
+2\lvert \widetilde{\Ric}\rvert_{\widetilde g}^{2}
-2\widetilde E_{ij}\widetilde{\Ric}_{ij}.
\end{equation}
The error term $\widetilde{E}$ satisfies
\begin{equation}\label{eqRescaledE}
    \widetilde{E} = \frac{L(\widetilde{T}) + \widetilde{T}*\widetilde{T}}{Q(\tau)}
\end{equation}
and the torsion derivative terms are all controlled by $C\frac{P(\tau)}{Q(\tau)}$ by the Shi-type estimates. Finally, the term $\widetilde{E}_{ij}\widetilde{\Ric}_{ij}$ is controlled by
\begin{equation}
    \frac{|L(\widetilde{T}) + \widetilde{T}*\widetilde{T}|}{Q(\tau)}|\widetilde{\Ric}| \leq C|\widetilde{\Ric}|^2 + C\frac{P(\tau)}{Q(\tau)},
\end{equation}
by an application of Young's inequality. Combining these estimates and rearranging, we obtain
\begin{equation}\label{eqEffectiveInequalityRtilde}
    \left|\left(\frac{\partial}{\partial s} - \widetilde{\Delta} \right) \widetilde{R} - 2|\widetilde{\Ric}|_{\tilde{g}}^2 \right | \leq |\widetilde{\Ric}|^2_{\tilde{g}}+ C\frac{P(\tau)}{Q(\tau)}.
\end{equation}
We now aim to use this to obtain an $L^2$-bound for $\widetilde{\Ric}$.
We define a cut-off function $\chi$, depending on space and time, satisfying the following conditions:
\begin{itemize}
    \item $\chi \equiv 1$ on $P_{1/2}(p,0)$,
    \item$\chi \equiv 0$ outside of $P_{1}(p,0)$,
    \item $|\nabla \chi| +|\partial_s\chi| + |\Delta \chi| \leq C$.
\end{itemize}
From the fundamental theorem of calculus and integration by parts (integrating with respect to a fixed background measure and using that $\chi_{s=-1} = 0$), we have that
\begin{align}\label{eqIntegrationbyparts}
    \left(\int_M\chi \widetilde{R}\right)\Bigg|_{s=0}  &= \int_{-1}^{0}\int_M \frac{\partial}{\partial s} (\chi \widetilde{R})\d s\\
                            &= \int_{-1}^{0}\int_M\widetilde{R}\left(\frac{\partial}{\partial s} + \widetilde\Delta \right)\chi + \chi\left(\frac{\partial}{\partial s} - \widetilde\Delta \right) \widetilde{R} \d s. \label{eqSecondLine}
\end{align}
Now, from \eqref{eqEffectiveInequalityRtilde}, we have that 
\[
\chi |\widetilde\Ric|^2 \leq \chi \left(\frac{\partial}{\partial s} - \widetilde\Delta \right) \widetilde{R} + C\frac{P(\tau)}{Q(\tau)}\chi.
\]
Integrating this inequality and combining with \eqref{eqSecondLine} gives
\begin{equation}
    \int_{-1}^{0}\int_M\chi|\widetilde{\Ric}|^2\d s \leq \left(\int_M\chi \widetilde{R}\right)\Bigg|_{s=0} - \int_{-1}^{0}\int_M\widetilde{R}\left(\frac{\partial}{\partial s} + \widetilde\Delta \right)\chi + C\frac{P(\tau)}{Q(\tau)}.
\end{equation}
Now, the first term on the right hand side is bounded by $C\frac{P(\tau)}{Q(\tau)}$ since $\chi$ is compactly supported and $\widetilde{R}$ is bounded by $C\frac{P(\tau)}{Q(\tau)}$. The second is bounded by the bound on $\widetilde{R}$ together with the third bullet point in the definition of $\chi$ above. So, the right hand side is uniformly bounded by $\frac{P(\tau)}{Q(\tau)}$. Finally, we have that
\begin{equation}\label{eqL2bound}
    \int_{P_{1/2}(p,0)}|\widetilde{\Ric}|_{\tilde{g}}^2\leq  \int_{-1}^{0}\int_M\chi|\widetilde{\Ric}|^2\d t\leq C \frac{P(\tau)}{Q(\tau)},
\end{equation}
yielding the required $L^2$-bound on $\widetilde{\Ric}$.

We now obtain an effective evolution equation for the Ricci tensor. Recall that, for a Riemannian metric $g$ evolving by $\ptt g = 2h$, the induced evolution equation for the Ricci tensor is \cite[Equation 2.31]{ChowLuNi}
\begin{equation}\label{eqRicciEvolution}
\frac{\partial}{\partial t}\Ric_{ij}
=
\nabla_k\nabla_i h_{jk}
+
\nabla_k\nabla_j h_{ik}
-
\Delta h_{ij}
-
\nabla_i\nabla_j(\operatorname{tr}_g h).
\end{equation}
Substituting the equation for the evolution of the metric under a reasonable flow \eqref{eqReasonableMetric} and arguing as we did for the scalar curvature to bound the lower order terms, we see that

\begin{equation}
    \left (\frac{\partial}{\partial s}  - \widetilde{D}\right)|\widetilde{ \Ric}|^2_{\tilde{g}} \leq C \frac{P(\tau)}{Q(\tau)},
\end{equation}
for some elliptic operator $\widetilde{D}$ with bounded coefficients and coefficients whose derivatives are bounded.
An application of the parabolic mean value inequality then gives
\begin{equation}
    |\widetilde{\Ric}|^2_{\tilde{g}} \leq C  \int_{P_{1/2}(p,0)}|\widetilde{\Ric}|_{\tilde{g}}^2 +  C \frac{P(\tau)}{Q(\tau)} \leq C \frac{P(\tau)}{Q(\tau)},
\end{equation}
where the second inequality follows from the $L^2$-bound \eqref{eqL2bound}.
Undoing the rescaling, writing this in terms of the original metric and taking a supremum, we get precisely
\begin{equation}
    \sup_{t\leq t_0}\sup_{x\in M}\left (|\Ric(x,t)|_{g(t)}\right ) \leq C \sqrt{P(t_0)Q(t_0)},
\end{equation}
as required.
\end{proof}
If, additionally, we have that 
\begin{equation}
    \sup_M (|R| + |T|^2) = o\left( \frac{1}{T-t}\right ),
\end{equation}
meaning that,
\begin{equation}
    \lim_{t\to T_0}(T_0-t)\sup_M (|R| + |T|^2) = 0,
\end{equation}
then we can show that any blow-up limit at a finite time is a maximum volume growth \Sp-manifold. More precisely, we have the following theorem. 
\begin{thm}\label{ThmMaximalVolumeGrowthSpin7Manifold}
Let $\Phi(t)$ be a solution to a reasonable flow of \Sp-structures on a compact manifold $M^8$, on a maximal time interval $[0,T_0)$, for some $T_0<\infty.$
    Assume that
    \begin{equation}\label{assumptionweightedtorsion}
        \int_0^{T_0}(T_0-t)\sup_{x\in M}|T(x,t)|_{g(t)}^4\d t < \infty,
    \end{equation}
    and
    \begin{equation}\label{eqSubType1}
        \sup_{x\in M}\left(|R(x,t)|_{g(t)} + |T(x,t)|_{g(t)}^2 \right) = o\left(\frac{1}{T_0-t} \right) \text{ as } t\to T_0.
    \end{equation}

    Then, there exists a sequence $t_k \to T_0$ and points $p_k\in M$ such that:
    \begin{enumerate}
        \item\label{item1} A blow up occurs:
        \begin{equation}
            Q_k \coloneqq \left(|\Riem(p_k,t_k)|_{g(t_k)}+|T(p_k,t_k)|_{g(t_k)}^2+|\nabla T(p_k,t_k)|_{g(t_k)}\right) \to \infty;
        \end{equation}
        \item\label{item2} The rescaled limit exists:
        \begin{equation}
            (M,\Phi_k,g_k,p_k) = (M, Q_k^2\Phi(t_k), Q_kg(t_k),p_k)
        \end{equation}
        subconverges to a complete limit
        \begin{equation}
            (M,\Phi_\infty,g_\infty,p_\infty);
        \end{equation}

        \item\label{item3} The limit is torsion-free:
        \begin{equation}
            T_\infty = 0;
        \end{equation}
        \item\label{item4} The limit has maximal volume growth:
        \begin{equation}
            \Vol_{g_{\infty}}(B_{g_\infty}(p_\infty,r)) \geq \kappa r^8,
        \end{equation}
        for some $\kappa>0$ and all $r>0$.
    \end{enumerate}
\end{thm}
\begin{proof}
    By Theorem \ref{ThmLimSups}, we have that
    \begin{equation}
        \limsup_{t_0 \to T_0}\left[(T_0-t_0)^2Q(t_0)P(t_0) \right] >0.
    \end{equation}
    Assumption \eqref{eqSubType1} gives that
    \begin{equation}
        (T_0-t)P(t) \to 0 \text{ as } t\to T_0,
    \end{equation}
    and so, for the product $(T_0-t)^2P(t)Q(t)$ to remain positive as $t\to T_0$, we must have that
    \begin{equation}
        (T_0-t)Q(t)\to \infty,
    \end{equation}
    as $t\to T_0$. So, there exists a sequence $t_k \to T_0$ and points $p_k\in M$ satisfying item \ref{item1} of the theorem.

    For item \ref{item2}, we start by choosing a sequence of times $t_k\to T_0$ such that
    \begin{equation}
        (T_0-t_k)Q(t_k) \to \infty \text{ as } t_k \to T_0.
    \end{equation}
    We define fixed, rescaled metrics and \Sp-structures at these times, in the following way:
    \begin{equation}
        g_k = Q_kg(t_k), \quad \Phi_k = Q_k^2\Phi(t_k).
    \end{equation}
    Under this rescaling, we have that
    \begin{equation}
        |\Riem_k|_k + |T_k|^2_k + |\nabla_kT_k|_k \leq C,
    \end{equation}
    for some $C < \infty$, where the subscripts $k$ refer to tensors and norms induced by the \Sp-structure $\Phi_k$, at the point $(p_k,t_k)$.
    Because of this, Theorem \ref{ThmShiType} applies, and we have
    \begin{equation}
        |\nabla_k^m\Riem_k|_k+ |\nabla^{m+1}_kT_k|_k \leq C_m,
    \end{equation}
    for all $m\geq 0$ and for some $C_m<\infty.$
    Moreover, by assumption \eqref{assumptionweightedtorsion}, Theorem \ref{kappaNonCollapsed} gives that the flow is $\kappa$-non-collapsed on the scale $\sqrt{T_0-t}$. After rescaling as above, distances scale by $\sqrt{Q_k}$, and so the rescaled non-collapsing scale becomes 
    \begin{equation}\label{eqnoncollapsingonascalegoingtoinfinity}
        \sqrt{Q_k(T_0-t_k)} \to \infty.
    \end{equation}
    Thus, for any fixed $r>0$, we have that
    \begin{equation}
        \Vol_{g_k}(B_{g_k}(p_k,r))\geq \kappa r^8,
    \end{equation}
    for some $\kappa>0$. Since curvature is uniformly bounded and volumes of unit balls are bounded below, we have a uniform injectivity radius lower bound:
    \begin{equation}
        \inj(M,g_k,p_k)\geq\iota>0,
    \end{equation}
    for some $\iota>0$. Thus, all of the conditions of Theorem \ref{ThmSpin7Compact} are satisfied, and so there exists a subsequence of $\{t_k\}$ such that
    \begin{equation}
            (M,\Phi_k,g_k,p_k) = (M, Q_k^2\Phi(t_k), Q_kg(t_k),p_k)
        \end{equation}
         converges to a complete limit
        \begin{equation}
            (M,\Phi_\infty,g_\infty,p_\infty),
        \end{equation}
    in the sense defined in Section \ref{SectionCompactness}, hence proving item \ref{item2}.

    For item \ref{item3}, we need to show that the limiting \Sp-structure $\Phi_\infty$ is torsion-free. This follows directly from how we rescaled $\Phi$. Indeed, under the rescaling, we have that
    \begin{equation}
        |T_k|^2_k = Q_k^{-1}|T|^2,
    \end{equation}
    which clearly tends to $0$ as $t_k\to T_0,$ proving item \ref{item3}.

    Finally, since non-collapsing holds on arbitrarily large scales in the rescaled metrics \eqref{eqnoncollapsingonascalegoingtoinfinity}, the limit satisfies
    \begin{equation}
            \Vol_{g_{\infty}}(B_{g_\infty}(p_\infty,r)) \geq \kappa r^8,
        \end{equation}
        for some $\kappa>0$ and all $r>0$,
    ending the proof of item \ref{item4}, and thus the theorem.   
\end{proof}

\section{Outlook}
The results obtained in this paper motivate several potential directions for future work. We briefly discuss some of them here.

Firstly, the general theory obtained for reasonable flows, and the fact that the Ricci-harmonic flow is reasonable (Example \ref{propRHFreasonable}), suggests that the Ricci-harmonic flow, as well as any other flows that turn out to be reasonable, may be useful tools in the study of \Sp-structures more broadly. In particular, it would be interesting to study the question of dynamical stability in order to move towards the use of this flow to tackle the question of existence of torsion-free \Sp-structures. Along similar lines, obtaining new examples of explicit solutions in more complicated settings than Example \ref{exampleRHF} would be valuable, to further our understanding of how these flows behave.

Additionally, we have seen in Theorem \ref{ThmMaximalVolumeGrowthSpin7Manifold} that blow-up limits at finite-time singularities are modelled on maximal volume growth torsion-free \Sp-structures. It would be interesting to see if any of the known examples (e.g., \cite{BryantSalamon}) can arise as finite-time singularities of a carefully constructed reasonable flow.

It would be very desirable to find conditions on the initial \Sp-structure under which a particular reasonable flow exists for all time and converges to a torsion-free \Sp-structure. Motivated by Joyce’s existence theorem for torsion-free \Sp-structures \cite{Joyce1996}, one may expect that a suitable smallness condition on the initial torsion might suffice. Such a result would provide a parabolic approach to Joyce’s existence theory, and would clarify the extent to which the flow can be used as an analytic tool in the study of \Sp-geometry.

\printbibliography

@article{Dwivedi24,
    author = {Dwivedi, Shubham},
    title = {A gradient flow of Spin(7)-structures},
    journal = {The Quarterly Journal of Mathematics},
    pages = {haaf018},
    year = {2025},
    month = {07},
    issn = {0033-5606},
    doi = {10.1093/qmath/haaf018},
    url = {https://doi.org/10.1093/qmath/haaf018},
    
}

@misc{DwivediRHF,
      title={Ricci-harmonic flow of $\mathrm{G}_2$ and Spin(7)-structures}, 
      author={Shubham Dwivedi},
      year={2026},
      eprint={2601.05210},
      archivePrefix={arXiv},
      primaryClass={math.DG},
      url={https://arxiv.org/abs/2601.05210}, 
}

@article {KarigiannisDeformations,
    AUTHOR = {Karigiannis, Spiro},
     TITLE = {Deformations of {$G_2$} and {${\rm Spin}(7)$} structures},
   JOURNAL = {Canad. J. Math.},
  FJOURNAL = {Canadian Journal of Mathematics. Journal Canadien de
              Math\'ematiques},
    VOLUME = {57},
      YEAR = {2005},
    NUMBER = {5},
     PAGES = {1012--1055},
      ISSN = {0008-414X,1496-4279},
   MRCLASS = {53C29 (53C10)},
  MRNUMBER = {2164593},
MRREVIEWER = {Andrew\ Swann},
       DOI = {10.4153/CJM-2005-039-x},
       URL = {https://doi.org/10.4153/CJM-2005-039-x},
}

@inproceedings{KarigiannisFlows,
   title={Flows of Spin(7)-structures},
   url={http://dx.doi.org/10.1142/9789812790613_0023},
   DOI={10.1142/9789812790613_0023},
   booktitle={Differential Geometry and Its Applications},
   publisher={World Scientific},
   author={Karigiannis, Spiro},
   year={2008},
   month=jul }

@book{JoyceBook,
    author = {Joyce, Dominic D},
    title = {Riemannian Holonomy Groups and Calibrated Geometry},
    publisher = {Oxford University Press},
    year = {2007},
    month = {02},
    isbn = {9780199215607},
    doi = {10.1093/oso/9780199215607.001.0001},
    url = {https://doi.org/10.1093/oso/9780199215607.001.0001},
}

@article {Bonan,
    AUTHOR = {Bonan, Edmond},
     TITLE = {Sur des vari\'et\'es riemanniennes \`a{} groupe d'holonomie
              {$G\sb{2}$}\ ou Spin{$(7)$}},
   JOURNAL = {C. R. Acad. Sci. Paris S\'er. A-B},
  FJOURNAL = {Comptes Rendus Hebdomadaires des S\'eances de l'Acad\'emie des
              Sciences. S\'eries A et B},
    VOLUME = {262},
      YEAR = {1966},
     PAGES = {A127--A129},
      ISSN = {0151-0509},
   MRCLASS = {53.70 (53.55)},
  MRNUMBER = {196668},
MRREVIEWER = {A.\ G.\ Walker},
}

@book {ChowRicci,
    AUTHOR = {Chow, Bennett and Knopf, Dan},
     TITLE = {The {R}icci flow: an introduction},
    SERIES = {Mathematical Surveys and Monographs},
    VOLUME = {110},
 PUBLISHER = {American Mathematical Society, Providence, RI},
      YEAR = {2004},
     PAGES = {xii+325},
      ISBN = {0-8218-3515-7},
   MRCLASS = {53C44 (35K60 53C21)},
  MRNUMBER = {2061425},
MRREVIEWER = {John\ Urbas},
       DOI = {10.1090/surv/110},
       URL = {https://doi.org/10.1090/surv/110},
}

@article {DLE24,
    AUTHOR = {Dwivedi, Shubham and Loubeau, Eric and S\'a{} Earp, Henrique},
     TITLE = {Harmonic flow of {${\rm Spin}(7)$}-structures},
   JOURNAL = {Ann. Sc. Norm. Super. Pisa Cl. Sci. (5)},
  FJOURNAL = {Annali della Scuola Normale Superiore di Pisa. Classe di
              Scienze. Serie V},
    VOLUME = {25},
      YEAR = {2024},
    NUMBER = {1},
     PAGES = {151--215},
      ISSN = {0391-173X,2036-2145},
   MRCLASS = {53C27 (53C29 53C43 53E99 58E20 58J35 58J60)},
  MRNUMBER = {4732637},
MRREVIEWER = {Roger\ Nakad},
       DOI = {10.2422/2036-2145.202111\_026},
       URL = {https://doi.org/10.2422/2036-2145.202111_026},
}

@article {Chen,
    AUTHOR = {Chen, Gao},
     TITLE = {Shi-type estimates and finite-time singularities of flows of
              {${\rm G}_2$} structures},
   JOURNAL = {Q. J. Math.},
  FJOURNAL = {The Quarterly Journal of Mathematics},
    VOLUME = {69},
      YEAR = {2018},
    NUMBER = {3},
     PAGES = {779--797},
      ISSN = {0033-5606,1464-3847},
   MRCLASS = {53C44 (53C29)},
  MRNUMBER = {3859207},
MRREVIEWER = {Yong\ Wei},
       DOI = {10.1093/qmath/hax060},
       URL = {https://doi.org/10.1093/qmath/hax060},
}

@article {Lotay-Wei,
    AUTHOR = {Lotay, Jason D. and Wei, Yong},
     TITLE = {Laplacian flow for closed {${\rm G}_2$} structures: {S}hi-type
              estimates, uniqueness and compactness},
   JOURNAL = {Geom. Funct. Anal.},
  FJOURNAL = {Geometric and Functional Analysis},
    VOLUME = {27},
      YEAR = {2017},
    NUMBER = {1},
     PAGES = {165--233},
      ISSN = {1016-443X,1420-8970},
   MRCLASS = {53C44 (53C10 53C25 53C29)},
  MRNUMBER = {3613456},
MRREVIEWER = {Marco\ Freibert},
       DOI = {10.1007/s00039-017-0395-x},
       URL = {https://doi.org/10.1007/s00039-017-0395-x},
}

@misc{perelman2002entropyformularicciflow,
      title={The entropy formula for the Ricci flow and its geometric applications}, 
      author={Grisha Perelman},
      year={2002},
      eprint={math/0211159},
      archivePrefix={arXiv},
      primaryClass={math.DG},
      url={https://arxiv.org/abs/math/0211159}, 
}

@misc{perelman2003ricciflowsurgerythreemanifolds,
      title={Ricci flow with surgery on three-manifolds}, 
      author={Grisha Perelman},
      year={2003},
      eprint={math/0303109},
      archivePrefix={arXiv},
      primaryClass={math.DG},
      url={https://arxiv.org/abs/math/0303109}, 
}

@article{Berger,
     author = {Berger, Marcel},
     title = {Sur les groupes d'holonomie homog\`enes de vari\'et\'es \`a connexion affine et des vari\'et\'es riemanniennes},
     journal = {Bulletin de la Soci\'et\'e Math\'ematique de France},
     pages = {279--330},
     publisher = {Soci\'et\'e math\'ematique de France},
     volume = {83},
     year = {1955},
     doi = {10.24033/bsmf.1464},
     mrnumber = {18,149a},
     zbl = {0068.36002},
     language = {fr},
     url = {https://www.numdam.org/articles/10.24033/bsmf.1464/}
}

@article {BryantSalamon,
    AUTHOR = {Bryant, Robert L. and Salamon, Simon M.},
     TITLE = {On the construction of some complete metrics with exceptional
              holonomy},
   JOURNAL = {Duke Math. J.},
  FJOURNAL = {Duke Mathematical Journal},
    VOLUME = {58},
      YEAR = {1989},
    NUMBER = {3},
     PAGES = {829--850},
      ISSN = {0012-7094,1547-7398},
   MRCLASS = {53C25 (53C57)},
  MRNUMBER = {1016448},
MRREVIEWER = {Krzysztof\ Galicki},
       DOI = {10.1215/S0012-7094-89-05839-0},
       URL = {https://doi.org/10.1215/S0012-7094-89-05839-0},
}

@article {Joyce1996,
    AUTHOR = {Joyce, D. D.},
     TITLE = {Compact {$8$}-manifolds with holonomy {${\rm Spin}(7)$}},
   JOURNAL = {Invent. Math.},
  FJOURNAL = {Inventiones Mathematicae},
    VOLUME = {123},
      YEAR = {1996},
    NUMBER = {3},
     PAGES = {507--552},
      ISSN = {0020-9910,1432-1297},
   MRCLASS = {53C25 (53C20)},
  MRNUMBER = {1383960},
MRREVIEWER = {Claude\ LeBrun},
       DOI = {10.1007/s002220050039},
       URL = {https://doi.org/10.1007/s002220050039},
}

@book{Chow-Knopf,
	author = {Chow, Bennett and Knopf, Dan},
	fseries = {Mathematical Surveys and Monographs},
	isbn = {0-8218-3515-7},
	issn = {0076-5376},
	language = {English},
	publisher = {Providence, RI: American Mathematical Society (AMS)},
	series = {Math. Surv. Monogr.},
	title = {The {Ricci} flow: an introduction},
	volume = {110},
	year = {2004},
	zbl = {1086.53085},
	zbmath = {2121403}}

@book {CCGGIIKLLN2,
    AUTHOR = {Chow, Bennett and Chu, Sun-Chin and Glickenstein, David and
              Guenther, Christine and Isenberg, James and Ivey, Tom and
              Knopf, Dan and Lu, Peng and Luo, Feng and Ni, Lei},
     TITLE = {The {R}icci flow: techniques and applications. {P}art {II}},
    SERIES = {Mathematical Surveys and Monographs},
    VOLUME = {144},
      NOTE = {Analytic aspects},
 PUBLISHER = {American Mathematical Society, Providence, RI},
      YEAR = {2008},
     PAGES = {xxvi+458},
      ISBN = {978-0-8218-4429-8},
   MRCLASS = {53C44 (35K55 53C21)},
  MRNUMBER = {2365237},
MRREVIEWER = {James\ Alexander\ McCoy},
       DOI = {10.1090/surv/144},
       URL = {https://doi.org/10.1090/surv/144},
}

@book {CCGGIIKLLN1,
    AUTHOR = {Chow, Bennett and Chu, Sun-Chin and Glickenstein, David and
              Guenther, Christine and Isenberg, James and Ivey, Tom and
              Knopf, Dan and Lu, Peng and Luo, Feng and Ni, Lei},
     TITLE = {The {R}icci flow: techniques and applications. {P}art {I}},
    SERIES = {Mathematical Surveys and Monographs},
    VOLUME = {135},
      NOTE = {Geometric aspects},
 PUBLISHER = {American Mathematical Society, Providence, RI},
      YEAR = {2007},
     PAGES = {xxiv+536},
      ISBN = {978-0-8218-3946-1},
   MRCLASS = {53C44 (35K55 53C21 57M40)},
  MRNUMBER = {2302600},
MRREVIEWER = {James\ Alexander\ McCoy},
       DOI = {10.1090/surv/135},
       URL = {https://doi.org/10.1090/surv/135},
}

@article {Shi,
    AUTHOR = {Shi, Wan-Xiong},
     TITLE = {Deforming the metric on complete {R}iemannian manifolds},
   JOURNAL = {J. Differential Geom.},
  FJOURNAL = {Journal of Differential Geometry},
    VOLUME = {30},
      YEAR = {1989},
    NUMBER = {1},
     PAGES = {223--301},
      ISSN = {0022-040X,1945-743X},
   MRCLASS = {58G30 (53C20 58D25 58G11)},
  MRNUMBER = {1001277},
MRREVIEWER = {Friedbert\ Pr\"ufer},
       URL = {http://projecteuclid.org/euclid.jdg/1214443292},
}

@article {Hamilton3Manifolds,
    AUTHOR = {Hamilton, Richard S.},
     TITLE = {Three-manifolds with positive {R}icci curvature},
   JOURNAL = {J. Differential Geometry},
  FJOURNAL = {Journal of Differential Geometry},
    VOLUME = {17},
      YEAR = {1982},
    NUMBER = {2},
     PAGES = {255--306},
      ISSN = {0022-040X,1945-743X},
   MRCLASS = {53C25 (35K55 58G30)},
  MRNUMBER = {664497},
MRREVIEWER = {J.\ L.\ Kazdan},
       URL = {http://projecteuclid.org/euclid.jdg/1214436922},
}

@article {HamiltonCompactness,
    AUTHOR = {Hamilton, Richard S.},
     TITLE = {A compactness property for solutions of the {R}icci flow},
   JOURNAL = {Amer. J. Math.},
  FJOURNAL = {American Journal of Mathematics},
    VOLUME = {117},
      YEAR = {1995},
    NUMBER = {3},
     PAGES = {545--572},
      ISSN = {0002-9327,1080-6377},
   MRCLASS = {53C21 (58E11 58G30)},
  MRNUMBER = {1333936},
MRREVIEWER = {Ben\ Andrews},
       DOI = {10.2307/2375080},
       URL = {https://doi.org/10.2307/2375080},
}

@inproceedings {SalamonWalpuski,
    AUTHOR = {Salamon, Dietmar A. and Walpuski, Thomas},
     TITLE = {Notes on the octonions},
 BOOKTITLE = {Proceedings of the {G}\"okova {G}eometry-{T}opology
              {C}onference 2016},
     PAGES = {1--85},
 PUBLISHER = {G\"okova Geometry/Topology Conference (GGT), G\"okova},
      YEAR = {2017},
      ISBN = {978-1-57146-340-1},
   MRCLASS = {53C29 (15-01 17A35 53C38)},
  MRNUMBER = {3676083},
MRREVIEWER = {Frank\ Reidegeld},
}

@incollection {HamiltonSingularities,
    AUTHOR = {Hamilton, Richard S.},
     TITLE = {The formation of singularities in the {R}icci flow},
 BOOKTITLE = {Surveys in differential geometry, {V}ol.\ {II} ({C}ambridge,
              {MA}, 1993)},
     PAGES = {7--136},
 PUBLISHER = {Int. Press, Cambridge, MA},
      YEAR = {1995},
      ISBN = {1-57146-027-6},
   MRCLASS = {53C21 (58G30)},
  MRNUMBER = {1375255},
MRREVIEWER = {Man\ Chun\ Leung},
}

@book {ChowLuNi,
    AUTHOR = {Chow, Bennett and Lu, Peng and Ni, Lei},
     TITLE = {Hamilton's {R}icci flow},
    SERIES = {Graduate Studies in Mathematics},
    VOLUME = {77},
 PUBLISHER = {American Mathematical Society, Providence, RI; Science Press
              Beijing, New York},
      YEAR = {2006},
     PAGES = {xxxvi+608},
      ISBN = {978-0-8218-4231-7},
   MRCLASS = {53C44 (35K55 53C21 57M40 57M50)},
  MRNUMBER = {2274812},
MRREVIEWER = {James\ Alexander\ McCoy},
       DOI = {10.1090/gsm/077},
       URL = {https://doi.org/10.1090/gsm/077},
}

@misc{Duthie2025,
      title={Explicit solutions to the gradient flow of Spin(7)-structures}, 
      author={Joseph Duthie},
      year={2025},
      eprint={2511.17356},
      archivePrefix={arXiv},
      primaryClass={math.DG},
      url={https://arxiv.org/abs/2511.17356}, 
}

@article {JoyceSpin7,
    AUTHOR = {Joyce, D. D.},
     TITLE = {Compact {$8$}-manifolds with holonomy {${\rm Spin}(7)$}},
   JOURNAL = {Invent. Math.},
  FJOURNAL = {Inventiones Mathematicae},
    VOLUME = {123},
      YEAR = {1996},
    NUMBER = {3},
     PAGES = {507--552},
      ISSN = {0020-9910,1432-1297},
   MRCLASS = {53C25 (53C20)},
  MRNUMBER = {1383960},
MRREVIEWER = {Claude\ LeBrun},
       DOI = {10.1007/s002220050039},
       URL = {https://doi.org/10.1007/s002220050039},
}

\end{document}